\documentclass[a4paper,10pt]{amsart}

\usepackage{amsmath,amssymb}
\usepackage{amsthm}
\usepackage{amssymb}
\usepackage{mathrsfs}
\usepackage[shortlabels]{enumitem}
\usepackage{mathtools}
\usepackage[all,cmtip,2cell]{xy}
\usepackage{array}
\usepackage{braket}
\usepackage{graphicx}
\usepackage{stmaryrd}
\usepackage{comment}
\usepackage{diagbox}

\theoremstyle{plain}
\newtheorem{thm}{Theorem}[section]
\newtheorem{prop}[thm]{Proposition}
\newtheorem{lem}[thm]{Lemma}
\newtheorem{claim}[thm]{Claim}
\newtheorem*{cor}{Corollary}

\theoremstyle{definition}
\newtheorem{defi}[thm]{Definition}
\newtheorem{nota}[thm]{Notation}
\newtheorem{sett}[thm]{Setting}
\newtheorem*{NC}{Notation and Convention}
\newtheorem*{ACK}{Acknowledgement}

\theoremstyle{remark}
\newtheorem{rem}[thm]{Remark}

\numberwithin{equation}{section}

\newcommand{\Z}{\mathbb{Z}}
\newcommand{\Q}{\mathbb{Q}}

\newcommand{\C}{\mathbb{C}}
\renewcommand{\P}{\mathbb{P}}
\newcommand{\F}{\mathbb{F}}

\renewcommand{\a}{\alpha}
\renewcommand{\b}{\beta}
\renewcommand{\d}{\delta}

\newcommand{\e}{\varepsilon}
\newcommand{\E}{\mathbb{E}}
\renewcommand{\F}{\mathbb{F}}
\newcommand{\g}{\gamma}
\newcommand{\G}{\Gamma}
\newcommand{\s}{\sigma}
\renewcommand{\t}{\tau}
\newcommand{\vp}{\varphi}

\newcommand{\ds}{\displaystyle}
\newcommand{\mc}{\mathcal}
\newcommand{\ms}{\mathscr}

\newcommand{\wh}{\widehat}

\newcommand{\emp}{\varnothing}

\newcommand{\ol}{\overline}
\newcommand{\wt}{\widetilde}

\newcommand{\hra}{\hookrightarrow}
\newcommand{\epm}{\twoheadrightarrow}
\newcommand{\dra}{\dashrightarrow}

\DeclareMathOperator{\id}{id}

\DeclareMathOperator{\Sym}{\mathrm{Sym}}

\DeclareMathOperator{\Spec}{\mathrm{Spec}}

\DeclareMathOperator{\Proj}{\mathrm{Proj}}

\DeclareMathOperator{\Exc}{\mathrm{Exc}}

\DeclareMathOperator{\Pic}{\mathrm{Pic}}
\DeclareMathOperator{\Gr}{\mathrm{Gr}}
\DeclareMathOperator{\Fl}{\mathrm{Fl}}
\DeclareMathOperator{\red}{\mathrm{red}}

\DeclareMathOperator{\Bl}{\mathrm{Bl}}

\DeclareMathOperator{\Bs}{\mathrm{Bs}}

\DeclareMathOperator{\pr}{pr}

\DeclareMathOperator{\Hom}{Hom}

\DeclareMathOperator{\Ker}{Ker}

\DeclareMathOperator{\Ext}{Ext}

\DeclareMathOperator{\LG}{LG}
\DeclareMathOperator{\LF}{LF}

\let\Im\relax
\DeclareMathOperator{\Im}{\mathrm{Im}}

\title[{\tiny Weak Fano $3$-folds with sextic del Pezzo fibrations}]{Refinement of the classification of weak Fano threefolds with sextic del Pezzo fibrations}
\author[T. FUKUOKA]{Takeru Fukuoka}
\date{\today}
\address{Graduate School of Mathematical Sciences\\The University of Tokyo\\3-8-1 Komaba\\Meguro-ku, Tokyo 153-8914, Japan}
\email{tfukuoka@ms.u-tokyo.ac.jp}

\begin{document}
\maketitle

\begin{abstract}
We refine the classification of weak Fano threefolds with sextic del Pezzo fibrations by considering the Hodge numbers of them. 
By the refined classification result, such threefolds are classified into 17 cases. 
The main result of this paper is to show that there exists an example for each of 17 cases, except for 2 cases. 
\end{abstract}

\section{Introduction} 

This paper
concerns classification problems of weak Fano threefolds of Picard rank $2$. 
\emph{Weak Fano threefolds} are defined to be smooth projective threefolds with nef and big anti-canonical divisors. 
By the classification result of Fano threefolds of Picard rank $2$ due to Mori-Mukai \cite{MM81}, this problem is reduced to classify \emph{almost Fano threefolds}, which are defined to be smooth projective threefolds whose anti-canonical divisors are nef and big but not ample, i.e., weak Fano but not Fano. 
The possibilities of almost Fano threefolds $X$ of Picard rank $2$ have been treated by \cite{JPR05,JPR11,Takeuchi09,CM13}. 
These results include the list of the possible values of $(-K_{X})^{3}$ and the possible types of the Sarkisov links of them. 
For each possibility, it is natural to ask how to find examples. 
After the above works, the articles \cite{ACM17,CLM17,Fukuoka17} treated such problems. 
Especially, the classification of almost Fano del Pezzo fibrations of degree $d \neq 6$ was completed by \cite{Takeuchi09}, \cite{JPR05}, and the author's paper \cite{Fukuoka17}. 
Moreover, such threefolds with terminal anti-canonical models were classified up to deformation by Takeuchi \cite{Takeuchi09}. 

In this paper, we deal with almost Fano threefolds with sextic del Pezzo fibrations. 
Note that by Mori-Mukai's classification of Fano threefolds of Picard rank $2$ \cite{MM81}, it was known that there is no Fano threefold admitting a sextic del Pezzo fibration structure. 
In other words, all weak Fano threefolds with sextic del Pezzo fibrations are almost Fano. 
In the articles \cite{JPR05,JPR11}, Jahnke-Peternell-Radloff classified the possibilities of the anti-canonical degree and the Sarkisov links of such threefolds. 
Moreover, for each of the possibilities, Jahnke-Peternell-Radloff \cite{JPR11} and the author \cite{Fukuoka17} found examples for that possibility.

\subsection{The refined table}
The main purpose of this paper is to give a refinement of the classification of almost Fano threefolds with sextic del Pezzo fibration. 
To the author's best knowledge, 
the Hodge number $h^{1,2}(X)$ was not studied yet for the classification of almost Fano threefolds with sextic del Pezzo fibrations. 
Fortunately, the author's work \cite{Fukuoka18} enables us to determine the possible values of the Hodge number $h^{1,2}(X)$. 
For a given sextic del Pezzo fibration $\vp \colon X \to C$, 
there exist the double covering $\vp_{B} \colon B \to C$ and the triple covering $\vp_{T} \colon T \to C$ associated to $\vp$ (see \cite[Definition~3.5]{Fukuoka18}). 
Note that when $-K_{X}$ is nef and big, the base curve $C$ must be $\P^{1}$. 
Then by \cite[Theorem~A]{Fukuoka18}, we now already have the formulas 
\begin{align}
(-K_{X})^{3}=22-(6g(B)+4g(T)) \text{ and } h^{1,2}(X)=g(B)+g(T) \label{exist-eq-volh12}
\end{align}
which make the computation of $h^{1,2}(X)$ easier. 
Indeed, for a given almost Fano threefold with a sextic del Pezzo fibration $X \to \P^{1}$, 
we have $22 \geq (-K_{X})^{3} \geq 0$ by (\ref{exist-eq-volh12}) and hence we can easily determine the possible values of $g(B)$, $g(T)$, and $h^{1,2}(X)$ by (\ref{exist-eq-volh12}). 
By the result of \cite{JPR05,JPR11} and \cite[Theorem~A]{Fukuoka18}, we obtain the following refined table of almost Fano sextic del Pezzo fibrations. 

\begin{cor}[{of \cite{JPR05,JPR11} and \cite[Theorem~A]{Fukuoka18}}]\label{thm-invofWF}
Every weak Fano threefold of Picard rank $2$ with a sextic del Pezzo fibration $\vp \colon X \to \P^{1}$ satisfies one of the conditions in Table~\ref{table-wF}. 
\end{cor}

\begin{table}[h]
\caption[]{The refined table.}\label{table-wF}
\vspace*{1ex}
\begin{tabular}{|c||c|c|c|c|c||c|}
\hline 
No.& $(-K_{X})^{3}$ &$h^{1,2}(X)$ & $\dim \Exc(\psi)$& $V$ &$\vp^{+} \colon X^{+} \to V$&$\exists$ \\
\hline \hline 
1&$22$ &$0$ & $1$ &$B_{5}$&$(g,d)=(0,4)$&\cite{JPR11} \\
\hline
2& $18$ &$1$ & $1$ &$\Q^{3}$&$(g,d)=(1,6)$&\cite{JPR11}\\
\hline
3& $16$ &$1$ & $1$ &$\P^{3}$&$(g,d)=(1,6)$&\cite{JPR11} \\
\hline
4&$14$ &$2$ & $1$ &$\P^{2}$&$\deg \text{(disc.)}=4$&\cite{Fukuoka17} \\
\hline
5&$12$ &$2$& $1$ &$\P^{1}$&dP$_{6}$&\cite{Fukuoka17} \\
\hline
6&$12$ &$2$& $2$&\diagbox[dir=SW,height=1em]&\diagbox[dir=SW,height=1em]&\cite{Fukuoka17} \\
\hline
7&$10$ &$2$ &$1$&$V_{10}$&point&\cite{Takeuchi89} \\
\hline
8&$10$ &$3$ &$1$&$V_{9}$&$(g,d)=(0,2)$&\cite{Takeuchi89} \\
\hline
9&$8$ &$3$ &$1$&$B_{4}$&$(g,d)=(1,6)$&\cite{Fukuoka17}\\
\hline
10&$6$ &$3$ &$1$&$V_{10}$&$(g,d)=(1,6)$&\S~\ref{KG2-section}\\
\hline
11&$6$ &$3$ &$1$&$\P^{1}$&dP$_{6}$&?\\
\hline
12&$6$ &$4$ &$1$&$\P^{1}$&dP$_{6}$&\cite{Fukuoka17}\\
\hline
13&$4$ &$3$ &$1$&$\P^{1}$&dP$_{6}$&\cite{Fukuoka17}\\
\hline
14&$4$ &$4$ &$1$&$V_{9}$&$(g,d)=(1,6)$&\S~\ref{LG-section}\\
\hline
15&$4$ &$4$ &$1$&$\P^{1}$&dP$_{6}$&?\\
\hline
16&$2$ &$4$ &$1$&$\P^{1}$&dP$_{6}$&\S~\ref{24-section}\\
\hline
17&$2$ &$5$ &$1$&$\P^{1}$&dP$_{6}$&\cite{Fukuoka17}\\
\hline
\end{tabular}
\end{table}

\begin{center}
\textbf{Notation for Table~\ref{table-wF}.}
\end{center}
\begin{itemize}
\item The fourth row from the right denotes the dimension of the exceptional locus  of the contraction $\psi \colon X \to \ol{X}$ given by the semi-ample divisor $-K_{X}$. 
When $\psi$ is small, i.e., $\psi$ is a flopping contraction, we denote the flop of $X$ by $X^{+}$ and the contraction of the $K_{X^{+}}$-negative ray by $\vp^{+} \colon X^{+} \to V$. 
\item The third row from the right denotes the types of $V$. 
``$B_{d}$'' means that $V$ is a smooth Fano threefold of index $2$ with $(-K_{V})^{3}=8d$. 
``$V_{g}$'' means that $V$ is a smooth Fano threefold of index $1$ with $(-K_{V})^{3}=2g-2$. 
\item The second row from the right denotes the types of $\vp^{+} \colon X^{+} \to V$. 
\begin{itemize}
\item ``$(g,d)$'' means that $\vp^{+}$ is the blowing-up along a non-singular curve $C$ of genus $g$ with $-K_{V}.C=i_{V} \cdot d$, where $i_{V}$ denotes the Fano index of $V$. 
\item ``point'' means that $\vp^{+}$ is the blowing-up of $V$ at a point. 
\item ``deg(disc.)=$4$'' means that $\vp^{+} \colon X^{+} \to \P^{2}$ is a conic bundle and the degree of the discriminant divisor is 4. 
\item ``dP$_{6}$'' means that $\vp^{+} \colon X^{+} \to \P^{1}$ is a sextic del Pezzo fibration. 
\end{itemize}
\item Each entry of the rightmost row denotes the part including the proof of the existence of corresponding examples. 
\end{itemize}

This table is more accurate than that of \cite{JPR05} and \cite{JPR11} since we classify the possible values of $h^{1,2}(X)$; from the known results  \cite{JPR05} and \cite{JPR11}, we could not distinguish No.11 from No.12, No.14 from No.15, and No.16 from No.17.

\subsection{Main result}
The main result of this paper is to give examples of No.10, No.14, and No.16: 
\begin{thm}\label{mainthm-wf}
For each $n \in \{10,14,16\}$, there exists an example $X$ of No.$n$. 
Moreover, when $n=10$ or $14$, the flop $X \dra X^{+}$ of this example $X$ is an Atiyah flop. 
When $n=10$, $X$ has $18$ sections as its flopping locus. 
When $n=14$, $X$ has $24$ sections and $3$ bisections as its flopping locus. 
\end{thm}

\begin{rem}
It is still uncertain whether there exists an example of No.11 or No.15. 
\end{rem}

\begin{rem}
The existence of an example of No.10 (resp. No.14) was treated by the author in \cite{Fukuoka17}. 
However, 
the author realized that the proof of the existence of examples of No.10 and No.14 in \cite{Fukuoka17} has a gap by Joseph Cutrone's pointing out. 
In \cite[Sections~3.2 and 3.3]{Fukuoka17}, the author stated that  ``$\vp^{+}$ is a del Pezzo fibration because $X^{+}$ is not on the tables on \cite{JPR11} and \cite{CM13}''. 
However, the numerical condition of $X^{+}$ in that proof actually satisfies No.80 (resp. No.59) on the table of \cite{CM13}. 
The author's proof in the existence of No.10 (resp. No.14) did not exclude the possibility that $X^{+}$ belongs to No.80 (resp. No.59) of \cite{CM13}. 

Another purpose of this paper is to attain the correct proof of the existence of examples. 
As a result, the main result of the paper \cite{Fukuoka17} is correct. 
Our construction of examples in this paper is more explicit than that of the paper \cite{Fukuoka17}. 
In particular, we describe the flopping locus of our examples in detail, which is also a refinement. 
\end{rem}


\subsection{Our strategy for the existence of examples}\label{wf-subsec-sketch}

By \cite{JPR11}, it is known that every threefold of No.14 (resp. No.10) must fit into the following diagram: 
\[\xymatrix{
&\ar[ld]_{\vp_{X^{+}}}X^{+} \ar@{-->}[r]^{\chi}&X\ar[rd]^{\vp_{X}}&\\
\P^{1}&&&V,
}\]
where $\chi$ is a flop of $X^{+}$, $V$ is a smooth Fano threefold of index $1$ and genus $10$ (resp. $9$), and $\vp_{X}$ is the blowing-up along an elliptic curve $C \subset V$ of degree $6$. 
In order to construct an example of No.14 (resp. No.10), 
we will construct an elliptic curve $C$ of degree $6$ on a smooth Fano threefold $V$ of genus $10$ (resp. $9$) such that the flop of $\Bl_{C}V$ has a sextic del Pezzo fibration structure. 

By the virtue of Mukai's classification \cite{Muk95}, it was established that $V$ is a linear section of $K(G_{2})$ (resp. $\LG(3,6)$). 
For each case, we will take such an elliptic curve as a linear section of a sextic del Pezzo threefold, which is the zero locus of a general global section of the specific vector bundle on $K(G_{2})$ (resp. $\LG(3,6)$). 
In the case of $K(G_{2})$, we will take the universal rank $2$ vector bundle $\mc E$ corresponding to the special lines on $\Q^{5}$ (cf. \cite{Ott90}). 
In the case of $\LG(3,6)$, we will take the wedge product of the universal quotient bundle $\bigwedge^{2}\mc Q_{\LG}$. 
Then the zero locus of a general global section of $\mc E$ (resp. $\bigwedge^{2}\mc Q_{\LG}$) turns out to be isomorphic to $\P(T_{\P^{2}})$ (resp. $(\P^{1})^{3}$), which is a sextic del Pezzo threefold. 
In the case of $(K(G_{2}),\P(T_{\P^{2}}))$, 
taking a codimension $2$ linear section, 
we obtain a Fano threefold $V^{3}_{10}$ of genus $10$ containing an elliptic curve $C$ of degree $6$. 
As a result, the flop of the blowing-up of $V^{3}_{10}$ along $C$ belongs to No.14. 
In the case of $(\LG(3,6),(\P^{1})^{3})$, 
taking a codimension $2$ linear section, 
we obtain a Fano fourfold $V^{4}_{9}$ of genus $9$ containing an elliptic curve $C$ of degree $6$. 
We will take a hyperplane section of $V^{4}_{9}$, say $V^{3}_{9}$, which is smooth and contains $C$. 
As a result, the flop of the blowing-up of $V^{3}_{9}$ along $C$ belongs to No.14. 

The existence of examples of No.16 will be done by the same arguments as in \cite{Fukuoka17}. 


\begin{ACK}
Most of the contents of this paper were studied during the author's stay at University of Warwick. 
The author wishes to express his deep gratitude to Professor Miles Reid for valuable discussions, comments, and warm encouragement. 
The author is also grateful to the institution for the hospitality. 
The author also wishes to express his deep gratitude to Professor Hiromichi Takagi, his supervisor, for many valuable suggestions and warm encouragement to write this paper.
The author also wishes to express his deep gratitude to Professor Yoshinori Gongyo, his supervisor, for warm encouragement. 

The author wishes to express his gratitude to Professors Joseph Cutrone and Nicholas Marshburn for pointing out the error of the author's paper \cite{Fukuoka17} to him. 
The author is grateful to Doctor Akihiro Kanemitsu for valuable discussions on homogenous spaces and vector bundles. 
The author is also grateful to Doctor Yuya Takeuchi for valuable discussions on Lemma~\ref{LG-lem-P1n}. 
The author also would like to show his gratitude to Professors 
Hamid Ahmadinezhad, Gavin Brown, Ivan Cheltsov, DongSeon Hwang, and Doctor Isac Hed\'{e}n for warm encouragements. 

This work was financially supported by the Program for Leading Graduate Schools, MEXT, Japan.
This work was also partially supported by JSPS's Research Fellowship for Young Scientists (JSPS KAKENHI Grant Number 18J12949).
\end{ACK}

\begin{NC}
Throughout this paper, we work over the complex number field $\C$. 
Vector bundles and line bundles just mean locally free sheaves and invertible sheaves respectively. 
For a locally free sheaf $\mc E$, we define $\P(\mc E)$ as $\Proj \Sym \mc E$. 
The $i$-th projection $X_{1} \times \cdots \times X_{k} \to X_{i}$ is denoted by $\pr_{i}$. 
On $\P^{n_{1}} \times \cdots \times \P^{n_{k}}$, the line bundle $\bigotimes_{i=1}^{l}\pr_{i}^{\ast}\mc O(a_{i})$ is denoted by $\mc O(a_{1},\ldots,a_{k})$. 
For a morphism $f \colon X \to Y$, the K\"ahler differential $\Omega_{X/Y}$ is also denoted by $\Omega_{f}$. 

We further employ the following definition for the flag varieties and the Grassmannian varieties. 
For a $n$-dimensional vector space $V$ and an increasing integers $k_{1}<\cdots<k_{l}$ with $k_{l}<\dim V=n$, 
$\Fl(k_{1},\ldots,k_{l};V)=\Fl(k_{1},\ldots,k_{l};n)$ denotes the flag variety parametrizing the increasing sequences $V_{1} \subset \cdots \subset V_{k_{l}}$ of subspaces of $V^{\vee}$ such that $\dim V_{i}=k_{i}$. 
When $l=1$, $\Fl(k_{1};V)$ is denoted by $\Gr(k_{1},V)$, which is the Grassmannian variety parametrizing the $k_{1}$-dimensional subspaces of $V^{\vee}$. 
On the Grassmannian variety $\Gr(m,V)$, the universal quotient (resp. sub) bundle of rank $m$ (resp. $n-m$) is denoted by $\mc Q_{\Gr(m,V)}$ (resp. $\mc S_{\Gr(m,V)}$). 
\end{NC}

\section{Construction of an example of No.10}\label{KG2-section}

In this section, we will construct an example of No.10. 

\subsection{Sketch of our construction}\label{KG2-subsec-sketch}
As mentioned in Subsection~\ref{wf-subsec-sketch}, 
we will show that the zero scheme $\Sigma$ of a global section of the rank $2$ universal bundle $\mc E$ on $K(G_{2})$ is isomorphic to $\P(T_{\P^{2}})$ (=Proposition~\ref{KG2-prop-TP2}) by using Ottaviani's work \cite{Ott90}. 
Taking a general codimension $2$ linear section of $(K(G_{2}),\P(T_{\P^{2}}))$, we obtain a sextic elliptic curve $C$ on a Fano threefold $V^{3}_{10}$ of genus $10$. 

In order to study $\Bl_{C}V^{3}_{10}$, we will study $W:=\Bl_{\Sigma}K(G_{2})$. 
We will see that $W$ is isomorphic to the projectivization of the Cayley bundle $\ms C_{\Q^{4}}$ on $\Q^{4}$ (see (\ref{KG2-dia-W})). 
We denote the $\P^{1}$-bundle structure by $q_{W} \colon W \to \Q^{4}$. 
By this construction, the blowing-up $X:=\Bl_{C}V^{3}_{10}$ is contained in $W$ and we obtain a morphism $q_{W}|_{X} \colon X \to \Q^{4}$. 
Set $\ol{X}:=q_{W}(X) \subset \Q^{4}$ and $\psi:=q_{W}|_{X} \colon X \to \ol{X}$. 
Then we will see that $\psi$ is a flopping contraction (=Proposition~\ref{KG2-prop-conclu}~(1)). 

To show that the flop of $X$, say $X^{+}$, has a sextic del Pezzo fibration structure, we need to construct $X^{+}$ explicitly. 
We denote the exceptional divisor of $X \to V^{3}_{10}$ by $E_{C}$ and set $R:=\psi(E_{C}) \subset \ol{X}$. 
In Proposition~\ref{KG2-lem-EC}, we prove that $\psi|_{E_{C}} \colon E_{C} \to R$ is isomorphic and hence $R$ is a smooth elliptic scroll in $\P^{5}$. 
Since $E_{C}$ is $\psi$-ample and $R$ is smooth, we can show that the blowing-up $X^{+}:=\Bl_{R}\ol{X}$ is the flop of $X$ (=Proposition~\ref{KG2-prop-conclu}~(2)). 

This elliptic scroll $R \subset \Q^{4} \subset \P^{5}$ is also constructed by Addington, Hassett, Tschinkel, and V\'{a}rilly-Alvarado in their paper \cite{AHTVA16}. 
We will study the blowing-up $Y:=\Bl_{R}\P^{5} \to \P^{5}$ by using some results in \cite{AHTVA16}. 
Then we will show that $Y$ has a flop $\Psi \colon Y \dra Y^{+}$ and that $Y^{+}$ is a smooth fivefold. 
Moreover, $Y^{+}$ has a $(\P^{1})^{3}$-fibration structure $\vp_{Y^{+}} \colon Y^{+} \to \P^{2}$. 
Let $Q_{Y} \subset Y$ and $Q_{Y^{+}} \subset Y^{+}$ be the proper transforms of the hyperquadric $\Q^{4} \subset \P^{5}$ containing $R$. 
Then we will prove that the birational map $\Psi|_{Q_{Y}} \colon Q_{Y} \dra Q_{Y^{+}}$ is a flip and that $Q_{Y^{+}}$ is smooth. 
Moreover, this $Q_{Y^{+}}$ is the pull-back of a line under the morphism $\vp_{Y^{+}} \colon Y^{+} \to \P^{2}$ and hence $Q_{Y^{+}} \to \P^{1}$ is a fibration with general fibers $(\P^{1})^{3}$. 
For more precise, see Proposition~\ref{KG2-prop-stflop}. 
In summary, we will have the following diagram: 
\[\xymatrix{
\Bl_{\Sigma}K(G_{2}) \ar@{}[r]|{\hspace{20pt}=}&W \ar[ld]_{p_{W}} \ar[rd]^{q_{W}}&\P_{\Q^{4}}(\ms C_{\Q^{4}}(2)) \ar@{}[l]|{=\hspace{20pt}}&\ar[ld] Q_{Y}:=\Bl_{R}\Q^{4} \ar@{-->}[r]^{\qquad \Psi|_{Q_{Y}}}_{\qquad \text{flip}}&Q_{Y^{+}} \ar[d]\\
K(G_{2})&&\Q^{4} &&\P^{1}.
}\]

We show that the flop $X^{+}$ of $X$ has a sextic del Pezzo fibration as follows. 
Since $\ol{X}$ is contained in $\Q^{4}$, $X^{+}=\Bl_{R}\ol{X}$ is also contained in $Q_{Y}=\Bl_{R}\Q^{4}$ from the construction. 
By taking the codimension $2$ linear section $(V^{3}_{10},C)$ of $(K(G_{2}),\Sigma=\P(T_{\P^{2}}))$ generally, we will see that $X^{+} \subset Q_{Y}$ is away from the flipping locus of $\Psi_{Q_{Y}}$ (=Proposition~\ref{KG2-prop-conclu}~(3)). 
Hence we obtain the following diagram: 

\[\xymatrix{
&\Bl_{C}V^{3}_{10}=X \ar[ld]_{p_{W}|_{X}} \ar[rd]^{q_{W}|_{X}=:\psi} \ar@{-->}[rr]^{\text{flop}}&&\ar[ld] \Bl_{R}\ol{X}=X^{+} \ar@{=}[r]^{\qquad \Psi|_{X^{+}}}&X^{+}\ar[d]&\\
V^{3}_{10}&&\ol{X}&&\P^{1}.
}\]

Moreover, we will show that $X^{+}$ is contained in the $(\P^{1})^{3}$-fibration $Q_{Y^{+}} \to \P^{1}$ as a relative hyperplane section. 
Therefore, $X^{+}$ has a sextic del Pezzo fibration structure.

\subsection{Organization of this section}
We organize this section as follows. 

In Subsection~\ref{KG2-subsec-2P2}, we will study the blowing-up of $\P^{5}$ along disjoint union of two $2$-planes as a preliminary and recall the construction of the sextic elliptic scroll $R$ due to \cite{AHTVA16}. 
For fixed disjoint two $2$-planes $P_{1}$ and $P_{2}$, 
the blowing-up of $\P^{5}$ along $P_{1} \sqcup P_{2}$ has a natural $\P^{1}$-bundle structure on $\P^{2} \times \P^{2}$ (=Lemma~\ref{KG2-lem-2P2}). 
Cutting $(\P^{2})^{2}$ three times, we will find an elliptic scroll $R$ on $\P^{5}$ of degree $6$ (=Lemma~\ref{KG2-lem-EC}). 
This $R$ will be the image of the exceptional divisor of $X:=\Bl_{C}V^{3}_{10} \to V^{3}_{10}$ under the flopping contraction $X \to \ol{X}$. 

In Subsection~\ref{KG2-subsec-ellscr}, we will study the blowing-up $Y:=\Bl_{R}\P^{5}$. 
The main purpose is to show Proposition~\ref{KG2-prop-stflop}, which states that $Y$ is flopped into a smooth fivefold $Y^{+}$ with a $(\P^{1})^{3}$-fibration structure $Y^{+} \to \P^{2}$. 
The flopped locus is the disjoint union of the proper transforms of $P_{1}$ and $P_{2}$. 
Taking a non-singular hyperquadric $\Q^{4}$ containing $P_{1}$, $P_{2}$, and $R$, 
we will see that $Q_{Y}:=\Bl_{R}\Q^{4}$ is flipped into a smooth fourfold $Q_{Y^{+}}$ with a $(\P^{1})^{3}$-fibration $Q_{Y^{+}} \to \P^{1}$. 

In Subsection~\ref{KG2-subsec-cayley}, 
we will study the zero scheme $\Sigma$ of a general section of the rank $2$ universal vector bundle $\mc E$ on $K(G_{2})$ (cf. Theorem~\ref{KG2-thm-Ott}) and the blowing-up $\Bl_{\Sigma}K(G_{2})$ by using \cite{Ott90}. 
In Proposition~\ref{KG2-prop-TP2}, we will prove that $\Sigma \simeq \P(T_{\P^{2}})$. 
The blowing-up $\Bl_{\Sigma}K(G_{2})$ has the $\P^{1}$-bundle structure $\Bl_{\Sigma}K(G_{2}) \to \Q^{4}$, which coincides with the projectivization of the Cayley bundle on $\Q^{4}$. 


In Subsection~\ref{KG2-subsec-conclu}, we will study the Sarkisov link of $\Bl_{C}V^{3}_{10}$, where $(V^{3}_{10},C)$ is a general codimension $2$ linear section of $(K(G_{2}),\Sigma)$. 
We will prove that $\Bl_{C}V^{3}_{10}$ is flopped into a sextic del Pezzo fibration in the way which was explained in Subsection~\ref{KG2-subsec-sketch}. 


\subsection{Preliminaries for the construction}\label{KG2-subsec-2P2}

We fix disjoint two $2$-planes $P_{1}$ and $P_{2}$ on $\P^{5}$. 

\begin{lem}\label{KG2-lem-2P2}
Set $f \colon \wt{\P^{5}}:=\Bl_{P_{1},P_{2}}\P^{5} \to \P^{5}$. 
Then $\wt{\P^{5}}$ is isomorphic to $\P_{(\P^{2})^{2}}(\mc O(1,0) \oplus \mc O(0,1))$. 
We consider the following diagram: 
\[\xymatrix{
(\P^{2})^{2}&\ar[l]_{\pi \hspace{50pt} } \P_{(\P^{2})^{2}}(\mc O(1,0) \oplus \mc O(0,1))=\wt{\P^{5}} \ar[r]^{\hspace{50pt} f} &\P^{5}.
}\]
Then we have 
\begin{align}\label{KG2-rel-E}
\pi^{\ast}\mc O(1,1) \simeq f^{\ast}\mc O_{\wt{\P^{5}}}(2) \otimes \mc O(-\Exc(f)). 
\end{align}
\end{lem}

\begin{proof}
See the proof of \cite[Proposition~3]{AHTVA16} for example. 
\end{proof}

\begin{lem}\label{KG2-lem-E}
Let $\Sigma \in |\mc O_{(\P^{2})^{2}}(1,1)|$ be a smooth element. 
Note that $\Sigma$ is isomorphic to $\P(T_{\P^{2}})$. 
Set $E:=\pi^{-1}(\Sigma)$. 
Then $f(E)$ is a smooth quadric fourfold $\Q^{4}$ containing $P_{1}$ and $P_{2}$: 
\[\xymatrix{
\P(T_{\P^{2}})=\Sigma&\ar[l]_{\pi|_{E} \hspace{60pt}} \P_{\P(T_{\P^{2}})}(\mc O_{\P(T_{\P^{2}})}(1,0) \oplus \mc O_{\P(T_{\P^{2}})}(0,1))=E \ar[r]^{\hspace{60pt} f|_{E}} & f(E)=\Q^{4}. 
}\]
\end{lem}
\begin{proof}
By (\ref{KG2-rel-E}), the image $f(E)$ is a quadric. 
Set $\Sigma_{i}:=f^{-1}(P_{i}) \cap E$ for each $i$. 
Then $\Sigma_{i} \simeq \P(T_{\P^{2}})$ and $f|_{\Sigma_{i}} \colon \Sigma_{i} \to P_{i}$ is a $\P^{1}$-bundle structure. 
Therefore, we have $\Exc(f|_{E})=\Sigma_{1} \sqcup \Sigma_{2}$ and the image of $f|_{E}$ is smooth by Ando's theorem \cite[Theorem~2.3]{Ando85}. 
Thus $f(E)$ is smooth. 
\end{proof}

\begin{lem}\label{KG2-lem-ES}
Let $S \subset \Sigma=\P(T_{\P^{2}})$ be a smooth hyperplane section, $E_{S}=\pi^{-1}(S)$, and $\ol{E_{S}}:=f(E_{S})$. 
Then $\ol{E_{S}}$ is a complete intersection of two quadrics containing $P_{1}$ and $P_{2}$. 
Moreover, the morphism $f|_{E_{S}} \colon E_{S} \to \ol{E_{S}}$ is a flopping contraction:  
\[\xymatrix{
S&\ar[l]_{\pi|_{E_{S}}} E_{S} \ar[r]^{f|_{E_{S}}} & \ol{E_{S}}. 
}\]
\end{lem}

\begin{proof}
It follows from (\ref{KG2-rel-E}) and Lemma~\ref{KG2-lem-E} that $\ol{E_{S}}$ is the complete intersection of two quadrics containing $P_{1}$ and $P_{2}$. 
Set $S_{i}:=E_{S} \cap \Sigma_{i}$ for each $i \in \{1,2\}$. 
Then $S_{i} \to P_{i}$ is the blowing-up of $\P^{2}=P_{i}$ at three points. 
Since $\Exc(f|_{E_{S}}) \subset \Exc(f) \cap E_{S}=\bigsqcup_{i=1}^{2} S_{i}$, 
the exceptional locus $\Exc(f|_{E_{S}})$ is purely $1$-dimensional. 
In particular, $\ol{E_{S}}$ is regular in codimension $1$. 
Hence $\ol{E_{S}}$ is normal and $f|_{E_{S}}$ is a small birational contraction. 
Since the adjunction formula gives $\omega_{E_{S}}^{-1} = f|_{E_{S}}^{\ast}\mc O_{\ol{E_{S}}}(2)$, we conclude that $f|_{E_{S}}$ is flopping. 
\end{proof}

\begin{lem}\label{KG2-lem-EC}
Let $C \subset (\P^{2})^{2}$ be a codimension $3$ linear section with respect to $|\mc O(1,1)|$. 
Set $E_{C}:=\pi^{-1}(C)$ and $R:=f(E_{C})$: 
\[\xymatrix{
C&\ar[l]_{\pi|_{E_{C}}} E_{C} \ar[r]^{f|_{E_{C}}} &R.
}\]
Then the following assertions hold:
\begin{enumerate}
\item The restriction $f|_{E_{C}} \colon E_{C} \to R$ is isomorphic. 
\item The intersection $C_{i}:=P_{i} \cap R$ be a smooth cubic curve on $P_{i}=\P^{2}$ for each $i \in \{1,2\}$. 
\item If a hyperquadric $Q$ on $\P^{5}$ contains $R$, then $Q$ contains $P_{1}$ and $P_{2}$. 
\item Let $\Sigma_{a},\Sigma_{b},\Sigma_{c} \subset (\P^{2})^{2}$ be general hyperplane sections such that $C=\Sigma_{a} \cap \Sigma_{b} \cap \Sigma_{c}$. 
Set $Q_{i}:=f(\pi^{-1}(\Sigma_{i}))$ for each $i$, which is a smooth quadric by Lemma~\ref{KG2-lem-E}. 
Then $R \cup P_{1} \cup P_{2}$ with the reduced scheme structure is a complete linear section of the three smooth quadrics $Q_{a}$, $Q_{b}$, and $Q_{c}$. 
In particular, the divisor $Q_{a}+Q_{b}+Q_{c}$ is simply normal crossing at the generic point of $P_{1}$, $P_{2}$, and $R$. 
\end{enumerate}
\end{lem}
\begin{proof}
See \cite[Proposition~3]{AHTVA16} and its proof. 
\end{proof}

\subsection{Blowing-up of $\P^{5}$ along the elliptic scroll $R$}\label{KG2-subsec-ellscr}

\begin{sett}\label{KG2-sett-Y}
Fix a ladder $C \subset S \subset \Sigma \subset (\P^{2})^{2}$ by members of $|\mc O_{(\P^{2})^{2}}(1,1)|$ such that each rung is smooth. 
Set $E$, $Q$, $E_{S}$, $\ol{E_{S}}$, $E_{C}$, and $R$ as in Lemmas~\ref{KG2-lem-E}, \ref{KG2-lem-ES}, and \ref{KG2-lem-EC}. 

In this situation, we let 
$\s \colon Y:=\Bl_{R}\P^{5} \to \P^{5}$ 
be the blowing-up along $R$. 
Let $L_{Y}$ be the pull-back of a hyperplane on $\P^{5}$ by $\s$ and $G=\Exc(\s)$. 
Let $Q_{Y}$, $E_{S,Y}$, and $P_{i,Y}$ be the proper transform of $Q$, $E_{S}$, and $P_{i}$ on $Y$ respectively. 
Note that the subvarieties 
$E_{S,Y} \subset Q_{Y} \subset Y$ 
form a ladder of $|2L_{Y}-G|$. 
\end{sett}

The main purpose of this subsection is to show the following proposition. 
\begin{prop}\label{KG2-prop-stflop}
In Setting~\ref{KG2-sett-Y}, the following assertions hold. 
\begin{enumerate}
\item There exists a flop $\Psi \colon Y \dra Y^{+}$ of $P_{1,Y}$ and $P_{2,Y}$, which are disjoint and isomorphic to $\P^{2}$. 
The fivefold $Y^{+}$ is smooth and has an extremal contraction $\vp_{Y^{+}} \colon Y^{+} \to \P^{2}$. 
\item Let $Q_{Y^{+}}$ be the proper transform of $Q_{Y}$ on $Y^{+}$
Then the birational map $\Psi|_{Q_{Y}} \colon Q_{Y} \dra Q_{Y^{+}}$ is a flip of $P_{1,Y}$ and $P_{2,Y}$. 
The fourfold $Q_{Y^{+}}$ is smooth and the pull-back of a line under the morphism $\vp_{Y^{+}}$. 
The morphism $\vp_{Q_{Y^{+}}}:=\vp_{Y^{+}}|_{Q_{Y^{+}}} \colon Q_{Y^{+}} \to \P^{1}$ is also an extremal contraction. 
\item Let $E_{S,Y^{+}}$ be the proper transform of $E_{S,Y}$ on $Y^{+}$. 
Then the birational map $\Psi|_{E_{S,Y}} \colon E_{S,Y} \to E_{S,Y^{+}}$ is the blowing-down of $P_{1,Y}$ and $P_{2,Y}$ to two smooth points. 
Moreover, $E_{S,Y^{+}}$ is a fiber of $\vp_{Y^{+}}$ which is isomorphic to $(\P^{1})^{3}$. 
\item Every smooth $\vp_{Y^{+}}$-fiber is isomorphic to $(\P^{1})^{3}$. 
\end{enumerate}
\end{prop}
As a consequence of the above proposition, we obtain the following diagram: 

\begin{align}\label{KG2-dia-stflop}
\xymatrix{
&\ar[ld]_{\pi}\wt{\P^{5}}\ar[rd]^{f}&&\Bl_{R}\P^{5}=Y \ar@{-->}[r]^{\quad \Psi} \ar[ld]_{\s} & Y^{+} \ar[rd]^{\vp_{Y^{+}}} & \\
(\P^{2})^{2}&\ar[ld]_{\quad \pi|_{E}} \ar@{}[l]|{\Box}E\ar[rd]^{f|_{E}\quad } \ar@{}[u]|{\rotatebox{90}{$\subset$}}&\P^{5}&\Bl_{R}\Q^{4}=Q_{Y}\ar@{-->}[r]^{\quad \Psi|_{Q_{Y}}} \ar[ld]_{\quad \s|_{Q_{Y}}}\ar@{}[u]|{\rotatebox{90}{$\subset$}}&Q_{Y^{+}}\ar[rd]^{\vp_{Q_{Y^{+}}}\qquad } \ar@{}[r]|{\Box} \ar@{}[u]|{\rotatebox{90}{$\subset$}}&\P^{2} \\
\Sigma \ar@{}[u]|{\rotatebox{90}{$\subset$}}&\ar[ld]_{\qquad \pi|_{E_{S}}} \ar@{}[l]|{\Box} E_{S}\ar[rd]^{f|_{E_{S}}} \ar@{}[u]|{\rotatebox{90}{$\subset$}}&\Q^{4} \ar@{}[u]|{\rotatebox{90}{$\subset$}}&E_{S,Y} \ar[ld]_{\qquad \s|_{E_{S,Y}}} \ar[r]^{\Psi|_{E_{S,Y}}} \ar@{}[u]|{\rotatebox{90}{$\subset$}}&E_{S,Y^{+}} \ar@{}[r]|{\Box} \ar[rd] \ar@{}[u]|{\rotatebox{90}{$\subset$}}&\P^{1}\ar@{}[u]|{\rotatebox{90}{$\subset$}} \\
S\ar@{}[u]|{\rotatebox{90}{$\subset$}}&\ar@{}[l]|{\Box} E_{C} \ar[ld]_{\qquad \pi|_{E_{C}}} \ar[rd]^{f|_{E_{C}}} \ar@{}[u]|{\rotatebox{90}{$\subset$}}&\ol{E_{S}} \ar@{}[u]|{\rotatebox{90}{$\subset$}}&&(\P^{1})^{3}\ar@{}[u]|{\rotatebox{90}{$\simeq$}}&pt. \ar@{}[u]|{\rotatebox{90}{$\in$}} \\
C\ar@{}[u]|{\rotatebox{90}{$\subset$}}&&R\ar@{}[u]|{\rotatebox{90}{$\subset$}}&&&
}
\end{align}


In order to show Proposition~\ref{KG2-prop-stflop}, we prepare the following lemma. 

\begin{lem}\label{KG2-lem-stflop}
In Setting~\ref{KG2-sett-Y}, the following assertions hold:
\begin{enumerate}
\item The proper transform $E_{S,Y} \subset Y$ of $E_{S}$ is the flop of $E_{S} \to \ol{E_{S}}$. In particular, $E_{S,Y}$ is smooth. 
\item The proper transform $P_{i,Y}$ is isomorphic to $\P^{2}$. 
Moreover, it holds that $\mc N_{P_{i,Y}/E_{S,Y}}=\mc O_{\P^{2}}(-1)$, $\mc N_{P_{i,Y}/Q_{Y}}=\mc O_{\P^{2}}(-1)^{2}$, and $\mc N_{P_{i,Y}/Y}=\mc O_{\P^{2}}(-1)^{3}$. 
\item The complete linear system $|2L_{Y}-G|$ is $2$-dimensional with $\Bs|2L_{Y}-G|=P_{1,Y} \sqcup P_{2,Y}$. 
\item The complete linear system $|3L_{Y}-G|$ is base point free. The Stein factorization $\psi_{Y} \colon Y \to \ol{Y}$ of the morphism given by $|3L_{Y}-G|$ is a flopping contraction. 
Its exceptional locus $\Exc(\psi_{Y})$ is a disjoint union of $P_{1,Y}$ and $P_{2,Y}$. 
Moreover, if $\Psi \colon Y \dra Y^{+}$ denotes the flop, then $Y^{+}$ is smooth. 
\item Set $\ol{Q_{Y}}:=\psi_{Y}(Q_{Y})$. 
Then $\psi_{Q_{Y}}:=\psi_{Y}|_{Q_{Y}} \colon Q_{Y} \to \ol{Q_{Y}}$ is a flipping contraction. 
Its exceptional locus $\Exc(\psi_{Q_{Y}})$ is the disjoint union of $P_{1,Y}$ and $P_{2,Y}$ and its flip is smooth. 
\item Let $Q_{Y^{+}}$ be the proper transform of $Q_{Y}$ on $Y^{+}$. 
Then $\Psi|_{Q_{Y}} \colon Q_{Y} \dra Q_{Y^{+}}$ is the flip for the contraction $\psi_{Q_{Y}} \colon Q_{Y} \to \ol{Q_{Y}}$. 
Hence $Q_{Y^{+}}$ is smooth. 
\item Let $E_{S,Y^{+}}$ be the proper transform of $E_{S}$ on $Y^{+}$. 
Then $\Psi|_{E_{S,Y}} \colon E_{S,Y} \to E_{S,Y^{+}}$ blows $P_{1,Y}$ and $P_{2,Y}$ down to two smooth points. 
\end{enumerate}
\end{lem}
\begin{proof}
(1) Note that $\s|_{Q_{Y}} \colon Q_{Y}=\Bl_{R}Q \to Q$ is the blowing-up of $Q$ along the smooth center $R$. 
Thus every non-trivial fiber of $\s|_{E_{S,Y}} \colon E_{S,Y} \to E_{S}$ is isomorphic to $\P^{1}$. 
Since $\ol{E_{S}}$ is normal and contains $R$ as a Weil divisor, $\s|_{E_{S,Y}}$ is a small contraction, which implies $E_{S,Y}$ is regular in codimension $1$. 
Since $E_{S,Y}$ is a divisor of the smooth fourfold $Q_{Y}$, $E_{S,Y}$ is Gorenstein. 
Therefore, $E_{S,Y}$ is normal. 
From the construction, $-G|_{E_{S,Y}}$ is relatively ample over $\ol{E_{S}}$. 
On the other hand, the proper transform of $G|_{E_{S,Y}}$ on $E_{S}$ is $E_{C}$, which is relatively ample over $\ol{E_{S}}$ as a divisor on $E_{S}$. 
Thus $E_{S,Y}$ is the flop of $E_{S}$. 

(2) It follows from Lemma~\ref{KG2-lem-EC}~(2) that $P_{i,Y} \simeq \P^{2}$ for each $i$. 
The proper transform $E_{S,Y}$ is smooth by (1) and contains $P_{1,Y}$ and $P_{2,Y}$ as divisors. 
Since $\omega_{E_{S,Y}}=(\s|_{E_{S,Y}})^{\ast}\omega_{\ol{E_{S}}}=(\s|_{E_{S,Y}})^{\ast}\mc O_{\ol{E_{S}}}(-2)$ and $\omega_{P_{i,Y}}^{-1} \simeq \mc O_{\P^{2}}(-3)$, 
we get $\mc N_{P_{i,Y}/E_{S,Y}}=\mc O_{\P^{2}}(-1)$. 
Using the inclusions $P_{i,Y} \subset E_{S,Y} \subset Q_{Y} \subset Y$, we get the following exact sequences for each $i \in \{1,2\}$: 
\begin{align*}
&0 \to \mc N_{P_{i,Y}/E_{S,Y}} \to \mc N_{P_{i,Y}/Q_{Y}} \to \mc N_{E_{S,Y}/Q_{Y}}|_{P_{i,Y}} \to 0 \text{ and }\\
&0 \to \mc N_{P_{i,Y}/Q_{Y}} \to \mc N_{P_{i,Y}/Y} \to \mc N_{Q_{Y}/Y}|_{P_{i,Y}} \to 0.
\end{align*}
Note that $E_{S,Y} \sim (2L_{Y}-G)|_{Q_{Y}}$ 
and $Q_{Y} \sim 2L_{Y}-G$. 
Thus we obtain $\mc N_{E_{S,Y}/Q_{Y}}|_{P_{i,Y}}=\mc N_{Q_{Y}/Y}|_{P_{i,Y}}=\mc O_{\P^{2}}(-1)$ by Lemma~\ref{KG2-lem-EC}~(2). 
Hence it follows from the above exact sequences that $\mc N_{P_{i,Y}/Q_{Y}}=\mc O_{\P^{2}}(-1)^{2}$ and $\mc N_{P_{i,Y}/Y}=\mc O_{\P^{2}}(-1)^{3}$.  

(3) By Lemma~\ref{KG2-lem-EC}~(3) and (4), we have $\dim |\mc O_{\P^{5}}(2) \otimes \mc I_{R}|=\dim |\mc O_{\P^{5}}(2) \otimes \mc I_{R \cup P_{1} \cup P_{2}}|=2$. 
Let $\Sigma_{a}$, $\Sigma_{b}$, $\Sigma_{c}$, $Q_{a}$, $Q_{b}$, and $Q_{c}$ be as in Lemma~\ref{KG2-lem-EC}~(4). 
Let $Q_{a,Y}$, $Q_{b,Y}$, and $Q_{c,Y}$ be the proper transform of $Q_{a}$, $Q_{b}$, and $Q_{c}$ on $Y$ respectively. 
We may assume that $S_{ab}:=\Sigma_{a} \cap \Sigma_{b}$ is smooth. 
By Lemma~\ref{KG2-lem-ES} and (1), $E_{S_{ab}} \to \ol{E_{S_{ab}}} \gets E_{S_{ab},Y}$ is the flop. 
Since $Q_{c,Y}|_{E_{S_{ab,Y}}} = P_{1,Y}+P_{2,Y}$ as divisors on $E_{S_{ab,Y}}$, 
we get $\Bs |2L_{Y}-G|=Q_{a,Y} \cap Q_{b,Y} \cap Q_{c,Y}=P_{1,Y} \sqcup P_{2,Y}$. 

(4) 
By (3) and its proof, we obtain the exact sequence 
$0 \to \mc O(-3L_{Y}+2G) \to \mc O(-L_{Y}+G)^{3} \to \mc O(L_{Y})^{3} \to \mc I_{P_{1,Y} \sqcup P_{2,Y}}(3L_{Y}-G) \to 0$. 
Since 
$h^{3}(\mc O(-3L_{Y}+2G))=h^{3}(\mc O(3L_{Y}+K_{Y}))=0$, 
$h^{2}(\mc O(-L_{Y}+G))=h^{2}(\mc O(5L_{Y}-G+K_{Y}))=0$, 
and $h^{1}(\mc O(L_{Y}))=0$ hold, 
we have the vanishing $H^{1}(Y, \mc I_{P_{1,Y} \sqcup P_{2,Y}}(3L_{Y}-G))=0$. 
By the exact sequence $0 \to \mc I_{P_{1,Y} \sqcup P_{2,Y}}(3L_{Y}-G) \to \mc O(3L_{Y}-G) \to \bigoplus_{i=1}^{2} \mc O(3L_{Y}-G)|_{P_{i,Y}} \to 0$, we obtain $\Bs|3L_{Y}-G| \subset \bigcup_{i=1,2} \Bs |(3L_{Y}-G)|_{P_{i,Y}}|$, which is empty by Lemma~\ref{KG2-lem-EC}~(2). 
Thus $\Bs|3L_{Y}-G|=\emp$. 

To show the property of $\psi_{Y}$, let $C \subset Y$ be a curve with $(3L_{Y}-G).C=0$. 
Then $(2L_{Y}-G).C=(3L_{Y}-G).C-L_{Y}.C \leq 0$ since $L_{Y}$ is nef. 
If $(2L_{Y}-G).C=0$, then we have $G.C=0$, which is a contradiction to that $-G$ is relatively ample over $\P^{5}$. 
Thus $(2L_{Y}-G).C<0$, which implies $C \subset P_{1,Y} \sqcup P_{2,Y}$. 
On the other hand, since $\mc O_{P_{i,Y}}(3L_{Y}-G) \simeq \mc O_{P_{i,Y}}$, we get $\Exc(\psi_{Y})=P_{1,Y} \sqcup P_{2,Y}$. 
Hence $\psi_{Y}$ is a small birational morphism contracting only $P_{1,Y} \sqcup P_{2,Y}$. 
Since $-K_{Y} \sim 6L_{Y}-2G$, $\psi_{Y}$ is flopping. 
The last assertion follows from (2). 

(5) Since $-Q_{Y} \sim -2L_{Y}+G \sim_{\ol{Y}} L_{Y}$ is relatively ample over $\ol{L}$, we have $R^{1}{\psi_{Y}}_{\ast}\mc O(-Q_{Y})=0$ by the relative Kodaira vanishing theorem. 
Thus $\psi_{Q_{Y}}$ is a contraction. 
Since $Q_{Y}$ contains $P_{1,Y} \sqcup P_{2,Y}$, 
we obtain $\Exc(\psi_{Q_{Y}})=P_{1,Y} \sqcup P_{2,Y}$. 
The last assertion follows from (2). 

(6) For $i \in \{1,2\}$, let $P_{i,Y^{+}}$ be the flopped plane corresponding $P_{i,Y}$. 
By (2), the flop $\Psi \colon Y \dra Y^{+}$ is given by the composition of the blowing-up and the blowing-down 
$Y \overset{\t}{\gets} \Bl_{P_{1,Y} \sqcup P_{2,Y}}Y=:\wh{Y}=\Bl_{P_{1,Y^{+}} \sqcup P_{2,Y^{+}}} \overset{\t^{+}}{\to} Y^{+}$, 
where $\t$ (resp. $\t^{+}$) denotes the blowing-up of $Y$ (resp. $Y^{+}$) along $P_{1,Y} \sqcup P_{2,Y}$ (resp. $P_{1,Y^{+}} \sqcup P_{2,Y^{+}}$). 
By (2), the flip of $Q_{Y}$ is given by taking the proper transform. 

(7) This follows from the same argument as in the proof of (6). 
By (2), $\Psi|_{E_{S,Y}}$ contracts $P_{i,Y}$ to a smooth point. 
\end{proof}

\begin{proof}[Proof of Proposition~\ref{KG2-prop-stflop}]
By Lemma~\ref{KG2-lem-stflop}~(4), (5), and (6), we have the flop $\Phi \colon Y \dra Y^{+}$ and the flip $\Phi|_{Q_{Y}} \colon Q_{Y} \dra Q_{Y^{+}}$. 
Let $L_{Y^{+}}$ and $G_{Y^{+}}$ be the proper transform on $Y^{+}$ of $L_{Y}$ and $G$ respectively. 
Let $Q_{a,Y^{+}},Q_{b,Y^{+}},Q_{c,Y^{+}} \in |2L_{Y^{+}}-G_{Y^{+}}|$ be the proper transform of them on $Y^{+}$. 
We already know that  $Q_{a,Y} \cap Q_{b,Y} \cap Q_{c,Y}=P_{1,Y} \sqcup P_{2,Y}$ by Lemma~\ref{KG2-lem-stflop}~(3) and 
$Q_{a,Y}+Q_{b,Y}+Q_{c,Y}$ is simply normal crossing. 
Thus $Q_{a,Y^{+}} \cap Q_{b,Y^{+}} \cap Q_{c,Y^{+}} \subsetneqq P_{1,Y^{+}} \sqcup P_{2,Y^{+}}$ and hence we have $Q_{a,Y^{+}} \cap Q_{b,Y^{+}} \cap Q_{c,Y^{+}}=\emp$, which implies $|2L_{Y^{+}}-G_{Y^{+}}|$ is base point free. 
Then $|2L_{Y^{+}}-G_{Y^{+}}|$ gives a morphism $\vp_{Y^{+}} \colon Y^{+} \to \P^{2}$. 
A fiber of $\vp_{Y^{+}}$ is $Q_{a,Y^{+}} \cap Q_{b,Y^{+}}$, which is the proper transform $E_{S_{ab},Y^{+}}$. 
By this argument, $\vp_{Y^{+}}$ has connected fibers. 
Since $\rho(Y^{+})=2$, $\vp_{Y^{+}}$ is an extremal contraction. 
Therefore, every smooth fiber of $\vp_{Y^{+}}$ is a Fano manifold. 
Since $-K_{Y^{+}}=6L_{Y^{+}}-2G_{Y^{+}}=2(Q_{Y^{+}}+L_{Y^{+}})$, the Fano index of every smooth fiber is divided by $2$. 
Let us consider the fiber $E_{S_{ab,Y^{+}}}$. 
By Lemma~\ref{KG2-lem-ES} and Lemma~\ref{KG2-lem-stflop}~(1) and (7), we have $(-K_{E_{S_{ab,Y^{+}}}})^{3}=(-K_{E_{S_{ab,Y}}})^{3}+16=48$ and $\rho(E_{S_{ab,Y^{+}}})=\rho(E_{S_{ab,Y}})-2=\rho(E_{S_{ab}})-2=\rho(S_{ab})-1=3$. 
By Fujita's classification~\cite{Fujita80}, we conclude that $E_{S_{ab,Y^{+}}}=(\P^{1})^{3}$. 
Hence all smooth fibers are diffeomorphic to $(\P^{1})^{3}$ and hence isomorphic again by \cite{Fujita80}. 
By construction, $Q_{Y^{+}} \subset Y^{+}$ is the pull-back of a line of $\P^{2}$ by $\vp_{Y^{+}}$. 
Since $\rho(Q_{Y^{+}})=2$, the morphism 
$\vp_{Q_{Y^{+}}}:=\vp_{Y^{+}}|_{Q_{Y^{+}}} \colon Q_{Y^{+}} \to \P^{1}$ 
is also an extremal contraction whose smooth fibers are $(\P^{1})^{3}$. 
We complete the proof. 
\end{proof}

\begin{rem}\label{KG2-rem-cubic}
Let $\wt{V} \in |3H_{Y}-G_{Y}|$ be a general element and $V=\s(\wt{V})$. 
Then by \cite[Theorems~2 and 6]{AHTVA16}, $V$ is a smooth cubic fourfold containing $R$ and $\wt{V}=\Bl_{R}V$. 
By Lemma~\ref{KG2-lem-stflop}, we may assume that $\wt{V}$ is away from the flipping locus of $Y \dra Y^{+}$. 
Hence we have a morphism $\vp_{Y^{+}}|_{\wt{V}} \colon \wt{V} \to \P^{2}$, whose general fiber is isomorphic to a sextic del Pezzo surface. 
Addington, Hassett, Tschinkel, and V\'{a}rilly-Alvarado studied the cubic fourfold $V$ by using this del Pezzo fibration structure over $\P^{2}$ \cite{AHTVA16}. 
\end{rem}

\subsection{$K(G_{2})$ and the sextic del Pezzo threefold $\P(T_{\P^{2}})$}\label{KG2-subsec-cayley}

The main purpose of this subsection is to see some relationships between $K(G_{2})$ and the sextic del Pezzo threefold $\P(T_{\P^{2}})$, which is mostly established by Ottaviani \cite{Ott90}. 

Let $\Q^{5}$ be a $5$-dimensional non-singular hyperquadric in $\P^{6}$. 
We fix a Cayley bundle $\ms C$ on $\Q^{5}$ \cite{Ott90}. 
This vector bundle can be defined as a stable vector bundle of rank $2$ with Chern classes $(c_{1}(\ms C),c_{2}(\ms C))=(-\mc O_{\Q^{5}}(1),\mc O_{\Q^{5}}(1)^{2})$, where $\mc O_{\Q^{5}}(1)$ denotes the polarization with respect to the embedding $\Q^{5} \hra \P^{6}$. 
The following Ottaviani's theorem enables us to treat $K(G_{2})$ as the parametrizing scheme of the special lines on $\Q^{5}$. 

\begin{thm}[{\cite[(iii) of Introduction]{Ott90}}]\label{KG2-thm-Ott}
Let $q \colon M:=\P(\ms C(2)) \to \Q^{5}$ be the projectivization of the vector bundle $\ms C(2)$ on $\Q^{5}$. 
The complete linear system of the tautological bundle $\mc O_{\P(\ms C(2))}(1)$ gives a $\P^{1}$-bundle structure $p \colon \P(\ms C(2)) \to K(G_{2})$. 
Here, $K(G_{2})$ denotes the $5$-dimensional contact homogeneous manifold of type $G_{2}$. 
\end{thm} 

By Theorem~\ref{KG2-thm-Ott}, we have $H^{\ast}(\Q^{5},\Z) \simeq H^{\ast}(K(G_{2}),\Z)$. 
Let $H_{K(G_{2})}$ be an ample divisor on $K(G_{2})$ such that $\Pic(K(G_{2}))=\Z \cdot [\mc O(H_{K(G_{2})})]$. 
Set 
\[\mc E:=p_{\ast}q^{\ast}\mc O_{\Q^{5}}(1).\]
Then $\mc E$ is a globally generated rank $2$ vector bundle on $K(G_{2})$ and $\mc O_{\P(\mc E)}(1)=q^{\ast}\mc O_{\Q^{5}}(1)$. 
Moreover, we have $\P_{K(G_{2})}(\mc E)=M \simeq \P_{\Q^{5}}(\ms C(2))$: 
\begin{align}\label{KG2-dia}
\xymatrix{
&\ar[ld]_{p} \P_{K(G_{2})}(\mc E)=M=\P_{\Q^{5}}(\ms C(2))\ar[rd]^{q}& \\
K(G_{2})&&\Q^{5}.
}
\end{align}

From now on, we set divisors $H,L$ on $M$ as follows:
\[H=p^{\ast}H_{K(G_{2})} \text{ and } L=q^{\ast}(\text{a hyperplane section of } \Q^{5}). \]
It follows from Theorem~\ref{KG2-thm-Ott} that $-K_{M} \sim 2H+2L$, $\mc O_{M}(L)=\mc O_{\P(\mc E)}(1)$ and $\mc O_{M}(H)=\mc O_{\P(\ms C(2))}(1)$. 
Let $W \in |L|$ be a general element. 
Let $s \in H^{0}(K(G_{2}),\mc E)$ be the section corresponding to $W$. 
Then $q(W) \subset \Q^{5}$ is a non-singular quadric fourfold $\Q^{4}$ and $W \simeq \P_{\Q^{4}}(\ms C(2)|_{\Q^{4}})$. 
The morphism $p|_{W} \colon W \to K(G_{2})$ is the blowing-up of $K(G_{2})$ along the zero scheme 
\begin{align}\label{def-S-KG2}
\Sigma:=(s=0) \subset K(G_{2}). 
\end{align}
Since $W$ is a general member and $\mc E$ is globally generated, we may assume that $\Sigma$ is smooth. 

The main proposition of this subsection is the following. 
\begin{prop}\label{KG2-prop-TP2}
$\Sigma$ is isomorphic to $\P(T_{\P^{2}})$, which is a hyperplane section of $(\P^{2})^{2}$ with respect to the Segre embedding. 
Moreover, it holds that 
\[\mc E|_{\P(T_{\P^{2}})} \simeq \mc O(A_{1}) \oplus \mc O(A_{2}),\]
where $A_{i}$ denotes the pull-back of a line on $\P^{2}$ under the restriction to $\P(T_{\P^{2}})$ of the $i$-th projection $(\P^{2})^{2} \to \P^{2}$. 
\end{prop}

For the proof, we use the following well-known lemma, which can be deduce from the diagram~(\ref{KG2-dia}). 
\begin{lem}[\cite{Ott90}]\label{KG2-lem-calc}
It holds that 
$H_{K(G_{2})}^{5}=18$, 
$c_{1}(\mc E)=H_{K(G_{2})}$, 
$c_{2}(\mc E)=\frac{1}{3}H_{K(G_{2})}^{2}$, and 
 $-K_{K(G_{2})}=3H_{K(G_{2})}$. 
\end{lem}

\begin{proof}[Proof of Proposition~\ref{KG2-prop-TP2}]
Since $W \simeq \P_{\Q^{4}}(\ms C(2)|_{\Q^{4}})$, 
we have $b_{4}(W)=2$ and $b_{2}(W)=2$. 
Since $\Bl_{\Sigma}K(G_{2}) \simeq W$, it follows from the blowing-up formula that $b_{0}(\Sigma)=1$, which means that $\Sigma$ is connected, and $b_{2}(\Sigma)=2$. 
By the adjunction formula and Lemma~\ref{KG2-lem-calc}, 
we obtain $-K_{\Sigma}=2H_{K(G_{2})}|_{\Sigma}$ and $(H_{K(G_{2})}|_{\Sigma})^{3}=c_{2}(\mc E).H_{K(G_{2})}^{3}=6$. 
Hence $\Sigma$ is a smooth del Pezzo threefold of degree $6$ of Picard rank $2$. 
Then $\Sigma$ is isomorphic to $\P(T_{\P^{2}})$ by \cite{Fujita80}. 

Let $l_{i}$ be a fiber of $\pi_{i} \colon \P(T_{\P^{2}}) \to \P^{2}$ for each $i$. 
Then we have $[A_{i}^{2}]=[l_{i}]$ and $[A_{1}.A_{2}]=[l_{1}+l_{2}]$. 
By Lemma~\ref{KG2-lem-calc}, we have $\det \mc E|_{\Sigma}=A_{1}+A_{2}$ and $c_{2}(\mc E|_{\Sigma})=[l_{1}+l_{2}]$. 
Noting that $\mc E|_{l_{1}}=\mc O \oplus \mc O(1)$, 
the sheaf ${\pi_{1}}_{\ast}(\mc E|_{\Sigma}(-A_{2}))$ is a line bundle on $\P^{2}$. 
Then the cokernel $\mc L$ of the injection $\pi_{1}^{\ast}{\pi_{1}}_{\ast}(\mc E|_{\Sigma}(-A_{2})) \to \mc E|_{\Sigma}(-A_{2})$ is locally free. 
Take the integer $d \in \Z$ such that ${\pi_{1}}_{\ast}(\mc E|_{\Sigma}(-A_{2}))=\mc O(d)$. 
Since $c_{1}(\mc E)=A_{1}+A_{2}$, we have $\mc L=\mc O((1-d)A_{1}-A_{2})$. 
Hence we obtain the following exact sequence: 
\[0 \to \mc O(dA_{1}+A_{2}) \to \mc E|_{\Sigma} \to \mc O((1-d)A_{1}) \to 0.\]
Since $\mc E|_{\Sigma}$ is globally generated, we have $d \in \{0,1\}$. 
Noting that $c_{2}(\mc E|_{\Sigma})=[l_{1}+l_{2}]$, we have $d=0$ and hence the above exact sequence must split. 
We complete the proof.
\end{proof}


Letting $E=\Exc(p|_{W})$, we obtain the following diagram as a summary: 
\begin{align}\label{KG2-dia-W}
\xymatrixrowsep{5mm}
\xymatrixcolsep{5mm}
\xymatrix{
\P_{\P(T_{\P^{2}})}(\mc O(A_{1}) \oplus \mc O(A_{2})) \ar@{}[r]|{\hspace{40pt}=}&  E\ar[ld]\ar@{}[r]|{\subset\hspace{40pt}}&\ar[ld]_{p|_{W}} \Bl_{\Sigma}K(G_{2})=W \ar[rd]^{q|_{W}}&\ar@{}[l]|{=}\P_{\Q^{4}}(\ms C(2)|_{\Q^{4}}) \\
\P(T_{\P^{2}}) \simeq \Sigma& \ar@{}[l]|{\subset} K(G_{2})&&\Q^{4}, 
}
\end{align}
where $E$ is naturally isomorphic to $\P_{\Sigma}(\mc E|_{\Sigma})$ since $\Sigma$ is a zero set of $\mc E$.
Note that this $E$ already appeared in Lemma~\ref{KG2-lem-E}. 
Moreover, from our construction, we have $H|_{W}-E=L|_{W}$ and $-K_{W}=2H|_{W}-E=H|_{W}+L|_{W}$. 
In particular, $\mc O(L)|_{E}$ is nothing but the tautological line bundle of $\P_{\Sigma}(\mc E|_{\Sigma})$. 
By Lemma~\ref{KG2-lem-E}, the morphism $q|_{E} \colon E \to \Q^{4}$ is the blowing-up of $\Q^{4}$ along a disjoint union of two planes $P_{1}$ and $P_{2}$ on $\Q^{4}$. 

\subsection{Construction of an example of No.10}\label{KG2-subsec-conclu}

Let $V^{4}_{10} \subset K(G_{2})$ be a general hyperplane section and $S:=\Sigma \cap V^{4}_{10}$. 
Let $E_{S}$ be the exceptional divisor of $N:=\Bl_{S}V^{4}_{10} \to V^{4}_{10}$. 
Note that by Proposition~\ref{KG2-prop-TP2}, $S$ and $E_{S}$ coincide with what already appeared in Lemma~\ref{KG2-lem-ES}. 
Now $N$ is a member of $|\mc O_{\P_{\Q^{4}}(\ms C(2)|_{\Q^{4}})}(1)|$. 
Let $t \in H^{0}(\Q^{4},\ms C(2)|_{\Q^{4}})$ be a corresponding section and  $T=(t=0)$ the zero scheme of $t$. 
Then $N$ is also the blowing-up of $\Q^{4}$ along $T$. 
Let $p_{N} \colon N \to V^{4}_{10}$ and $q_{N} \colon N \to \Q^{4}$ be the blowing-downs. 
Then we obtain the following diagram: 
\begin{align}\label{KG2-dia-V410}
\xymatrix{
&\P_{S}(\mc O(A_{1})|_{S} \oplus \mc O(A_{2})|_{S}) \simeq E_{S}\ar[ld]&\ar@{}[l]|{\qquad \subset} \ar[ld]_{p_{N}} \Bl_{S}V^{4}_{10}=N=\Bl_{T}\Q^{4} \ar[rd]^{q_{N}}& \\
S& \ar@{}[l]|{\subset} V^{4}_{10}&&\Q^{4}.
}
\end{align}

Let $V^{3}_{10} \subset V^{4}_{10}$ be a general hyperplane section and set $C:=S \cap V^{3}_{10}$. 
Then the ladder $C \subset S \subset \Sigma$ satisfies Setting~\ref{KG2-sett-Y}. 
Let $E_{C}$ be the exceptional divisor of $\Bl_{C}V^{3}_{10} \to V^{3}_{10}$ and $R:=q(E_{C})$. 
Note that $C$, $E_{C}$, and $R$ coincide what appeared in Setting~\ref{KG2-sett-Y}. 

Consider the blowing-up $\Bl_{R}\Q^{4}$, which is nothing but $Q_{Y}$ as in Setting~\ref{KG2-sett-Y}. 
Then by Proposition~\ref{KG2-prop-stflop}, the flop $Q_{Y^{+}}$ of $Q_{Y}=\Bl_{R}\Q^{4}$ has a $(\P^{1})^{3}$-fibration $Q_{Y^{+}} \to \P^{1}$. 
Combining the diagram (\ref{KG2-dia-stflop}) with the diagram (\ref{KG2-dia-V410}), we obtain the following diagram: 
\[
\xymatrixrowsep{5mm}
\xymatrixcolsep{5mm}
\xymatrix{
\Bl_{S}V^{4}_{10}\ar@{}[r]|{=}&N \ar[ld]_{p_{N}} \ar[rd]^{q_{N}}&\ar@{}[l]|{=}\Bl_{T}\Q^{4}&\ar[ld] Q_{Y}:=\Bl_{R}\Q^{4} \ar@{-->}[r]^{\Psi}&Q_{Y^{+}} \ar[rd]&\\
V^{4}_{10}&E_{S}\ar[rd]_{q|_{E_{S}}} \ar[ld] \ar@{^{(}->}[u]&\Q^{4}&\ar[ld]^{\s|_{E_{S,Y}}}E_{S,Y}\ar[r]^{\Psi|_{E_{S,Y}} \quad}  \ar@{^{(}->}[u] &E_{S,Y^{+}} \simeq (\P^{1})^{3} \ar[rd]  \ar@{^{(}->}[u] &\P^{1} \\
S \ar@{^{(}->}[u]&&\ol{E_{S}} \ar@{^{(}->}[u]&&&pt.  \ar@{^{(}->}[u]
}\]
\begin{claim}\label{KG2-claim-calcN}
It holds that $p_{N}^{\ast}H \sim q_{N}^{\ast}\mc O_{\Q^{4}}(3)-\Exc(q_{N})$. 
\end{claim}
\begin{proof}
Recall that $\Bl_{T}\Q^{4}$ is a member of $|\mc O_{\P_{\Q^{4}}(\ms C(2))}(1)|$ and the embedding $\Bl_{T}\Q^{4} \hra \P_{\Q^{4}}(\ms C(2))$ is determined by the surjection $(\ms C(2))^{\vee} \epm \mc I_{T}$. 
Then the claim follows from the fact that $\det (\ms C(2))=\mc O_{\Q^{5}}(3)$. 
\end{proof}
Now $p|_{X} \colon X=\Bl_{C}V^{3}_{10} \to V^{3}_{10}$ is the blowing-up along $C$. 
Let $\ol{X}$ and $X^{+}$ be its proper transforms on $\Q^{4}$ and $Q_{Y}$ respectively. 
Then our construction of an example of No.10 is established by the following proposition. 

\begin{prop}\label{KG2-prop-conclu}
The following assertions hold. 
\begin{enumerate}
\item The threefold $X$ is weak Fano and $q|_{X} \colon X \to \ol{X}$ is a flopping contraction whose flopping locus is the disjoint union of $18$ smooth rational curves with normal bundles $\mc O_{\P^{1}}(-1)^{2}$. 
\item The proper transform $X^{+} \subset Q_{Y}$ is the flop of $X \to \ol{X}$. 
\item The flopped threefold $X^{+}$ is away from the flopping locus of $\Psi \colon Y \dra Y^{+}$. 
Moreover, the restriction $\vp_{X^{+}}:=\vp_{Y^{+}}|_{X^{+}} \colon X^{+} \to \P^{1}$ is a sextic del Pezzo fibration. 
\item Each of the $18$ flopped curves is a section of $\vp_{X^{+}}$. 
\end{enumerate}
In particular, the weak Fano sextic del Pezzo fibration $\vp_{X^{+}} \colon X^{+} \to \P^{1}$ is an example of No.10. 
\end{prop}
\begin{proof}
(1) 
Note that $X$ is a member of $|p_{N}^{\ast}H|$ on $N$. 
The adjunction formula gives $-K_{X} \sim q|_{X}^{\ast}\mc O_{\Q^{4}}(1)|_{\ol{X}}$, which implies that $X$ is weak Fano and $q|_{X}$ is crepant birational.
Since $X$ is general, $q(\Exc(q|_{X}))$ coincides with the zero scheme of a global section of $\ms C(2)|_{T}$, which is a finite set. 
Hence $q|_{X}$ is flopping. 
Since $X$ is general, we may assume that $\Theta:=\Exc(q_{N}) \cap X$ is a smooth surface.
Hence $q_{N}|_{\Theta} \colon \Theta \to T$ is a birational morphism between smooth projective surfaces. 
By Claim~\ref{KG2-claim-calcN}, we obtain 
\begin{align}\label{KG2-eq-key}
\Theta+E_{C} \sim -2K_{X}. 
\end{align}
Then for any $q|_{X}$-flopping curve $l$, 
$\Theta$ contains $l$ as a $(-1)$-curve and it holds that $\Theta.l=-1$. 
Hence we obtain $\mc N_{l/X} \simeq \mc O_{\P^{1}}(-1)^{2}$, which means that $q|_{X}$ is an Atiyah's flopping contraction. 
The number of the flopping curves is $c_{2}(\ms C(2)|_{T})=c_{2}(\ms C(2)|_{\Q^{4}})=18$ since $c_{2}(\ms C(2)|_{\Q^{4}})^{2}=3\mc O_{\Q^{4}}(1)^{2}$. 
We complete the proof of (1). 

(2) Since it holds that $X^{+}=\Bl_{R}\ol{X}$, we can see that $X^{+} \to \ol{X}$ is small by the same argument as in Lemma~\ref{KG2-lem-stflop}~(1). 
While $-G|_{X^{+}}$ is relatively ample over $\ol{X}$ from the construction, 
the proper transform of $G|_{X^{+}}$ on $X$ is $E_{C}$, which is relatively ample over $\ol{X}$ by (\ref{KG2-eq-key}) and the proof of (1). 
Thus $X^{+}$ is the flop of $X \to \ol{X}$. 

(3) By Lemma~\ref{KG2-lem-stflop}~(4), we may assume that $X^{+}$ is away from the flopping locus of $\Psi$. 
As a submanifold of $Y^{+}$, $X^{+}$ is a complete intersection of $Q_{Y^{+}} \in |2L_{Y^{+}}-G_{Y^{+}}|$ and a general member of $|3L_{Y^{+}}-G_{Y^{+}}|$. 
Hence a general fiber of $\vp_{X^{+}} \colon X^{+} \to \P^{1}$ is a sextic del Pezzo surface. 
The Picard rank of $X^{+}$ is $2$ since $\rho(X^{+})=\rho(X)=2$. 

(4) As in the proof of (1), for each flopping curve $l$, we have $\Theta.l=-1$ and $\Theta$ contains $l$ as a $(-1)$-curve. 
Letting $\Theta^{+}$ be its proper transform on $X^{+}$, we have $\Theta^{+}.l^{+}=1$, where $l^{+}$ is the flopped curve corresponding to $l$. 
Since $G_{Y^{+}}|_{X^{+}}$ is the proper transform of $E_{C}$ on $X^{+}$, 
$\Theta^{+} \sim (2L_{Y^{+}}-G_{Y^{+}})|_{X^{+}}$ is a fiber of $\vp_{X^{+}}$. 
Hence $l^{+}$ is a section of $\vp_{X^{+}}$. 
We complete the proof. 
\end{proof}

Now we complete the proof of Theorem~\ref{mainthm-wf} for No.10. 
\hfill$\square$

\begin{rem}
From our construction, $X^{+}$ is a complete intersection of $Q_{Y^{+}}$ and $\wt{V}$ in Remark~\ref{KG2-rem-cubic} on $Y^{+}$. 
The anti-canonical model $\ol{X}$ of $X^{+}$ is a complete intersection of the hyperquadric $Q$ and the hypercubic $V$ and has $18$ nodes as its singularities. 
\end{rem}


\section{Construction of an example of No.14}\label{LG-section}

In this section, we will prove the existence of a weak Fano threefold with a sextic del Pezzo fibration of No.14. 
As mentioned in Subsection~\ref{wf-subsec-sketch},  we will find a sextic del Pezzo threefold $(\P^{1})^{3}$ on $\LG(3,6)$ and cut them in order to obtain a desired sextic elliptic curve $C$ on a prime Fano threefold $V^{3}_{9}$ of genus $9$. 

\subsection{Sketch for our construction}\label{LG-subsec-sketch}


Let $V$ be a $6$-dimensional vector space and fix a non-degenerate skew form $\a \in \bigwedge^{2}V$. 
The Lagrangian Grassmannian $\LG(3,6)=\LG_{\a}(3,V)$ with respect to $\a$ is defined to be the parameter space of the Lagrangian $3$-spaces for $\a$ on $V^{\vee}$. 
We additionally take another general skew form $\b$. 
Then the Lagrangian Grassmannian with respect to $\a$ and $\b$, say $\LG_{\a,\b}(3,V)$, is isomorphic to $(\P^{1})^{3}$ (cf. Lemma~\ref{LG-lem-P1n}). 
By taking a general codimension $2$-linear section of the pair $(\LG_{\a}(3,V), \LG_{\a,\b}(3,V))$, 
we obtain a pair $(V^{4}_{9},C)$, where $C$ is a sextic elliptic curve on a Mukai fourfold $V^{4}_{9}$ of genus $9$. 
Then by taking a general hyperplane section of $V^{4}_{9}$ containing $C$, 
we will obtain a smooth prime Fano threefold $V^{3}_{9}$ of genus $9$ containing $C$ (=Proposition~\ref{LG-prop-conclu}~(1)). 

We will show that the flop of the blowing-up $X:=\Bl_{C}V^{3}_{9}$ has a sextic del Pezzo fibration structure. 
For this, we analyze the blowing-up $\mu \colon \wt{\LG_{\a}(3,V)} \to \LG_{\a}(3,V)$ along $\LG_{\a,\b}(3,V)$. 
In Subsection~\ref{LG-subsec-blowup}, we will show that $\wt{\LG_{\a}(3,V)}$ is embedded into the isotropic flag variety $\LF_{\a}(1,3;V)$ with respect to $\a$ and $E:=\Exc(\mu)$ is the isotropic flag variety $\LF_{\a,\b}(1,3;V)$ with respect to $\a$ and $\b$. 
Thus $\wt{\LG_{\a}(3,V)}$ has a natural morphism $f \colon \wt{\LG_{\a}(3,V)} \to \P(V)$. 


By taking the base change of the $\P^{2}$-bundle structure 
$E=\LF_{\a,\b}(1,3;V) \to \LG_{\a,\b}(3,V)$ by the inclusion $C \hra \LG_{\a,\b}(3,V)$, 
we obtain a $\P^{2}$-bundle $\mu^{-1}(C) \to C$. 
Then we set $R:=f(\mu^{-1}(C)) \subset \P(V)$ as in Section~\ref{KG2-section}. 
We will study an extraction of $\P^{5}=\P(V)$ along $R$ in order to find a fibration structure on the flop of $X$. 

The most different point from Section~\ref{KG2-section} is that $\mu^{-1}(C) \to R$ is not isomorphic. 
Indeed, $R$ is non-normal and singular along disjoint three lines $l_{1}$, $l_{2}$, and $l_{3}$ (cf. Proposition~\ref{LG-prop-R}). 
We, however, will construct the divisorial extraction $Z \to \P(V)$ along $R$ and show that $Z$ is a smooth fivefold (=Proposition~\ref{LG-prop-Z}). 
Moreover, we will show that $Z$ is flipped into a smooth fivefold $Z^{+}$ and $Z^{+}$ has a $(\P^{2})^{2}$-fibration structure $\vp_{Z^{+}} \colon Z^{+} \to \P^{1}$ (Proposition~\ref{LG-prop-stflip}). 
In summary, we obtain the following diagram: 
\[
\xymatrixrowsep{5mm}
\xymatrixcolsep{5mm}
\xymatrix{
&\Bl_{\LG_{\a,\b}(3,V)}\LG_{\a}(3,V) \ar@{}[r]|{\qquad =}&\wt{\LG_{\a}(3,V)}\ar[ld]_{\mu} \ar[rd]^{f}&Z\ar[d]^{g} \ar@{-->}[r]^{\Psi}&Z^{+}\ar[d]^{\vp_{Z^{+}}}& \\
\LG_{\a,\b}(3,V) \ar@{}[r]|{\subset}&\LG_{\a}(3,V) &\mu^{-1}(C) \ar[rd] \ar@{}[u]|{\rotatebox{90}{$\subset$}}&\P(V)&\P^{1} \\
C\ar@{}[r]|{\subset}\ar@{}[u]|{\rotatebox{90}{$\subset$}} \ar@{}[ru]|{\Box}&V^{4}_{9} \ar@{}[u]|{\rotatebox{90}{$\subset$}}&&R\ar@{}[u]|{\rotatebox{90}{$\subset$}}& \\
C\ar@{}[r]|{\subset}\ar@{=}[u] &V^{3}_{9}. \ar@{}[u]|{\rotatebox{90}{$\subset$}}&&& 
}\]
Let $A:=\Bl_{C}V^{4}_{9} \subset \wt{\LG_{\a}(3,V)}$, $\ol{A}:=f(A) \subset \P(V)$ and $A' \subset Z$ its proper transform. 
In Lemma~\ref{LG-lem-A}, we will prove that $f|_{A} \colon A \to \ol{A}$ is flopping and that $A'$ is the flop of $A$. 
By taking $V^{4}_{9}$ generally, we will also prove that the fourfold $A'$ is smooth and away from the flipping locus of $\Psi$. 
Hence we obtain the following diagram: 
\[\xymatrix{
&\Bl_{C}V^{4}_{9}=A\ar[ld]_{\mu|_{A}} \ar[rd]^{f|_{A}}&&A' \ar[ld]_{g|_{A'}} \ar@{=}[r]^{\Psi|_{A'}}&A'\ar[rd]^{\vp_{Z^{+}}|_{A'}}& \\
V^{4}_{9} &&\ol{A}&&&\P^{1}. 
}\]
Then the morphism $\vp_{Z^{+}}|_{A'}$ is a fibration with general fiber $\P(T_{\P^{2}})$. 
Let $X:=\Bl_{C}V^{3}_{9}$, $\ol{X}:=f(X)$, and $X' \subset A'$ its proper transfrom. 
We will see that $\ol{X}$ is a hyperplane section of $\ol{A}$ (cf. Lemma~\ref{LG-lem-relonwtLG}). 
Hence $f|_{X} \colon X \to \ol{X}$ is also flopping and $X \dra X'$ is the flop (=Proposition~\ref{LG-prop-conclu}~(3)). 
Finally, we obtain the following diagram: 
\[\xymatrix{
&\Bl_{C}V^{3}_{9}=X\ar[ld]_{\mu|_{X}} \ar[rd]^{f|_{X}}&&X' \ar[ld]_{g|_{X'}} \ar@{=}[r]^{\Psi|_{X'}}&X'\ar[rd]^{\vp_{Z^{+}}|_{X'}}& \\
V^{3}_{9} &&\ol{X}&&&\P^{1}.
}\]
The morphism $\vp_{Z^{+}}|_{X'} \colon X' \to \P^{1}$ is a sextic del Pezzo fibration and hence $X'$ is an example of No.14. 

\subsection{Organization of this section}

We organize this section as follows. 

Subsections~\ref{LG-subsec-2skews} and \ref{LG-subsec-onlya} are spent for preliminaries about isotropic flag varieties. 
In Subsection~\ref{LG-subsec-2skews}, we will define the isotropic flag varieties for finitely many skew forms. 
We will also study the isotropic flag variety $\LF_{\a,\b}(1,n;2n)$ for two skew forms $\a$ and $\b$. 

From Subsection~\ref{LG-subsec-onlya}, we fix a $6$-dimensional vector space $V \simeq \C^{6}$ and a non-degenerate skew form $\a \in \bigwedge^{2}V$. 
The main purpose of Subsection~\ref{LG-subsec-onlya} is to study the isotropic flag varieties for $V$ with respect to $\a$ in detail. 

In Subsection~\ref{LG-subsec-blowup}, we take another general skew form $\b \in \bigwedge^{2}V$. 
The main purpose is to prove Proposition~\ref{LG-prop-blowupLG}, which states that the blowing-up $\LG_{\a}(3,V)$ along $\LG_{\a,\b}(3,V)$ is contained in $\LF_{\a}(1,3;V)$ and the exceptional divisor coincides with $\LF_{\a,\b}(1,3;V)$. 

In Subsection~\ref{LG-subsec-R}, we will take an codimension $2$ section $C$ of $\LG_{\a,\b}(1,3;V)=(\P^{1})^{3}$ and study the threefold $R:=f(\mu^{-1}(C)) \subset \P(V)$. 
As mentioned in Subsection~\ref{LG-subsec-sketch}, $R$ is non-normal along the three lines $l_{1},l_{2},l_{3}$. 
We will prove that $\wt{R}:=\Bl_{l_{1},l_{2},l_{3}}R$ is a non-singular threefold in Proposition~\ref{LG-prop-wtR}. 

In Subsection~\ref{LG-subsec-extraction}, we will construct the divisorial extraction $Z \to \P(V)$ along $R$ and show that $Z$ is smooth. 
Let $\wt{\P(V)}:=\Bl_{l_{1},l_{2},l_{3}}\P(V)$ and consider $\Bl_{\wt{R}}\wt{\P(V)}$. 
Then we will construct a birational morphism $\Bl_{\wt{R}}\wt{\P(V)} \to Z$ over $\P(V)$ in Proposition~\ref{LG-prop-Z}. 
This morphism contracts the proper transforms of the exceptional divisors of $\wt{\P(V)} \to \P(V)$. 
Then the birational contraction $Z \to \P(V)$ is nothing but the desired divisorial extraction along $R$. 

In Subsection~\ref{LG-subsec-flip}, we will show that $Z$ has a flip $Z \dra Z^{+}$. 
Moreover, we will show that $Z^{+}$ is smooth and has a $(\P^{2})^{2}$-fibration structure $\vp_{Z^{+}} \colon Z^{+} \to \P^{1}$.  

In Subsection~\ref{LG-subsec-conclu}, we construct an example of No.14 by taking the proper transforms as mentioned in Subsection~\ref{LG-subsec-sketch}.

\subsection{Isotropic flag varieties for two skew forms}\label{LG-subsec-2skews}
We start with an even-dimensional vector space $V=\C^{2n}$. 

\begin{defi}\label{LG-defi-IF}
Let $\a_{1},\ldots,\a_{m} \in \bigwedge^{2}V$ be skew forms on $V^{\vee}$. 
For an increasing sequence of positive integers $k_{1}<\cdots<k_{l}$ with $k_{l} \leq \dim V=2n$, 
we define 
\[\LF_{\a_{1},\ldots,\a_{m}}(k_{1},\ldots,k_{l};V)\]
as \emph{the isotropic flag variety for skew forms $\a_{1},\ldots,\a_{m}$}, which is namely given by 
\[\{[\C^{k_{1}} \subset \cdots \subset \C^{k_{l}} \subset V^{\vee}] \mid \forall i \in \{1,\ldots,m\}, \forall v, w \in \C^{k_{l}}, \a_{i}(v,w)=0\}. \]
For a subsequence $(k_{i_{1}},\ldots,k_{i_{j}})$ of $(k_{1},\ldots,k_{l})$, 
let $f^{k_{1},\ldots,k_{l}}_{k_{i_{1}},\ldots,k_{i_{j}}}$ denote the natural morphism  
\[f^{k_{1},\ldots,k_{l}}_{k_{i_{1}},\ldots,k_{i_{j}}} \colon \LF_{\a_{1},\ldots,\a_{m}}(k_{1},\ldots,k_{l};V) \to \LF_{\a_{1},\ldots,\a_{m}}(k_{i_{1}},\ldots,k_{i_{j}};V)\]
which is given by forgetting suitable vector spaces of flags. 
If $l=1$, we set 
\[\LG_{\a_{1},\ldots,\a_{m}}(k_{1},V):=\LF_{\a_{1},\ldots,\a_{m}}(k_{1};V)\]
and call it \emph{the isotropic Grassmannian for skew forms $\a_{1},\ldots,\a_{m}$}. 

On $\LG_{\a_{1},\ldots,\a_{m}}(k;V)$, 
the universal quotient bundle of rank $k$ and universal subbundle of rank $\dim V-k$ are denoted by 
\[\mc Q_{\LG_{\a_{1},\ldots,\a_{m}}(k;V)} \text{ and } \mc S_{\LG_{\a_{1},\ldots,\a_{m}}(k;V)}\]
respectively. 

We call $\LG_{\a_{1},\ldots,\a_{m}}(n;V)$ the \emph{Lagrangian Grassmannian for $\a_{1},\ldots,\a_{m}$}. 
If $m=1$ and $\a_{1}$ is non-degenerate, then $\LG_{\a_{1}}(n,V)$ is nothing but the \emph{Lagrangian Grassmannian} $\LG(n,2n)$ in the ordinary sense. 
\end{defi}

\begin{rem}\label{LG-rem-basechange}
In the above setting, the following diagram is cartesian: 
\[\xymatrix{
\LF_{\a_{1},\ldots,\a_{m}}(k_{1},\ldots,k_{l};V) \ar@{^{(}->}[r] \ar[d]_{f^{k_{1},\ldots,k_{l}}_{k_{l}}} \ar@{}[rd]|{\Box} & \Fl(k_{1},\ldots,k_{l};V) \ar[d]^{f^{k_{1},\ldots,k_{l}}_{k_{l}}} \\
\LG_{\a_{1},\ldots,\a_{m}}(k_{l};V) \ar@{^{(}->}[r] & \Gr(k_{l},V).
}\]
Note that $\LG_{\a_{1},\ldots,\a_{m}}(k_{l};V)$ is the zero set of a global section $(\a_{1},\ldots,\a_{m}) \in \bigoplus_{i=1}^{m} \bigwedge^{2}\mc Q_{\Gr(k_{l},V)}$, where $\mc Q_{\Gr(k_{l},V)}$ is the rank $k_{l}$ universal quotient bundle on the Grassmannian $\Gr(k_{l},V)$. 
Therefore, $\LF_{\a_{1},\ldots,\a_{m}}(k_{1},\ldots,k_{l};V)$ is smooth for general $\a_{1},\ldots,\a_{m}$ since $\bigwedge^{2}\mc Q_{\Gr(k_{l},V)}$ is globally generated. 
\end{rem}

\begin{defi}\label{defi-ortho}
Let $\a \in \bigwedge^{2}V$ be a non-degenerate skew form on $V^{\vee}$. 
We may regard $\a$ as an isomorphism $\a \colon V^{\vee} \to V$. 
Then $\a$ gives a bilinear form on $V^{\vee}$ as follows: 
\[\a(v,w):=v(\a(w)) \text{ for } v,w \in V^{\vee}.\]
\begin{enumerate}
\item For each linear subvariety $L \subset \P(V)$ which is possibly a point, 
we define 
\[L^{\perp}_{\a}:=\{[w] \in \P(V)\mid \a(v,w)=0 \text{ for all } v \in L  \}. \]
We call $L^{\perp}_{\a}$ \emph{the orthogonal subspace} of $L$ with respect to $\a$. 
Note that $\dim L+\dim L^{\perp}_{\a}=\dim \P(V)-1$. 
\item Let $x \in \P(V)$ be a point and $x \colon V \to \C$ be a corresponding surjection. Note that the hyperplane given by $\a(x) \colon \C \to V$ is $x^{\perp}_{\a}$. 
Since $x^{\perp}_{\a}$ contains $x$, we have $x \circ \a(x)=0$. Define 
\[V_{x}:=\Ker(x)/\Im(\a(x)).\]
This $V_{x}$ is $2n-2$-dimensional and has the non-degenerate skew form $\a_{x}$ induced by $\a$. 
\end{enumerate}
\end{defi}

The following lemma is well-known. 
\begin{lem}\label{LG-lem-LG}
Let $\a \in \bigwedge^{2}V$ be a non-degenerate skew form. 
Then the fiber of $x \in \P(V)$ under the morphism $f^{1n}_{1} \colon \LF_{\a}(1,n;V) \to \P(V)$ is isomorphic to $\LG_{\a_{x}}(n-1,V_{x}) \simeq \LG(n-1,2n-2)$. 
\end{lem}

For a general skew form $\a \in \bigwedge^{2}V=H^{0}(\P(V),\Omega_{\P(V)}(2))$, we obtain the following exact sequence
\begin{align}\label{LG-ex-null}
0 \to \mc O \mathop{\to}^{\a} \Omega_{\P(V)}(2) \to \mc N_{\P(V)}^{\a}(1) \to 0,
\end{align}
where the cokernel $\mc N_{\P(V)}^{\a}$ is so-called the \emph{null-correlation} bundle associated with the skew form $\a$ \cite[\S~4.2]{OSS80}. 

We take another general skew form $\b \in \bigwedge^{2}V$. 
By the natural map $\bigwedge^{2}V=H^{0}(\Omega_{\P(V)}(2)) \to H^{0}(\mc N_{\P(V)}^{\a}(1))$, we obtain the section $\ol{\b} \in H^{0}(\mc N_{\P(V)}^{\a}(1))$ as the image of $\b$. 
Let $l$ be the zero scheme of $\ol{\b}$, which is smooth since $\mc N^{\a}_{\P(V)}(1)$ is globally generated and $\b$ is general. 

\begin{lem}[{\cite[Proposition~3.2.3]{IlievRanestad05}}]\label{LG-lem-P1n}
Take general two skew forms $\a,\b \in \bigwedge^{2}V$. 
We regard $\a,\b$ as isomorphisms $V^{\vee} \to V$. 
\begin{enumerate}
\item The isomorphism $\b^{-1} \circ \a \colon V^{\vee} \to V^{\vee}$ has distinct $n$ eigenvalues $\lambda_{1},\ldots,\lambda_{n}$. For each $i$, the eigenspace $V_{i}^{\vee}$ of $\lambda_{i}$ is $2$-dimensional and gives the decomposition $V^{\vee}=\bigoplus_{i=1}^{n}V_{i}^{\vee}$. 
\item Set $l_{i}:=\P(V_{i}) \subset \P(V)$. 
Then $l$ is the disjoint union of $n$ lines $l_{1},\ldots,l_{n}$, which span $\P(V)$. 
\item For each $x_{i} \in l_{i}$, the linear span $\braket{x_{1},\ldots,x_{n}}$ is a Lagrangian subspace of $\P^{2n-1}$ for $\a$ and $\b$. 
\item For each Lagrangian subspace $\P(L) \subset \P(V)$, $\P(L)$ meets $l_{i}$ at a single point for each $i$. 
\item The Lagrangian Grassmannian $\LG_{\a,\b}(n;V)$ associated to $\a$ and $\b$ is canonically isomorphic to $l_{1} \times \cdots \times l_{n} \simeq (\P^{1})^{n}$. 
\item The universal quotient bundle $\mc Q_{\LG_{\a,\b}(n;V)}$ is isomorphic to $\bigoplus_{i=1}^{n} \pr_{i}^{\ast}\mc O(1)$ under the identification $\LG_{\a,\b}(n;V) \simeq l_{1} \times \cdots \times l_{n}$. 
Hence there is a natural isomorphism $\LF_{\a,\b}(1,n;V) \simeq \P_{(\P^{1})^{n}}\left( \bigoplus_{i=1}^{n} \pr_{i}^{\ast}\mc O(1) \right)$. 
\item The natural morphism $\LF_{\a,\b}(1,n;V) \to \P(V)$ is the composition of the blowing-ups
\[\Bl_{\wt{l_{1}^{\perp}}}\Bl_{\wt{l_{2}^{\perp}}} \cdots \Bl_{\wt{l_{n-1}^{\perp}}}\Bl_{l_{n}^{\perp}}\P(V) \to \P(V),\]
where $\wt{l_{i}^{\perp}}$ denotes the proper transform of the orthogonal space $l_{i}^{\perp}$ of $l_{i}$ with respect to $\a$. 
\end{enumerate}
\end{lem}
\begin{proof}
Most of the statements are proved in \cite[Proposition~3.2.3]{IlievRanestad05}. 
Nevertheless, we prove them for the convenience of later arguments. 

First we prove (1) and (2). 
For each $v \in \P(V)$, there is a corresponding surjection $x \colon V \to \C$. 
Let $v^{\vee} \colon \C \to V$ be its dual. 
Note that $x^{\perp}_{\a}=x^{\perp}_{\b}$ if and only if there is a scalar $\lambda \in \C$ such that $(\b^{-1} \circ \a-\lambda\id_{V^{\vee}})(v^{\vee})=0$, which means that $v^{\vee}$ is an eigenvector for $\lambda$. 
Therefore we obtain
\[l=\bigsqcup_{\lambda \in \C} \P_{\ast}(\Ker(\b^{-1} \circ \a-\lambda \cdot \id_{V^{\vee}})) \subset \P(V),\]
where $\P_{\ast}$ means the projectivization in the classical sense. 
By the exact sequence (\ref{LG-ex-null}), we have $[l]=n[\mc O(1)^{2n-2}]$ on $H_{2}(\P(V),\Z)$. 
Since $l$ is smooth and $1$-dimensional, $l$ is a disjoint union of $n$ lines $l_{1},\ldots,l_{n}$ and the linear span of $l$ is $\P(V)$. 
In particular, $\b^{-1} \circ \a$ has $n$ distinct eigenvalues $\lambda_{1},\ldots,\lambda_{n}$ and the eigenspace 
$V_{i}^{\vee}:=\Ker(\b^{-1} \circ \a - \lambda_{i} \cdot \id_{V^{\vee}}) \subset V^{\vee}$ is $2$-dimensional since $l_{i}=\P_{\ast}(V_{i}^{\vee})$. 

(3) Take a point $x_{i} \in l_{i}$ and let $v_{i} \in V^{\vee}$ be a non-zero element such that $[v_{i}]=x_{i}$. 
For each $w \in V^{\vee}$ and each $i$, we have $\a(v_{i},w)=\lambda_{i}\b(v_{i},w)$. 
Since $\a$ and $\b$ are skew forms, we have $\a(w,v_{i})=\lambda_{i}\b(w,v_{i})$. 
Thus we have $\lambda_{j}\b(v_{i},v_{j})=\a(v_{i},v_{j})=\lambda_{i}\b(v_{i},v_{j})$ for each $i,j$. 
If $i \neq j$, then $\b(v_{i},v_{j})=0$ by the above equality since $\lambda_{i} \neq \lambda_{j}$. 
Hence we have $\a(v_{i},v_{j})=0$ also. 
Therefore, the space $\braket{v_{1},\ldots,v_{n}}$ is Lagrangian for $\a$ and $\b$. 

(4) Let $\P(L) \subset \P(V)$ be a Lagrangian subspace. 
Let $\pi \colon V \epm L$ be the corresponding surjection. 
Then we obtain the following diagram:
\[\xymatrix{
0 \ar[r] & L^{\vee} \ar[r]^{ \pi^{\vee}} & V^{\vee} \ar[r]^{\pi \circ \a} \ar[d]_{\b^{-1} \circ \a}& L \ar[r]\ar@{=}[d] & 0 \\
0 \ar[r] & L^{\vee}  \ar[r]_{\pi^{\vee}} & V^{\vee} \ar[r]_{\pi \circ \b} & L \ar[r] & 0
}\]
Thus we obtain a natural automorphism $(\b^{-1} \circ \a)|_{L^{\vee}} \colon L^{\vee} \to L^{\vee}$. 
It is enough to show that the eigenvalues of $(\b^{-1} \circ \a)|_{L^{\vee}}$ are $n$ distinct scalars. 

We take a basis of $L^{\vee}$ as $e_{1},\ldots,e_{n}$. 
We may take $e_{n+1},\ldots,e_{2n} \in V^{\vee}$ such that $\{e_{1},\ldots,e_{2n}\}$ is a basis of $V^{\vee}$ and $\a(e_{i},e_{n+j})=\d_{ij}$ holds for each $1 \leq i,j \leq n$. 
Then the representation matrix $A$ of $\a$ is $\begin{pmatrix}0 & I_{n} \\ -I_{n}&\ast\end{pmatrix}$. 
By taking another $e_{n+1},\ldots,e_{2n}$, we may assume that $A=\begin{pmatrix}0 & I_{n} \\ -I_{n}&0\end{pmatrix}$. 
Let $B$ be the representation matrix of $\b$. 
Since $B$ is skew-symmetric and $\braket{e_{1},\ldots,e_{n}}$ is a Lagrangian, 
there are $n \times n$-matrices $B_{1}$ and $B_{2}$ such that 
$B=\begin{pmatrix}0 & B_{1} \\ -\!^{t}B_{1}&B_{2}\end{pmatrix}$. 
Then $B^{-1}=\begin{pmatrix}X & -\!^{t}(B_{1}^{-1}) \\ B_{1}^{-1}&0\end{pmatrix}$ for some $X$ and hence $B^{-1}A=\begin{pmatrix}\!^{t}(B_{1}^{-1})&X' \\ 0&B_{1}^{-1}\end{pmatrix}$ for some $X'$. 
Hence the set of the eigenvalues of $B^{-1}A$ equals that of $B_{1}^{-1}$. 
The representation matrix of $(\b^{-1} \circ \a)|_{L^{\vee}}$ is $\!^{t}(B_{1}^{-1})$. 
Therefore, $(\b^{-1} \circ \a)|_{L^{\vee}}$ has $n$ distinct eigenvectors. 

(5) By (3), we obtain a morphism 
\[l_{1} \times \cdots l_{n} \ni (v_{1},\ldots,v_{n}) \mapsto \braket{v_{1},\ldots,v_{n}} \in \LG_{\a,\b}(n,2n). \]
Since $\a,\b$ are general, we may assume that $\LG_{\a,\b}$ is smooth (cf. Remark~\ref{LG-rem-basechange}). 
By (4) and Zariski's main theorem, this morphism is isomorphic. 

(6) 
Set $V_{i}$ is the dual of $V_{i}^{\vee}$. 
Then $V=\bigoplus_{i=1}^{n}V_{i}$ and 
$\P(V_{i}) \subset \P(V)$ is nothing but the line $l_{i}$. 
Then we have $l_{i}^{\perp}:=(l_{i})_{\a}^{\perp}=\bigcap_{x \in l_{i}} x_{\a}^{\perp}$. 
Note that $x_{\a}^{\perp}=x_{\b}^{\perp}$ for any $x \in l$ and hence $l_{i}^{\perp}$ is the orthogonal space of $l_{i}$ with respect to $\b$ also. 

For a fixed $x \in \P(V)$, 
$x \not\in l_{i}^{\perp}$ for all $i$ 
if and only if 
$x$ is not contained in the indeterminacy of the rational map $\Phi_{i} \colon \P(V) \dra \P(V_{i})$ induced by $V_{i} \hra V$ for all $i$. 
Thus we obtain 
the point $x_{i} \in l_{i}$ corresponding to such an $x$. 
By (3) and (4), the linear span $L_{x}:=\braket{x_{1},\ldots,x_{n}}$ is the unique Lagrangian subspace containing $x$. 
Therefore, we have a rational map 
$\Phi \colon \P(V) \dra \LG_{\a,\b}(n,V)$ 
by sending a general $x \in \P(V)$ to $L_{x}$. 
By regarding $\LF_{\a,\b}(1,n;V)$ as the incidence variety $\{(x,L) \in \P(V) \times \LG_{\a,\b}(n,V) \mid x \in L\}$, 
the diagram $\P(V) \gets \LF_{\a,\b}(1,n;V) \to \LG_{\a,\b}(n,V)$ is the graph of $\Phi$. 
Under the identification $\LG_{\a,\b}(n,V) \simeq \P(V_{1}) \times \cdots \times \P(V_{n})$, 
$\Phi \colon \P(V) \dra \P(V_{1}) \times \cdots \times \P(V_{n})$ coincides with the natural birational map which is induced by $x \mapsto (\Phi_{i}(x))_{i}$. 
Therefore, we have a natural isomorphism
\begin{align*}
\LF_{\a,\b}(1,n;V) \simeq \ol{\left\{(x,y_{1},\ldots,y_{n}) \in \P(V) \times \prod_{i=1}^{n}\P(V_{i}) \mid x \notin \bigsqcup_{i}l_{i}, \pi_{i}(x)=y_{i}\right\}}. 
\end{align*}
Thus $\LF_{\a,\b}(1,n;V)$ is isomorphic to 
$\P_{(\P^{1})^{n}}\left( \bigoplus_{i=1}^{n} \pr_{i}^{\ast}\mc O(1)\right)$. 
Hence $\mc Q_{\LG_{\a,\b}(n;V)}$ is isomorphic to $\bigoplus_{i=1}^{n} \pr_{i}^{\ast}\mc O(1)$. 

(7) Let $\Lambda_{i}$ be the linear system on $\P(V)$ consisting of the hyperplane containing $l_{i}^{\perp}$. 
Then the proper transform of $\Lambda_{i}$ on $\Bl_{\wt{l_{1}^{\perp}}}\cdots\Bl_{l_{n}^{\perp}}\P(V)$ gives a morphism to $l_{i}$, and hence a morphism to $l_{1} \times \cdots \times l_{n}$, which is an elimination of the indeterminacy of $\Phi \colon \P(V) \dra l_{1} \times \cdots \times l_{n}$. 
Since $\LF_{\a,\b}(1,n;V)$ is the graph of $\Phi$, we have a morphism $\Bl_{\wt{l_{1}^{\perp}}} \cdots \Bl_{l_{n}^{\perp}}\P(V) \to \LF_{\a,\b}(1,n;V)$. 
This morphism is birational and hence isomorphic since the Picard rank of the source equals to that of the target. 
We complete the proof. 
\end{proof}

From now on, let $V$ denotes a $6$-dimensional vector space until the end of this section. 

\subsection{Lagrangian flag varieties for a $6$-dimensional vector space}\label{LG-subsec-onlya}
Take a non-degenerate skew form $\a \in \bigwedge^{2}V$ and consider the following diagram: 
\[\xymatrix{
&\ar[d]_{f^{123}_{13}}\LF_{\a}(1,2,3;V) \ar[rd]^{f^{123}_{12}} \ar[ld]_{f^{123}_{23}}& \\
\LF_{\a}(2,3;V)\ar[d]_{f^{23}_{3}}&\LF_{\a}(1,3;V) \ar[rd]_{f^{13}_{1}} \ar[ld]^{f^{13}_{3}}&\LF_{\a}(1,2;V) \ar[d]^{f^{12}_{1}}  \\
\LG_{\a}(3,V)&&\P(V)
}\]
All morphisms in the above diagram are the natural morphisms as in Definition~\ref{LG-defi-IF}. 
In particular, the morphisms $f^{123}_{13}$, $f^{123}_{23}$, $f^{23}_{3}$, $f^{13}_{3}$ are $\P^{1}$-bundles by Remark~\ref{LG-rem-basechange}. 

For each point $x \in \P(V)$, let $V_{x}$ be the $4$-dimensional vector space and $\a_{x} \in \bigwedge^{2}V_{x}$ the non-degenerate skew form induced by $\a$ as in Definition~\ref{defi-ortho}~(2). 
By Lemma~\ref{LG-lem-LG}, we have $(f^{13}_{1})^{-1}(x)=\LG_{\a_{x}}(2,V_{x})$, which is isomorphic to $\Q^{3}$. 
Moreover, the rank $2$ universal quotient bundle on $\LG_{\a_{x}}(2,V_{x})$ is the dual of the Spinor bundle $\ms S^{\vee}_{\Q^{3}}$ under the identification $\LG_{\a_{x}}(2,V_{x}) \simeq \Q^{3}$ \cite[Example~1.5]{Ottaviani88}. 
Then we obtain the following diagram by taking the base change to $\{x\} \hra \P(V)$: 
\begin{align}\label{LG-dia-P3Q3}
\xymatrix{
\P_{\Q^{3}}(\ms S^{\vee}_{\Q^{3}})=\LF_{\a_{x}}(1,2;V_{x})=(f^{123}_{1})^{-1}(x)\ar[d] \ar[rd] & \P_{\P(V_{x})}(\mc N_{\P(V_{x})}^{\a_{x}}(1)) \ar@{}[l]|{\hspace{50pt}=}\\
\Q^{3}=\LG_{\a_{x}}(2,V_{x})=(f^{13}_{1})^{-1}(x) \ar[rd] &  \P^{3}=\P(V_{x})=(f^{12}_{1})^{-1}(x) \ar[d] \\
&x, 
}
\end{align}
where $\mc N^{\a_{x}}_{\P(V_{x})}$ is the null-correlation bundle on $\P(V_{x})$ with respect to $\a_{x}$. 
In particular, $f^{12}_{1}$ is a $\P^{3}$-bundle.

\begin{nota}\label{LG-nota-Hi}
From now on, $H_{k_{i}}$ denotes the divisor on $\LF_{\a_{1},\ldots,\a_{m}}(k_{1},\ldots,k_{l};V)$ that is the pull-back of a hyperplane section with respect to the Pl\"{u}cker embedding under the morphism $\LF_{\a_{1},\ldots,\a_{m}}(k_{1},\ldots,k_{l};V) \to \LG_{\a_{1},\ldots,\a_{m}}(k_{i},V) \hra \Gr(k_{i},V)$ 
by the abuse of notation. 
\end{nota}

\begin{lem}\label{LG-lem-null}
The $\P^{3}$-bundle $f^{12}_{1} \colon \LF_{\a}(1,2;V) \to \P(V)$ is the projectivization of the null-correlation bundle $\mc N_{\P(V)}^{\a}(H_{1})$ and the tautological bundle is $\mc O_{\LF_{\a}(1,2;V)}(H_{2})$. 
Moreover, we have $h^{0}(\Omega_{f^{12}_{1}}(2H_{2}-2H_{1}))=h^{0}\left( \bigwedge^{2}\mc N^{\a}_{\P(V)} \right) = 1$. 
\end{lem}

\begin{proof}
First, we prove that $\LF_{\a}(1,2;V) \simeq \P_{\P(V)}(\mc N_{\P(V)}^{\a}(H_{1}))$. 
Consider the natural morphism $\Fl(1,2;V) \to \Gr(2,V)$. 
Then $\LF_{\a}(1,2;V)$ is the pull-back of a smooth hyperplane section of $\Gr(2,V)$. 
Under the identification $H^{0}(\Omega_{\P(V)}(2)) \simeq \bigwedge^{2}V$, 
this hyperplane section is corresponding to $\a$, which is nothing but $\LG_{\a}(2,V)$. 
By the exact sequence (\ref{LG-ex-null}), we have $\LF_{\a}(1,2;V) \simeq \P(\mc N^{\a}_{\P(V)}(1))$. 
Now we obtain the following Euler sequence on $\LF_{\a}(1,2;V)$: 
\begin{align}\label{LG-ex-Euler12a}
0 \to \Omega_{f^{12}_{1}}(H_{2}) \to (f^{12}_{1})^{\ast}(\mc N^{\a}_{\P(V)}(H_{1})) \to \mc O(H_{2}) \to 0. 
\end{align}

Next, we prove that $h^{0}(\Omega_{f^{12}_{1}}(2H_{2}-2H_{1}))=h^{0}\left( \bigwedge^{2}\mc N^{\a}_{\P(V)} \right)$. 
Taking the filtration of $\bigwedge^{2}(f^{12}_{1})^{\ast}\mc N^{\a}_{\P(V)}$ by using (\ref{LG-ex-null}), we have 
\begin{align}\label{LG-eq-wedgenull}
0 \to \left(\bigwedge^{2}\Omega_{f^{12}_{1}}\right)(2H_{2}) \to \bigwedge^{2}(f^{12}_{1})^{\ast}(\mc N^{\a}_{\P(V)}(H_{1})) \to \Omega_{f^{12}_{1}}(2H_{2}) \to 0
\end{align}
Since $R^{i}(f^{12}_{1})_{\ast}\left( \left(\bigwedge^{2}\Omega_{f^{12}_{1}} \right) (2H_{2})  \right)=0$ for any $i$, we obtain the equality. 

Finally, we prove that $h^{0}(\bigwedge^{2}\mc N_{\P(V)}^{\a})=1$. 
By taking the wedge product of the exact sequence (\ref{LG-ex-null}), we have $0 \to \mc N^{\a}_{\P(V)}(-1) \to \Omega_{\P(V)}^{2}(2) \to \bigwedge^{2}\mc N^{\a}_{\P(V)} \to 0$. 
Hence we have $h^{0}(\bigwedge^{2}\mc N^{\a}_{\P(V)})=h^{1}(\mc N^{\a}_{\P(V)}(-1))=1$ again by (\ref{LG-ex-null}). 
\end{proof}

\begin{defi}\label{LG-defi-mcE}
We define the coherent sheaf $\mc E$ on $\LF_{\a}(1,2;V)$ as follows. 
By Lemma~\ref{LG-lem-null}, we have $h^{0}(\Omega_{f^{12}_{1}}(2H_{2}-2H_{1}))=1$. 
We define $\mc E$ is the cokernel of the corresponding injection $\iota \colon \mc O(H_{1}) \to \Omega_{f^{12}_{1}}(2H_{2}-H_{1})$: 
\begin{align}\label{LG-ex-relnull}
0 \to \mc O(H_{1}) \mathop{\to}^{\iota} \Omega_{f^{12}_{1}}(2H_{2}-H_{1}) \to \mc E \to 0. 
\end{align}
\end{defi}

\begin{prop}\label{LG-prop-onlya}
The coherent sheaf $\mc E$ in Definition~\ref{LG-defi-mcE} is locally free and there is a natural isomorphism $\P_{\LF_{\a}(1,2;V)}(\mc E) \simeq \LF_{\a}(1,2,3;V)$. 
The tautological bundle of $\P_{\LF_{\a}(1,2;V)}(\mc E)$ coincides with $\mc O_{\LG_{\a}(1,2,3;V)}(H_{3})$ under this identification. 
\end{prop}
\begin{proof}
First, we show that $\mc E$ is locally free and $\P(\mc E)=\LF_{\a}(1,2,3;V)$. 
For each $x \in \P(V)$, the fiber $(f^{12}_{1})^{-1}(x)$ is naturally isomorphic to $\P(V_{x})$ and the morphism $\iota|_{\P(V_{x})} \colon \mc O \to \Omega_{\P(V_{x})}(2)$ is given by the non-degenerate skew form $\a_{x} \in \bigwedge^{2}V_{x}$ as in Definition~\ref{defi-ortho}~(2). 
Hence we have 
\begin{align}\label{LG-eq-Gfib}
\mc E|_{\P(V_{x})} \simeq \mc N^{\a_{x}}_{\P(V_{x})}(1).
\end{align}
In particular, $\mc E$ is locally free. 
Let $\pi_{\mc E} \colon \P_{\LF_{\a}(1,2;V)}(\mc E) \to \LF_{\a}(1,2;V)$ be the projectivization. 
Since $\mc E$ is a quotient of $\Omega_{f^{12}_{1}}(2H_{2}-H_{1})$ and $\Omega_{f^{12}_{1}}$ is a quotient of $\Omega_{\Fl(1,2;V)/\P(V)}$, $\P(\mc E)$ is contained in the ordinary flag $\Fl(1,2,3;V)$. 
From the diagram (\ref{LG-dia-P3Q3}), each flag contained in $\P_{\LF_{\a}(1,2;V)}(\mc E)$ is isotropic for $\a$. 
Thus we obtain the isomorphism $\P(\mc E) \simeq \LF_{\a}(1,2,3;V)$. 

Next, we prove that $\xi_{\mc E} \sim H_{3}$ under the above identification, where $\xi_{\mc E}$ is a tautological divisor of $\P(\mc E)$. 
By (\ref{LG-eq-Gfib}), we have $\xi_{\mc E} \sim_{\P(V)} H_{3}$ since every $f^{123}_{1}$-fiber is isomorphic to $\LF(1,2;4)$, which is a Fano manifold. 
Take an integer $a$ such that $\xi_{\mc E} \sim H_{3}+aH_{1}$ on $\LF_{\a}(1,2,3;V)$. 
If $\xi_{\mc E}$ is nef, then $a \geq 0$.  
If $a>0$, then $\xi_{\mc E}$ is the pull-back of an ample divisor on $\LF_{\a}(1,3;V)$ and hence we have $\xi_{\mc E}^{8}.H_{2}>0$. 
In other words, to show that $\xi_{\mc E} \sim H_{3}$, it suffices to prove that $\xi_{\mc E}$ is nef and $\xi_{\mc E}^{8}.H_{2}=0$. 

First, we prove the nefness of $\xi_{\mc E}$. 
By the exact sequence (\ref{LG-ex-relnull}) and (\ref{LG-eq-wedgenull}), it suffices to see that $\left( \bigwedge^{2}\mc N^{\a}_{\P(V)} \right)(H_{1})$ is globally generated, which follows from (\ref{LG-ex-null}). 
Next we compute $(\xi_{\mc E}^{8}.H_{2})_{\LF_{\a}(1,2,3;V)}=(-1)^{7} \cdot (s_{7}(\mc E).H_{2})_{\LF_{\a}(1,2;V)}$, where $s_{i}$ denotes the Segre class. 
By the exact sequences (\ref{LG-ex-Euler12a}) and (\ref{LG-ex-relnull}), we get 
\begin{align}\label{LG-eq-chernG}
c_{1}(\mc E)=2H_{2} \text{ and } c_{2}(\mc E)=2(H_{1}^{2}-H_{1}H_{2}+H_{2}^{2})
\end{align}
and hence 
$-s_{7}(\mc E).H_{2}=(c_{1}(\mc E)^{7}-6c_{1}(\mc E)^{5}c_{2}(\mc E)+10c_{1}(\mc E)^{3}c_{2}(\mc E)^{2}-4c_{1}(\mc E)c_{2}(\mc E)^{3}).H_{2}
=192H_{1}^{5}H_{2}^{3}-64H_{1}^{4}H_{2}^{4}-192H_{1}^{3}H_{2}^{5}+192H_{1}^{2}H_{2}^{6}-64H_{1}H_{2}^{7}$. 
Using the equation $H_{1}^{i}(H_{2}-H_{1})^{8-i}=(-1)^{5-i} s_{5-i}(\mc N_{\P(V)}).H_{1}^{i}$ and the fact that the Chern polynomial of $\mc N_{\P(V)}$ is given by $1+H_{1}^{2}t^{2}+H_{1}^{4}$, 
we obtain 
$H_{1}^{5}H_{2}^{3}=1$, $H_{1}^{4}H_{2}^{4}=4$, $H_{1}^{3}H_{2}^{5}=9$, $H_{1}^{2}H_{2}^{6}=14$ and $H_{1}H_{2}^{7}=14$. 
Thus we have 
$-s_{7}(\mc E).H_{2}=192 \cdot 1-64 \cdot 4-192 \cdot 9+192 \cdot 14 -64 \cdot 14=0$. 
We complete the proof. 
\end{proof}

\subsection{The blowing-up of $\LG_{\a}(3,V)$ along $\LG_{\a,\b}(3,V)$}\label{LG-subsec-blowup}

Let $\b \in \bigwedge^{2}V$ be another general skew form on $V^{\vee}$. 
In this subsection, we study the structure of $\Bl_{\LG_{\a,\b}(3,V)}\LG_{\a}(3,V)$. 

By Lemma~\ref{LG-lem-P1n}~(2), the two skew forms $\a$ and $\b$ produce the union of the three disjoint lines 
\begin{align}\label{LG-defi-li}
l = l_{1} \sqcup l_{2} \sqcup l_{3},
\end{align}
whose linear span is $\P(V)$. 
Consider the blowing-up 
\begin{align}\label{LG-defi-wtPV}
\s \colon \wt{\P(V)}:=\Bl_{l_{1} \sqcup l_{2} \sqcup l_{3}} \P(V) \to \P(V) \simeq \P^{5}.
\end{align}
Since $l=l_{1} \sqcup l_{2} \sqcup l_{3}$ is the zero scheme of the global section $\ol{\b} \in H^{0}(\mc N^{\a}_{\P^{5}}(1))$, we obtain a canonical surjection 
$(\mc N^{\a}_{\P(V)})^{\vee}(-1) \epm \mc I_{l/\P(V)}$. 
Since there is a natural isomorphism 
$(\mc N^{\a}_{\P(V)})^{\vee} \simeq \mc N^{\a}_{\P(V)}$, 
we obtain the surjection 
\begin{align}\label{LG-eq-surjwtPV}
\mc N^{\a}_{\P(V)}(1) \epm \mc I_{l/\P(V)}(2),
\end{align}
which gives the embedding 
$\wt{\P(V)} \hra \LF_{\a}(1,2;V)$ 
by Proposition~\ref{LG-prop-onlya}~(2). 
We set 
\[M:=(f^{123}_{12})^{-1} \left( \wt{\P(V)} \right) =\P_{\wt{\P(V)}}(\mc E|_{\wt{\P(V)}}) \subset \LF_{\a}(1,2,3;V)=\P_{\LF_{\a}(1,2;V)}(\mc E).\]
The last equality follows from Proposition~\ref{LG-prop-onlya}~(3). 
Then the morphism 
$\pi:=f^{123}_{12}|_{M} \colon M \to \wt{\P(V)}$ 
is the projectivization of $\mc E|_{\wt{\P(V)}}$. 
We also set 
\begin{align*}
\wt{\LG_{\a}(3,V)}:=f^{123}_{13}(M) \subset \LF_{\a}(1,3;V) \text{ and } f:=f^{13}_{1}|_{\wt{\LG_{\a}(3,V)}} \colon \wt{\LG_{\a}(3,V)} \to \P(V). 
\end{align*}
Additionally we define 
\begin{align}\label{LG-def-numu}
\nu:=f^{123}_{13}|_{M} \colon M \to \wt{\LG_{\a}(3,V)} \text{ and } \mu:=f^{13}_{1}|_{\wt{\LG_{\a}(3,V)}} \colon \wt{\LG_{\a}(3,V)} \to \LG_{\a}(3,V). 
\end{align} 
Then we consider the following diagram:
\begin{align}\label{LG-dia-M}
\xymatrix{
&M=\P_{\wt{\P(V)}}(\mc E|_{\wt{\P(V)}}) \ar[rd]^{\pi:=f^{123}_{12}|_{M}} \ar[d]_{f^{123}_{13}|_{M}=:\nu}& \\
&\ar[ld]_{f^{13}_{3}|_{\wt{\LG_{\a}(3,V)}}=:\mu}\wt{\LG_{\a}(3,V)} \ar[rd]_{f^{13}_{1}|_{\wt{\LG}}=:f}&\wt{\P(V)}\ar[d]^{\s:=f^{12}_{1}|_{\wt{\P(V)}}} \\
\LG_{\a}(3,V)&&\P(V).
}
\end{align}

\begin{lem}\label{LG-lem-Qi}
For each $i \in \{1,2,3\}$, set $Q_{i}:=f^{-1}(l_{i})$. 
We recall $l=l_{1} \sqcup l_{2} \sqcup l_{3}$ as in (\ref{LG-defi-li}). 
\begin{enumerate}
\item For each $x \in \P(V)$, $f^{-1}(x)=\P^{1}$ if and only if $x \notin l$. 
Moreover, $f|_{Q_{i}} \colon Q_{i} \to l_{i}$ is a $\Q^{3}$-bundle. 
\item The variety $\wt{\LG_{\a}(3,V)}$ is smooth and the morphism$\nu$ is the blowing-up of $\wt{\LG_{\a}(3,V)}$ along $Q_{1}$, $Q_{2}$, and $Q_{3}$. 
\item For each point $x \in l$, let $Q_{x} \simeq \Q^{3}$ be the fiber of $f$. Then $\mc N_{Q_{x}/\wt{\LG_{\a}(3,V)}} \simeq \ms S_{\Q^{3}} \oplus \mc O_{\Q^{3}}$. 
\end{enumerate}
\end{lem}
\begin{proof}
(1) If $x \in l$, then it follows from the diagram (\ref{LG-dia-P3Q3}) that 
$\s^{-1}(x)=(f^{12}_{1})^{-1}(x)=\P^{3}$, 
$\pi^{-1}(\s^{-1}(x))=(f^{123}_{1})^{-1}(x)=\LF(1,2;4)$ and 
$f^{-1}(x)=(f^{13}_{1})^{-1}(x)=\Q^{3}$. 
If $x \in \P(V) \setminus l$, then $\s^{-1}(x)$ is a point in $\P(V_{x})$ and $\pi^{-1}(\s^{-1}(x))=\P^{1}$, which implies (1). 

(2) For $x \in \P(V) \setminus l$, $f^{-1}(x)=\nu(\pi^{-1}(\s^{-1}(x)))$ is a line on $\Q^{3} \simeq \LG(2,4)=\LG_{\a_{x}}(2;V_{x})$. 
Hence $\nu$ is a birational morphism and $\nu(\Exc(\nu))=\bigsqcup_{i} Q_{i}$. 
Since $\nu^{-1}(Q_{i}) \to Q_{i}$ is a $\P^{1}$-bundle, $\nu$ is the blowing-up along $\bigsqcup_{i}Q_{i}$ by \cite[Theorem~2.3]{Ando85}. 

(3) Let $x \in l_{i}$. Consider the normal bundle sequence 
\[0 \to \mc N_{Q_{x}/Q_{i}} \to \mc N_{Q_{x}/\wt{\LG_{\a}(3,V)}} \to \mc N_{Q_{i}/\wt{\LG_{\a}(3,V)}}|_{Q_{x}} \to 0.\]
Since $\mc N_{Q_{x}/Q_{i}}=\mc O_{Q_{x}}$ and $\mc N_{Q_{i}/\wt{\LG_{\a}(3,V)}}|_{Q_{x}} \simeq \ms S_{\Q^{3}}$ by (2), we obtain the assertion. 
\end{proof}

Here we define the three $3$-planes on $\P(V)$ as follows. 

\begin{defi}\label{LG-def-Pij}
For each $1 \leq i<j \leq 3$, we set 
\begin{align}
P_{ij}:=\braket{l_{i},l_{j}} \simeq \P^{3} \text{ and } P=(P_{12} \cup P_{13} \cup P_{23})_{\red}. 
\end{align}
This $P_{ij}$ is nothing but $l_{k}^{\perp}$ for the unique element $k \in \{1,2,3\} \setminus \{i,j\}$, which is defined in Lemma~\ref{LG-lem-P1n}. 
Therefore, $\LF_{\a,\b}(1,3;V)$ is naturally isomorphic to $\Bl_{\wt{P_{23}}}\Bl_{\wt{P_{13}}}\Bl_{P_{12}}\P(V)$ by Lemma~\ref{LG-lem-P1n}~(7). 
\end{defi}

The main conclusion of this subsection is the following proposition. 

\begin{prop}\label{LG-prop-blowupLG}
Recall the morphism $\mu \colon \wt{\LG_{\a}(3,V)} \to \LG_{\a}(3,V)$ as in (\ref{LG-def-numu}). 
Then $\mu$ coincides with the blowing-up of $\LG_{\a}(3,V)$ along $\LG_{\a,\b}(3,V)$. 
Moreover, $E:=\Exc(\mu)$ is canonically isomorphic to $\LF_{\a,\b}(1,3;V)$ and the restriction $\mu|_{E} \colon E \to \LG_{\a,\b}(3,V)$ coincides with the natural morphism $g^{13}_{3} \colon \LF_{\a,\b}(1,3;V) \to \LG_{\a,\b}(3,V)$ under this identification. 
\end{prop}

\begin{proof}
We proceed with 5 steps. 

\noindent
\textbf{Step 1.} 
First, we confirm that $\mu \colon \wt{\LG_{\a}(3,V)} \to \LG_{\a}(3,V)$ is birational. 

We use the notation as in Notation~\ref{LG-nota-Hi}. 
By Remark~\ref{LG-rem-basechange}, $\LG_{\a}(3,V)$ is the zero scheme of $\a \in H^{0}(\Gr(3,V),\bigwedge^{2}\mc Q_{\Gr(3,V)})$. 
Hence we have $(H_{3})_{\LG_{\a}(3,V)}^{6}=c_{3}(\bigwedge^{2}\mc Q_{\Gr(3,V)}).c_{1}(\mc Q_{\Gr(3,V)})^{6}=16$. 
Thus it suffices to see that $(H_{3}|_{M})^{6}=16$, which is equivalent to say that $-s_{5}(\mc E).\wt{\P(V)}=16$, where $s_{5}$ is its fifth Segre class. 
From (\ref{LG-eq-surjwtPV}) we have $H_{2}=2H_{1}-\Exc(\s)$ on $\wt{\P(V)}$. 
Hence it follows from (\ref{LG-eq-chernG}) that 
\begin{align}\label{LG-eq-chernG1}
c_{1}(\mc E|_{\wt{\P(V)}})=4H_{1}-2\Exc(\s) \text{ and }  c_{2}(\mc E|_{\wt{\P(V)}})=2(3H_{1}^{2}-3H_{1}\Exc(\s)+\Exc(\s)^{2}).
\end{align}
It is easy to see that 
\begin{align}\label{LG-eq-intwtPV}
H_{1}^{2}\Exc(\s)=0, \quad H_{1}^{5}=1,\quad H_{1}\Exc(\s)^{4}=-3 \text{ and } \Exc(\s)^{5}=-12.
\end{align}
Hence we have 
$-s_{5}(\mc E)
=c_{1}(\mc E)^{5}-4c_{1}(\mc E)^{3}c_{2}(\mc E)+3c_{1}(\mc E)c_{2}(\mc E)^{2}
=16$. 
Therefore, $\mu \colon \wt{\LG_{\a}(3,V)} \to \LG_{\a}(3,V)$ is birational.

\noindent
\textbf{Step 2.} Let us consider the following diagram: 
\[\xymatrix{
&\ar[d]_{g^{123}_{13}}\LF_{\a,\b}(1,2,3;V) \ar[rd]^{g^{123}_{12}} \ar[ld]_{g^{123}_{23}}& \\
\LF_{\a,\b}(2,3;V)\ar[d]_{g^{23}_{3}}&\LF_{\a,\b}(1,3;V) \ar[rd]^{g^{13}_{1}} \ar[ld]_{g^{13}_{3}} \ar@{}[d]|{\rotatebox{90}{$\simeq$}}&\LF_{\a,\b}(1,2;V) \ar[d]^{g^{12}_{1}}  \\
\LG_{\a,\b}(3,V) \ar@{}[d]|{\rotatebox{90}{$\simeq$}}&\P_{(\P^{1})^{3}}\left(\bigoplus_{i=1}^{3} \pr_{i}^{\ast}\mc O(1)\right) \ar[ld] \ar[rd]&\P(V) \ar@{=}[d] \\
(\P^{1})^{3}&&\P(V) 
}\]
All morphisms in the above diagram are the natural morphisms as in Definition~\ref{LG-defi-IF}. 
Note that $\LG_{\a,\b}(3,V) \simeq l_{1} \times l_{2} \times l_{3}$ and $\mc Q_{3} \simeq \bigoplus_{i=1}^{3} \pr_{i}^{\ast}\mc O(1)$ on $\LG_{\a,\b}(3,V)$ by Lemma~\ref{LG-lem-P1n}. 
For each $i$, we define 
\begin{align}\label{LG-defi-Sigmai}
\Sigma_{i} \subset \LF_{\a,\b}(1,3;V)
\end{align}
as the $g^{13}_{3}$-section corresponding to the surjection $\bigoplus_{i=1}^{3} \pr_{i}^{\ast}\mc O(1) \epm \pr_{i}^{\ast}\mc O(1)$. 

\noindent
\textbf{Step 3.} In this step, we show that a general member $E' \in |H_{1}+H_{2}-H_{3}|$ on $\LF_{\a,\b}(1,2,3;V)$ is isomorphic to $\Bl_{\Sigma_{1},\Sigma_{2},\Sigma_{3}}\LF_{\a,\b}(1,3;V)$. 

Set $\mc H:=\bigoplus_{i=1}^{3} \pr_{i}^{\ast}\mc O(1)$ on $\LG_{\a,\b}(3,V)=(\P^{1})^{3}$. 
On $\LF_{\a,\b}(1,3;V) \simeq \P_{(\P^{1})^{3}}(\mc H)$, there is the relative Euler sequence: 
\begin{align}\label{LG-ex-Euler13ab}
0 \to \Omega_{g^{13}_{3}}(H_{1}) \to (g^{13}_{3})^{\ast}\mc H \to \mc O(H_{1}) \to 0.
\end{align}
Thus $\Omega_{g^{13}_{3}}(2H_{1})$ is globally generated. 
Then $\LF_{\a,\b}(1,2,3;V)$ coincides with the projectivization 
$\P_{\LF_{\a,\b}(1,3;V)}(\Omega_{g^{13}_{3}}(2H_{1}))$ and 
its tautological bundle is isomorphic to $\mc O(H_{2})$. 
Hence we have 
\begin{align*}
H^{0}(\LF_{\a,\b}(1,2,3;V),\mc O(H_{1}+H_{2}-H_{3}))
&=H^{0}(\LF_{\a,\b}(1,3;V),\Omega_{g^{13}_{3}}(3H_{1}-H_{3})) \\
&=H^{0}(\LF_{\a,\b}(1,3;V),T_{g^{13}_{3}}) .
\end{align*}
By taking the dual of the exact sequence (\ref{LG-ex-Euler13ab}) and tensoring with $\mc O(H_{1})$, we obtain the following exact sequence: 
\begin{align}\label{LG-ex-Euler13ab'}
0 \to \mc O \to (g^{13}_{3})^{\ast}\mc H^{\vee} \otimes \mc O(H_{1}) \to T_{g^{13}_{3}} \to 0. 
\end{align}
Note that $H^{i}((g^{13}_{3})^{\ast}\mc H^{\vee} \otimes \mc O(H_{1}))=\Ext^{i}(\mc H,\mc H)$, which implies $h^{0}(T_{g^{13}_{3}})=2$. 
Moreover, the exact sequence (\ref{LG-ex-Euler13ab'}) gives 
\begin{align}\label{LG-eq-c2}
c_{2}(T_{g^{13}_{3}})=[\Sigma_{1}+\Sigma_{2}+\Sigma_{3}]. 
\end{align}
Let $E' \in |H_{1}+H_{2}-H_{3}|$ be a member and $s \in H^{0}(\LF_{\a,\b}(1,2,3;V),\mc O(H_{1}+H_{2}-H_{3}))=H^{0}(\LF_{\a,\b}(1,3;V),T_{g^{13}_{3}})$ be the corresponding section. 
Let $\Sigma=(s=0) \subset \LF_{\a,\b}(1,3;V)$ be the zero scheme of $s$. 

\begin{claim}
If $s$ is general, then $\Sigma$ is purely of codimension $2$. 
\end{claim}
\begin{proof}
It suffices to show that if $T_{g^{13}_{3}}(-D)$ has global sections for some effective divisor $D$, then $h^{0}(T_{g^{13}_{3}}(-D))=1$. 

Let $x,y_{1},y_{2},y_{3}$ be integers such that 
$D \in |\mc O(xH_{1}) \otimes (g_{3}^{13})^{\ast}\mc O(y_{1},y_{2},y_{3})|$. 
Assume that $T_{g^{13}_{3}} \otimes \mc O(-xH_{1}) \otimes (g_{3}^{13})^{\ast}\mc O(-y_{1},-y_{2},-y_{3})$ has a section.
Then we have $x \in \{0,1\}$.  
If $x=0$, then $y_{i} \geq 0$ for each $i$ and we can easily deduce a contradiction to $h^{0}(T_{g^{13}_{3}} \otimes (g_{3}^{13})^{\ast}\mc O(-y_{1},-y_{2},-y_{3})) \neq 0$ from the sequence (\ref{LG-ex-Euler13ab'}). 
Thus we have $x=1$. 
By (\ref{LG-ex-Euler13ab'}), we have 
$0 \neq 
h^{0}(T_{g^{13}_{3}} \otimes \mc O(-H_{1}) \otimes (g_{3}^{13})^{\ast}\mc O(-y_{1},-y_{2},-y_{3}))=
h^{0}((\P^{1})^{3},\mc H^{\vee} \otimes \mc O(-y_{1},-y_{2},-y_{3}))$. 
On the other hand, we have 
$0 \neq 
h^{0}(\mc O(H_{1}) \otimes (g_{3}^{13})^{\ast}\mc O(y_{1},y_{2},y_{3}))=
h^{0}((\P^{1})^{3},\mc H \otimes \mc O(y_{1},y_{2},y_{3}))$. 
Therefore, one of $\{y_{i}\}_{i=1}^{3}$ is equal to $-1$ and the other elements are $0$. 
From the direct calculation, we have $h^{0}(T_{g^{13}_{3}}(-H_{1}) \otimes (g^{13}_{3})^{\ast}\mc O(y_{1},y_{2},y_{3}))=1$. 
We complete the proof. 
\end{proof}

Considering the restriction of (\ref{LG-ex-Euler13ab'}) on $\Sigma_{1} \simeq (\P^{1})^{3}$, we get 
\[0 \to \mc O_{\Sigma_{1}} \to \mc O \oplus \mc O(1,-1,0) \oplus \mc O(1,0,-1) \to T_{g^{13}_{3}}|_{\Sigma_{1}} \to 0.\]
Thus $T_{g^{13}_{3}}|_{\Sigma_{1}} \simeq \mc O(0,-1,0) \oplus \mc O(0,0,-1)$, which has no global section, which implies the zero scheme $(s=0)$ contains $\Sigma_{1}$. 
By the similar reasons, $(s=0)$ contains $\Sigma_{1} \sqcup \Sigma_{2} \sqcup \Sigma_{3}$. 
Since (\ref{LG-eq-c2}) holds, we have $(s=0)=\Sigma_{1} \sqcup \Sigma_{2} \sqcup \Sigma_{3}$. 
Thus the member $E' \in |H_{1}+H_{2}-H_{3}|$ is isomorphic to $\Bl_{\Sigma_{1},\Sigma_{2},\Sigma_{3}} \LF_{\a,\b}(1,3;V)$. 

\noindent
\textbf{Step 4.}
Let $E' \in |H_{1}+H_{2}-H_{3}|$ be a member, which is isomorphic to the blowing-up $\Bl_{\Sigma_{1},\Sigma_{2},\Sigma_{3}} \LF_{\a,\b}(1,3;V)$ by Step~3. 
The morphism $g^{123}_{1}|_{E'} \colon E' \to \P(V)$ is birational and $(g^{123}_{1}|_{E'})^{-1}(l_{i})$ is a divisor for each $i$. 
Then the universal property of blowing-up gives a birational morphism $E' \to \wt{\P(V)}$. 

We recall the three $3$-planes $P_{12},P_{13},P_{23}$ as in Definition~\ref{LG-def-Pij}. 
By the construction and Lemma~\ref{LG-lem-P1n}~(7), $E' \to \wt{\P(V)}$ is nothing but the blowing-up along $\wt{P_{12}} \sqcup \wt{P_{13}} \sqcup \wt{P_{23}}$, where $\wt{P_{ij}}$ denotes the proper transform of $P_{ij}$ on $\wt{\P(V)}$. 
Thus we obtain 
\begin{align}\label{LG-eq-inclu}
E' \subset \P_{\wt{\P(V)}}(\mc E|_{\wt{\P(V)}}) \text{ and hence } \LF_{\a,\b}(1,3;V) \subset \wt{\LG_{\a}(3,V)}. 
\end{align}

\noindent
\textbf{Step 5.} Finally, we prove Proposition~\ref{LG-prop-blowupLG}. 
From the construction, $\wt{\LG_{\a}(3,V)}$ is of Picard rank $2$. 
By Step~1, $\mu$ is birational and hence $\mu$ contracts a divisor. 
Since $\LF_{\a,\b}(1,3;V)$ is contained in $\wt{\LG_{\a}(3,V)}$ by (\ref{LG-eq-inclu}), $\mu$ contracts $\LF_{\a,\b}(1,3;V)$ onto $\LG_{\a,\b}(3,V)$. 
By the universality of blowing-up and Zariski's main theorem implies that $\mu$ is the blowing-up along $\LG_{\a,\b}(3,V)$.  
In particular, $E=\Exc(\mu)$ coincides with $\LF_{\a,\b}(1,3;V)$. 
\end{proof}

As a conclusion of this subsection, we obtain the following diagram: 
\begin{align*}
\xymatrix{
\Bl_{Q_{1},Q_{2},Q_{3}}\wt{\LG_{\a}(3,V)}=M\ar[d]_{\nu}&\ar@{_{(}->}[l] E'=\Bl_{\Sigma_{1},\Sigma_{2},\Sigma_{3}}E \ar[rd]^{\pi|_{E'}} \ar[d]_{\nu|_{E'}}& \\
\Bl_{\LG_{\a,\b}(3,V)}\LG_{\a}(3,V)=\wt{\LG_{\a}(3,V)}\ar[d]_{\mu}\ar@{}[rd]|{\Box}&\ar@{_{(}->}[l] \ar[d]_{\mu|_{E}=g^{13}_{3}}E=\LF_{\a,\b}(1,3;V) \ar[rd]^{f|_{E}=g^{13}_{1}}&\wt{\P(V)}\ar[d]^{\s} \\
\LF_{\a}(3,V)&\ar@{_{(}->}[l](\P^{1})^{3}=\LG_{\a,\b}(3,V)&\P(V), 
}
\end{align*}
where $E=\Exc(\mu)$ and $E' \subset M$ is the proper transform of $E$. 
Note that $\pi|_{E'}$ and $f|_{E}$ are birational. 
In particular, $E'$ is a rational section of the $\P^{1}$-bundle $\pi \colon M=\P(\mc E|_{\wt{\P(V)}}) \to \wt{\P(V)}$. 

We end this subsection by showing the following lemma, which will be necessary for proving Lemma~\ref{LG-lem-key}. 

\begin{lem}\label{LG-lem-relonwtLG}
It follows that $-K_{M}=2H_{3}+2H_{1}-\Exc(\nu)$ on $M$ and $H_{1}=H_{3}-E$ on $\wt{\LG_{\a}(3,V)}$. 
\end{lem}
\begin{proof}
By using the blowing-ups $\nu$ and $\mu$, we have 
$-K_{M}=4H_{1}-(2\nu^{\ast}E+\Exc(\nu))$. 
On the other hand, since $\pi$ is the projectivization of $\mc E|_{\wt{\P(V)}}$, 
we have 
$-K_{M}=2H_{3}+2H_{1}-\pi^{\ast}\Exc(\s)$. 
It follows from the construction that $\pi^{\ast}\Exc(\s)=\Exc(\nu)$. 
Therefore, it holds that $H_{3}=H_{1}+E$ on $\wt{\LG_{\a}(3,V)}$. 
\end{proof}


\subsection{The definition of $R$ and its resolution}\label{LG-subsec-R}
Let $R:=f(\mu^{-1}(C)) \subset \P(V)$, where $C$ is a smooth codimension $2$ linear section of $(\P^{1})^{3}=\LG_{\a,\b}(3,V)$. 
The purpose of this subsection is to study the variety $R$.
Unlike the case of Section~\ref{KG2-section}, $R$ is non-normal along $l_{1}$, $l_{2}$, and $l_{3}$. 
The main proposition of this section shows that $\Bl_{l_{1},l_{2},l_{3}}R$ is smooth. 

We first prove the following lemma. 

\begin{lem}\label{LG-lem-W}
Let $S \subset (\P^{1})^{3}=\LG_{\a}(3,V)$ be an arbitrary smooth element of $|\bigotimes_{i=1}^{3}\pr_{i}^{\ast}\mc O(1)|$. 
Set 
\begin{align*}
W:=\mu^{-1}(S) \subset E \text{ and } \ol{W}=f(W) \subset \P(V).
\end{align*}
Then $\ol{W}$ is a hypercubic containing $P_{12},P_{13},P_{23}$. 
Moreover, $f|_{W} \colon W \to \ol{W}$ is a flopping contraction and its flop $W^{+}$ is smooth. 
\end{lem}
\begin{proof}
First of all, we note that $W$ is a member of $|H_{3}|_{E}|=|3H_{1}|_{E}-\Exc(f|_{E})|$ and hence $\ol{W}$ is a hypercubic of $\P^{5}$ containing $P_{12}$, $P_{13}$, and $P_{23}$. 
We show that $f|_{W} \colon W \to \ol{W}$ is a flopping contraction. 
For each $i \in \{1,2,3\}$, let $L_{i}$ be the pullback of a point on $\P^{1}$ under the $i$-th projection $\pr_{i} \colon (\P^{1})^{3} \to \P^{1}$ and $S_{i}:=\Sigma_{i} \cap W$, where $\Sigma_{i}$ is defined in (\ref{LG-defi-Sigmai}). 
For a curve $\g \subset W$ satisfying $H_{1}.\g=0$, we have $(H_{1}-L_{j}).\g<0$ for each $j \in \{1,2,3\}$. 
Since $S_{1}$ is the intersection of the unique member of $|H_{1}-L_{2}|$ and that of $|H_{1}-L_{3}|$, we get $\g \subset S_{1}$. 
By such arguments, we get $\Exc(f|_{W}) \subset S_{1} \sqcup S_{2} \sqcup S_{3}$. 
Since $S_{i}$ is contracted by $f|_{W}$ onto the line $l_{i}$, we have $\Exc(f|_{W})=S_{1} \sqcup S_{2} \sqcup S_{3}$. 
Therefore, $\ol{W}$ is regular in codimension $1$. 
Hence $\ol{W}$ is normal and $f|_{W}$ is a small contraction. 
Since $-K_{W} \sim 3H_{1}|_{W}$ follows from the adjunction formula, $f|_{W}$ is a flopping contraction. 

Finally, we show that the flop $W^{+}$ of $W$ is smooth. 
Since $(\P^{1})^{3}$ is homogeneous, we may assume that 
\[S=\{z_{0}z_{2}z_{4}+z_{1}z_{3}z_{5}=0\} \subset (\P^{1})^{3}, \]
where $([z_{0}:z_{1}],[z_{2}:z_{3}],[z_{4}:z_{5}])$ denotes the coordinate of $(\P^{1})^{3}$. 
Then $\ol{W} \subset \P^{5}=\Proj \C[x_{0},x_{1},x_{2},x_{3},x_{4},x_{5}]$ is given by the same equation: 
\begin{align}\label{LG-defeq-olW}
\ol{W}=\{x_{0}x_{2}x_{4}+x_{1}x_{3}x_{5}=0\} \subset \P^{5}.
\end{align}
The smoothness of $W^{+}$ is proved by the similar arguments as in \cite[Example~2.3]{Kollar89} as follows. 
Now $\mu^{\ast}H_{3}|_{W}$ is $f|_{W}$-ample. 
Set $\ol{H_{3}}:={f|_{W}}_{\ast}\mu^{\ast}(H_{3}|_{W})$. 
Then the flop $W^{+}$ is isomorphic to 
$\Proj_{\ol{W}}\bigoplus_{m \geq 0} \mc O(-m\ol{H_{3}})$ 
by \cite[Lemma~6.2]{KM98}. 
Set $U_{0}:=\{x_{0} \neq 0\} \simeq \mathbb{A}^{5}$ and $\ol{W}_{0}:=\ol{W} \cap U_{0}$. 
Then $\ol{W}_{0}=\{x_{2}x_{4}+x_{1}x_{3}x_{5}=0\}$ has an involution 
$\iota \colon \ol{W}_{0} \ni (x_{1},x_{2},x_{3},x_{4},x_{5}) \mapsto (x_{1},x_{4},x_{3},x_{2},x_{5}) \in \ol{W}_{0}$. 
Since $\ol{W}_{0}/\braket{\iota} \simeq 
\mathbb{A}^{4}$, it holds that $\iota^{\ast}\ol{H_{3}}+\ol{H_{3}} \sim 0$ and hence 
$W^{+}$ is isomorphic to $\Proj_{\ol{W_{0}}} \bigoplus_{m \geq 0} \mc O(m\ol{H_{3}})=(f|_{W})^{-1}(\ol{W_{0}})$ over $\ol{W_{0}}$. 
By taking another affine chart and doing the same argument, we conclude that $W^{+}$ is smooth. 
\end{proof}


\begin{prop}\label{LG-prop-R}
Let $C \subset (\P^{1})^{3}$ be a smooth codimension $2$ linear section with respect to the polarization $\bigotimes_{i=1}^{3} \pr_{i}^{\ast}\mc O(1)$. 
We set 
\begin{align}\label{LG-def-R}
R:=f(\mu^{-1}(C)) \subset \P(V). 
\end{align}
Then the following assertions hold. 
\begin{enumerate}
\item 
Define $C_{i}:=\Sigma_{i} \cap \mu^{-1}(C)$, where $\Sigma_{i}$ is defined in (\ref{LG-defi-Sigmai}). 
Recall the three lines $l_{1}$, $l_{2}$, and $l_{3}$ as in (\ref{LG-defi-li}). 

Then $f|_{\mu^{-1}(C)}$ is isomorphic outside the three sections $C_{1} \sqcup C_{2} \sqcup C_{3}$ and $f|_{C_{i}} \colon C_{i} \to l_{i}$ is a double covering. 
In particular, $f|_{\mu^{-1}(C)} \colon \mu^{-1}(C) \to R$ is the normalization of $R$. 
\item 
Let $F_{ij}$ be the exceptional divisor of $f|_{E} \colon E \to \P^{5}$ whose center is $P_{ij}$, which is defined in Definition~\ref{LG-def-Pij}. 
Note that $\mu|_{F_{ij}} \colon F_{ij} \to (\P^{1})^{3}$ is the projectivization of the rank $2$ vector bundle $\pr_{i}^{\ast}\mc O(1) \oplus \pr_{j}^{\ast}\mc O(1)$. 

Then the geometrically ruled surface $F_{ij} \cap \mu^{-1}(C)$ is the blowing-up of $P_{ij} \cap R$ along $l_{i}$ and $l_{j}$. 
Moreover, $R \cap P_{ij}$ is a hyperquartic surface of $P_{ij}=\P^{3}$. 
\item Let $S,S'$ be smooth hyperplane sections of $(\P^{1})^{3}$ such that $C=S \cap S'$. 
Let $W=\mu^{-1}(S)$, $W'=\mu^{-1}(S')$, $\ol{W}=f(W)$, and $\ol{W'}=f(W')$ as in Lemma~\ref{LG-lem-W}. 

Then the scheme $P \cup R$ with the reduced structure is a complete intersection of the two hypercubics $\ol{W}$ and $\ol{W'}$. 
In particular, 
$\ol{W}+\ol{W'}$ is a simple normal crossing at the generic points of $P_{12},P_{13},P_{23}$ and $R$. 
\end{enumerate}
\end{prop}

\begin{proof}
(1) Note that $\mu^{-1}(C)=\P_{C}(\mc H|_{C})$. 
First we check that $f|_{\mu^{-1}(C)} \colon \mu^{-1}(C) \to R$ is given by the complete linear system $|\mc O_{\P_{C}(\mc H|_{C})}(1)|$. 
Take a smooth hyperplane section $S \subset (\P^{1})^{3}$ containing $C$. 
Then it suffices to show that $H^{1}(S,\mc I_{C/S}(\mc H|_{S}))=0$. 
Since $S$ is a sextic del Pezzo surface, there exists the blowing-down $S \to \P^{2}$. 
Let $h$ be the pull-back of a line and $e_{1},e_{2},e_{3}$ the exceptional divisors. 
Then $\mc H \simeq \bigoplus_{i=1}^{3} \mc O(h-e_{i})$ and $C \in |3h-\sum_{i}e_{i}|$. 
Now it is easy to check that $H^{1}(S,\mc H(-C))=0$. 

Since $\mc H|_{C} \simeq \left(\bigoplus_{i=1}^{3} (\pr_{i}^{\ast}\mc O_{\P^{1}}(1) \right)|_{C}$, 
the morphism $\P_{C}(\mc H|_{C}) \to R$ is finite, birational, and not isomorphic along only $\bigsqcup_{i=1}^{3} C_{i}$. 

(2) Without loss of generality, we may assume that $i=1$ and $j=2$. 
Then we have $F_{12} \simeq \P_{(\P^{1})^{3}}(\mc O(1,0,0) \oplus \mc O(0,1,0))$, which is the blowing-up of $\Bl_{l_{1},l_{2}}P_{12} \times \P^{1}$. 
Let $\pi_{12} \colon (\P^{1})^{3} \to (\P^{1})^{2}$ be the projection defined by forgetting the third component. 
Then $\pi_{12}|_{C}$ is an isomorphism onto its image. 
Therefore, the restriction of the morphism $\Bl_{l_{1},l_{2}}P_{12} \times \P^{1} \to \Bl_{l_{1},l_{2}}P_{12}$ to $F_{ij} \cap \mu^{-1}(C)$ is also isomorphic. 
By \cite[Exercise~III-(9)]{BeauvilleBook} or \cite[Example~6.11]{HP13} for example, it is well-known that the blowing-up of $R \cap P_{12}$ along $l_{1}$ and $l_{2}$ is $F_{ij} \cap \mu^{-1}(C) \subset \Bl_{l_{1},l_{2}}P_{12}$ and $R \cap P_{12}$ is a hyperquartic. 
We complete the proof of (2). 

(3) 
Since $R \cap P_{ij}$ is a quartic by (2), $\ol{W} \cap \ol{W'}$ must contain $P=\bigcup_{i,j}P_{ij}$. 
Hence $\ol{W} \cap \ol{W'}$ contains $R \cup P$. 
Since both of $R \cup P$ and $\ol{W'} \cap \ol{W}$ are of degree $9$, $R+P$ equals to $\ol{W'}|_{\ol{W}}$ as Weil divisors on $\ol{W}$. 
By inducing the scheme structure on $R+P$ with respect to the ideal $\mc O_{\ol{W}}(-R-P)$, $R+P$ satisfies the $S_{1}$-condition and hence isomorphic to $R \cup P$ with the reduced structure.  
Hence $R \cup P$ is the complete intersection of $\ol{W}$ and $\ol{W'}$. 
%
\end{proof}

\begin{prop}\label{LG-prop-wtR}
Let $C$ and $R$ be as in Proposition~\ref{LG-prop-R}. 
We set 
\[\wh{R}:=\nu^{-1}_{\ast}(\mu^{-1}(C)) \subset E' \text{ and } \wt{R}:=\pi(\wh{R}) \subset \wt{\P(V)}. \]
Let $C_{i}$ be as in Proposition~\ref{LG-prop-R}~(1). 
Note that $\wh{R} \simeq \Bl_{C_{1},C_{2},C_{3}}\mu^{-1}(C)$ by definition and $\wt{R}=\Bl_{l_{1},l_{2},l_{3}}R$ since $\wt{R}$ is the proper transform of $R$. 

Then the following assertions hold. 
\begin{enumerate}
\item The morphism $\pi|_{\wh{R}} \colon \wh{R} \to \wt{R}$ is isomorphic. 
In particular, $\s|_{\wt{R}} \colon \wt{R} \to R$ is a resolution of singularities of $\wt{R}$. 
\item The threefold $(\s|_{\wt{R}})^{-1}(l_{i})$ is isomorphic to $\P_{C}\left( \bigoplus_{j=\{1,2,3\} \setminus \{i\}} \pr_{j}^{\ast}\mc O(1)|_{C} \right)$ for each $i \in \{1,2,3\}$. 
\end{enumerate}
\end{prop}
\begin{proof}
(2) directly follows from (1). 
To prove (1), we consider the following diagram: 
\[\xymatrix{
&\LF_{\a,\b}(1,2,3;V) \ar[rd]^{g^{123}_{23}} \ar[ld]_{g^{123}_{13}} & \\
E=\LF_{\a,\b}(1,3;V)\ar[rd]_{\mu|_{E}=g^{13}_{3}}&&\LF_{\a,\b}(2,3;V) \ar[ld]^{g^{23}_{3}}  \\ 
&\LG_{\a,\b}(3,V).&
}\]
Note that $E' \subset \LF_{\a,\b}(1,2,3;V)$ and the birational morphisms $\LF_{\a,\b}(1,3;V) \gets E' \to \LF_{\a,\b}(1,2;V)$ is nothing but a family of Cremona transformations. 

Let the following diagram be the base change of the above diagram to $C \subset \LG_{\a,\b}(3,V)$: 
\[\xymatrix{
&R_{123} \ar[rd]^{h^{123}_{23}} \ar[ld]_{h^{123}_{13}} & \\
\mu^{-1}(C)=R_{13}\ar[rd]_{\mu|_{\mu^{-1}(C)}=h^{13}_{3}}&&R_{23} \ar[ld]^{h^{23}_{3}}  \\ 
&C.&
}\]
Note that $\wh{R}=R_{123} \cap E'$, which is the blowing-up of $R_{13}$ along three $h^{13}_{3}$-sections $\bigsqcup_{i=1}^{3}\Sigma_{i} \cap R_{13}=\bigsqcup_{i=1}^{3}C_{i}$. 
Set $R_{12}=g^{123}_{12}(R_{123}) \subset \LF_{\a,\b}(1,2;V)$ and $R_{2}:=g^{23}_{2}(R_{23}) \subset \LG_{\a,\b}(2;V)$. 
Then the image of $\wh{R}$ under $R_{123} \to R_{12}$ is $\wt{R}$: 
\[\xymatrix{
\wh{R} \ar[r]^{\pi|_{\wh{R}}} \ar@{^{(}->}[d]& \wt{R} \ar@{^{(}->}[d] \\
R_{123} \ar[r]^{g^{123}_{12}|_{R_{123}}} \ar[d]_{g^{123}_{23}} & R_{12} \ar[d] \\
R_{23} \ar[r]_{g^{23}_{2}|_{R^{23}_{2}}} & R_{2}.
}\]
Note that the bottom diagram is cartesian. 
If the morphism $R_{23} \to R_{2}$ is an isomorphism, then $g^{123}_{12}|_{R_{123}}$ is also isomorphic and hence so is $\pi|_{\wh{R}} \colon \wh{R} \to \wt{R}$. 

Let us prove that $g^{23}_{2}|_{R_{23}} \colon R_{23} \to R_{2}$ is isomorphic. 
$R_{23}$ is nothing but the projectivization of $\mc V|_{C}$, where $\mc V$ is the vector bundle on $(\P^{1})^{3}$ defined by 
\[\mc V:=\mc O(1,1,0) \oplus \mc O(1,0,1) \oplus \mc O(0,1,1).\]
Since $\mc O_{\P_{C}(\mc V|_{C})}(1)$ is very ample, 
it is enough to show that the restriction morphism 
$H^{0}((\P^{1})^{3},\mc V) \to H^{0}(C,\mc V|_{C})$ is surjective. 
This is equivalent to say that $H^{1}((\P^{1})^{3},\mc I_{C/(\P^{1})^{3}} \otimes \mc V)=0$, which follows from the resolution $0 \to \mc O(-2,-2,-2) \otimes \mc V \to (\mc O(-1,-1,-1) \otimes \mc V)^{\oplus 2} \to \mc I_{C} \otimes \mc V \to 0$. 
We complete the proof. 
\end{proof}

\subsection{The smooth extraction of $\P(V)$ along $R$}\label{LG-subsec-extraction}

The next out purpose is to construct the divisorial extraction $Z \to \P^{5}$ along $R$ and to show that $Z$ is smooth (Proposition~\ref{LG-prop-Z}). 
Our approach for the construction of $Z$ is as follows. 
Recall the blowing-up $\s \colon \wt{\P(V)} \to \P(V)$ along $l=l_{1} \sqcup l_{2} \sqcup l_{3}$ (\ref{LG-defi-wtPV}) and $R=f(\mu^{-1}(C)) \subset \P(V)$ (\ref{LG-def-R}). 
By Proposition~\ref{LG-prop-wtR}, the proper transform $\wt{R} \subset \wt{\P(V)}$ of $R$ is a non-singular threefold. 
We consider the blowing-up $\t \colon Y:=\Bl_{\wt{R}}\wt{\P(V)} \to \wt{\P(V)}$. 
In Proposition~\ref{LG-prop-Z}, we show that there is a birational morphism 
$p_{Y} \colon Y \to Z$ over $\P(V)$ and $Z$ is smooth: 
\begin{align}\label{LG-dia-Z}
\xymatrix{
&Y:=\Bl_{\wt{R}}\wt{\P(V)} \ar[d]^{p_{Y}} \ar[ld]_{\t} \\
\wt{\P(V)}\ar[d]_{\s}&Z \ar[ld]^{g} \\
\P(V).&
}
\end{align}
The exceptional divisors of $p_{Y}$ are the proper transforms of the exceptional divisors of $\s \colon \wt{\P(V)} \to \P(V)$. 
Hence the birational morphism $g \colon Z \to \P(V)$ is the extraction along $R$. 

We show the smoothness of $Z$ as follows. 
Let $D_{i}$ be the exceptional divisor of $\wt{\P(V)} \to \P(V)$ whose center is $l_{i}$ and $T_{i}:=\wt{R} \cap D_{i}$. 
Then the blowing-up $\Bl_{T_{i}}D_{i}$ is the proper transform of $D_{i}$ on $\wt{\P(V)}$. 
Let $\t_{i} \colon \Bl_{T_{i}}D_{i} \to D_{i}$ be the blowing-up morphism. 
In Proposition~\ref{LG-prop-quadspi}, we prove that $\Bl_{T_{i}}D_{i}$ is isomorphic to the projectivization of a rank $2$ vector bundle $\mc G_{i}$ on a smooth projective threefold $Z_{i}$. 
Moreover, $Z_{i}$ has a quadric fibration structure $z_{i} \colon Z_{i} \to l_{i}$: 
\[\xymatrix{
&\Bl_{T_{i}}D_{i}=\P_{Z_{i}}(\mc G_{i}) \ar[d]^{p_{i}} \ar[ld]_{\t_{i}} \\
D_{i}\ar[d]_{\s_{i}}&Z_{i} \ar[ld]^{z_{i}} \\
l_{i}.&
}\]
As a conclusion, we show that $\Exc(p_{Y})=\bigsqcup_{i} \Bl_{T_{i}}D_{i}$ and $p_{Y}|_{D_{i}}=p_{i}$ in Proposition~\ref{LG-prop-Z}. 
Then it follows from Ando's theorem~\cite[Theorem~2.3]{Ando85} that $Z$ is smooth.

\subsubsection{Structure of $\Bl_{T_{i}}D_{i}$. }\label{LG-subsubsec-li}

First, we see that $\Bl_{T_{i}}D_{i}$ has a $\P^{1}$-bundle structure on a quadric fibration $z_{i} \colon Z_{i} \to l_{i}$ (=Proposition~\ref{LG-prop-quadspi}). 

In $\S$~\ref{LG-subsubsec-li}, we work over the following notation. 
Set $D_{i}:=\s^{-1}(l_{i}) \subset \wt{\P(V)}$, $T_{i}:=\wt{R} \cap D_{i}$, and $\s_{i}:=\s|_{D_{i}} \colon D_{i} \to l_{i}$. 
Note that $D_{i} \simeq \P^{3} \times \P^{1}$ and hence $\s_{i}$ is the second projection. 
Let $\xi_{i}$ be the pullback of a hyperplane on $\P^{3}$ under $D_{i} \to \P^{3}$ and $H_{i}$ a fiber of $D_{i} \to \P^{1}$. 
Note that
\begin{align}\label{LG-def-xi}
\xi=(-D_{i}+H_{i})|_{D_{i}} 
\end{align}
since $\mc N_{l_{i}/D_{i}} \simeq \mc O(1)^{4}$ holds. 
For the blowing-up $\t_{i} \colon \Bl_{T_{i}}D_{i} \to D_{i}$, we set 
\begin{align*}
G_{i}:=\Exc(\t_{i}).
\end{align*}

\begin{prop}\label{LG-prop-quadspi}
For each $i$, there exist a smooth projective threefold $Z_{i}$ with a quadric fibration structure $z_{i} \colon Z_{i} \to l_{i}$ and a nef vector bundle $\mc G_{i}$ on $Z_{i}$ satisfying the following. 
\begin{enumerate}
\item It holds that $\P_{Z_{i}}(\mc G_{i})=\Bl_{T_{i}}D_{i}$ and the following diagram commutes: 
\[\xymatrix{
\P_{Z_{i}}(\mc G_{i})=\Bl_{T_{i}}D_{i} \ar[r]^{\qquad \t_{i}} \ar[d]_{p_{i}}& D_{i} \ar[d]_{\s_{i}} \\
Z_{i} \ar[r]_{z_{i}} & l_{i}.
}\]
\item For each $t \in l_{i}$, let $Z_{i,t}$ denotes the fiber of $z_{i}$, which is an irreducible and reduced quadric surface. 
Then $\mc G_{i}|_{Z_{i,t}}$ is the restriction of the dual of the Spinor bundle on $\Q^{3}$ to the irreducible reduced quadric surface $Z_{i,t}$. 
\item 
Let $H_{i}$ also denotes the linear equivalence class of a fiber of $Z_{i} \to l_{i}$. 
Let $A_{Z_{i}}$ be a divisor on $Z_{i}$ such that $\mc O(A_{Z_{i}}):=\det \mc G_{i}$. 
Then we obtain 
\begin{align}\label{LG-rel-AZ}
p_{i}^{\ast}A_{Z_{i}} \sim -G_{i}+2H_{i}+2\t_{i}^{\ast}\xi_{i}. 
\end{align}
\end{enumerate}
\end{prop}


Our proof of the above proposition use the following lemma. 

\begin{lem}\label{LG-lem-surj}
There is a surjection $\Omega_{\s_{i}} \otimes \s_{i}^{\ast}\mc O_{\P^{1}}(-2) \epm \mc I_{T_{i}/D_{i}}$. 
\end{lem}
\begin{proof}
By Proposition~\ref{LG-prop-wtR}~(3), the Stein factorization of $T_{i} \to l_{i}$ factors through the $\P^{1}$-bundle $T_{i} \to C_{i}$ and the double covering $C_{i} \to l_{i}$, where $C_{i}$ is defined in Proposition~\ref{LG-prop-wtR}~(1). 
For a point $t \in l_{i}$, let $T_{i,t}$ and $D_{i,t}=\P^{3}$ denote the scheme theoretic fibers under the morphisms $\s|_{T_{i}} \colon T_{i} \to l_{i}$ and $\s|_{D_{i}} \colon D_{i} \to l_{i}$ respectively. 
To show the lemma, we need the following claim. 

\begin{claim}\label{LG-claim-surj}
For every $t \in l_{i}$, it holds that $\Hom(\Omega_{\P^{3}},\mc I_{T_{i,t}/\P^{3}})=\C$. 
Moreover, every non-zero element $\a \in \Hom(\Omega_{\P^{3}},\mc I_{T_{i,t}/\P^{3}})$ is a surjective. 
\end{claim}

\begin{proof}
Set $\mc I:=\mc I_{T_{i,t}/\P^{3}}$. 
When the covering $C_{i} \to l_{i}$ is unramified (resp. ramified) over $t$, 
$T_{i,t}$ is a union of disjoint two lines on $\P^{3}$ (resp. $T_{i,t} \simeq \P^{1} \times \Spec \C[\e]/(\e^{2})$ and the reduction of $T_{i,t}$ is a line on $\P^{3}$). 
By the exact sequence $0 \to \mc I(n) \to \mc O_{\P^{3}}(n) \to \mc O_{T_{i,t}}(n) \to 0$ for each $n \in \Z$, we have $h^{0}(\P^{3},\mc I)=0$ and 
\begin{align}\label{LG-eq-cohom1}
h^{i}(\P^{3},\mc I(1))=0 \text{ for any } i. 
\end{align}
Applying $\Hom(-,\mc I)$ to the Euler sequence $0 \to \Omega_{\P^{3}} \to \mc O(-1)^{4} \to \mc O \to 0$, we have $\dim \Hom(\Omega_{\P^{3}},\mc I)=1$. 

Therefore, it is enough to show that there is a surjection $\Omega_{\P^{3}} \to \mc I$. 
For this purpose, we consider the syzygy of $\mc I$. 
Let $[x:y:z:w]$ be a coordinate of $\P^{3}$. 
If $T_{i,t}$ is a disjoint union of two lines, then we may assume that 
$\mc I=(xz,xw,yz,yw)$. 
By taking the syzygy, we have the following presentation on $\P^{3}$
\[0 \to \mc O(-4) \mathop{\to}^{A_{2}} \mc O(-3)^{4} \mathop{\to}^{A_{1}} \mc O(-2)^{4} \mathop{\to}^{A_{0}} \mc I \to 0,\]
where 
\begin{align*}
A_{0}=(xz,xw,yz,yw),\  A_{1}=\begin{pmatrix}
-w&-y&0&0 \\
z&0&0&-y \\
0&x&-w&0 \\
0&0&z&x
\end{pmatrix} \text{ and } A_{2}=\begin{pmatrix}-y \\ w \\ x \\ -z\end{pmatrix}. 
\end{align*}
Let $B_{2}:=\!^{t}A_{2}$ be the transpose of $A_{2}$. 
Then we have a surjection 
$\ds \mc O(3)^{4} \mathop{\epm}^{B_{2}} \mc O(4)$. 
The syzygy of the morphism $B_{2}$ is given by the graded Koszul resolution of the quotient ring $\C[x,y,z,w]/(-y,w,x,-z)$ as follows. 
\[0 \to \mc O \mathop{\to}^{B_{-1}} \mc O(1)^{4} \mathop{\to}^{B_{0}} \mc O(2)^{6} \mathop{\to}^{B_{1}} \mc O(3)^{4} \mathop{\to}^{B_{2}} \mc O(4) \to 0.\]
Let $\wt{A_{i}}=\!^{t}B_{i}$ be the transpose. 
From the construction, we have $\wt{A_{2}}=A_{2}$ and a projection $p \colon \mc O(-2)^{6} \to \mc O(-2)^{4}$ such that $p \circ \wt{A_{1}}=A_{1}$: 
\[\xymatrix{
0 \ar[r] & \mc O(-4) \ar[r]^{\wt{A_{2}}} \ar@{=}[d] & \mc O(-3)^{4} \ar[r]^{\wt{A_{1}}} \ar@{=}[d]&  \mc O(-2)^{6} \ar[r]^{\wt{A_{0}}} \ar@{->>}[d]_{p} & \Omega_{\P^{3}} \ar[r] & 0 \\
0 \ar[r] & \mc O(-4) \ar[r]_{A_{2}} \ar[r] & \mc O(-3)^{4} \ar[r]_{A_{1}} & \mc O(-2)^{4} \ar[r]_{A_{0}} & \mc I \ar[r] & 0. 
}\]
Thus we have a surjection $\Omega_{\P^{3}} \to \mc I$. 

Finally, we consider the case when $T_{i,t}$ is non-reduced. 
Let $l$ be the reduction of $T_{i,t}$, which is a line. 
Then we may assume that $\mc I_{l/\P^{3}}=\sqrt{\mc I}=(x,y)$, 
where $[x:y:z:w]$ is a coordinate of $\P^{3}$. 
Since $T_{i,t} \simeq l \times \Spec \C[\e]/(\e^{2})$, 
it holds that $(\sqrt{\mc I})^{2} \subset \mc I$ as $\mc O_{\P^{3}}$-modules. 
Hence we have $x^{2},xy,y^{2} \in H^{0}(\P^{3},\mc I(2))$. 
Since $h^{0}(\P^{3},\mc I(2))=4$, we have a quadratic equation $q \neq 0$ such that $\{x^{2},xy,y^{2},q\}$ is a basis of $H^{0}(\P^{3},\mc I(2))$. 
Since $\{q=0\}$ contains the line $l=\{x=y=0\}$, 
there exist two linear terms $m_{1},m_{2}$ such that $q=m_{1}x+m_{2}y$.
We may assume that $m_{i}=a_{i}z+b_{i}w$ for some $a_{i},b_{i} \in \C$. 
Set $M=(m_{1}=m_{2}=0)$. 
If $\dim M=2$, then the scheme theoretic intersection $M \cap T_{i,t}$ is $0$-dimensional and of length $3$, which contradicts that $T_{i,t} \simeq \P^{1} \times \Spec \C[\e]/(\e^{2})$. 
Therefore, $m_{1}$ and $m_{2}$ are linearly independent. 
Then we may assume that $m_{1}=z$ and $m_{2}=w$. 
Thus we have $\mc I=(x^{2},xy,y^{2},xz+yw)$. 
By taking syzygy, we have the following presentation of $\mc I$: 
\[0 \to \mc O(-4) \mathop{\to}^{A_{2}} \mc O(-3)^{4} \mathop{\to}^{A_{1}} \mc O(-2)^{4} \mathop{\to}^{A_{0}} \mc I \to 0,\]
where 
\begin{align*}
A_{0}=(x^{2},xy,y^{2},xz+yw),\  A_{1}=\begin{pmatrix}
-y&0&0&-z \\
x&-y&-z&-w \\
0&x&-w&0 \\
0&0&y&x
\end{pmatrix} \text{ and } A_{2}=\begin{pmatrix}z \\ w \\ x \\ -y\end{pmatrix}. 
\end{align*}
By the same arguments as in the case when $T_{i,t}$ is reduced, 
we have the following morphism between the following exact sequences: 
\[\xymatrix{
0 \ar[r] & \mc O(-4) \ar[r]^{\wt{A_{2}}} \ar@{=}[d] & \mc O(-3)^{4} \ar[r]^{\wt{A_{1}}} \ar@{=}[d]&  \mc O(-2)^{6} \ar[r]^{\wt{A_{0}}} \ar@{->>}[d]_{p} & \Omega_{\P^{3}} \ar[r] & 0 \\
0 \ar[r] & \mc O(-4) \ar[r]_{A_{2}} \ar[r] & \mc O(-3)^{4} \ar[r]_{A_{1}} & \mc O(-2)^{4} \ar[r]_{A_{0}} & \mc I \ar[r] & 0, 
}\]
where $p$ is a projection. 
Thus we have a surjection $\Omega_{\P^{3}} \to \mc I$. 
We complete the proof of Claim~\ref{LG-claim-surj}. 
\end{proof}

We go back to the proof of Lemma~\ref{LG-lem-surj}. 
We consider the Euler sequence $0 \to \mc O \to \mc O(\xi_{i})^{4} \to T_{\s_{i}} \to 0$. 
Tensoring the above sequence with $\mc I_{T_{i}}$ and taking the cohomology of the direct image ${\s_{i}}_{\ast}$, we have 
$\displaystyle {\s_{i}}_{\ast}(\mc O(\xi_{i})^{4} \otimes \mc I_{T_{i}/D_{i}}) \to {\s_{i}}_{\ast}(T_{\s_{i}} \otimes \mc I_{T_{i}/D_{i}}) \mathop{\to}^{a} \wt{R}^{1}{\s_{i}}_{\ast}\mc I_{T_{i}/D_{i}} \to 0$. 
Since ${\s_{i}}_{\ast}(\mc O(\xi_{i})^{4} \otimes \mc I_{T_{i}/D_{i}})=0$ by (\ref{LG-eq-cohom1}), 
the map $a$ is an isomorphism. Considering the exact sequence 
$0 \to \mc I_{T_{i}/D_{i}} \to \mc O_{D_{i}} \to \mc O_{T_{i}} \to 0$, 
we have the isomorphism $\wt{R}^{1}{\s_{i}}_{\ast}\mc I_{T_{i}/D_{i}} \simeq \mc O_{\P^{1}}(-2)$ on $l_{i} \simeq \P^{1}$ since $T_{i} \to l_{i}$ factors the elliptic curve $C$ as the Stein factorization. 
Thus we have ${\s_{i}}_{\ast}(T_{\s_{i}} \otimes \mc I_{T_{i}/D_{i}}) \simeq \mc O(-2)$. 
We take a non-zero element $\a \in H^{0}(D_{i},T_{\s_{i}} \otimes \mc I_{T_{i}/D_{i}} \otimes \s_{i}^{\ast}\mc O(2))=H^{0}(l_{i},{\s_{i}}_{\ast}(T_{\s_{i}} \otimes \mc I_{T_{i}/D_{i}}) \otimes \mc O(2))=\C$. 
For each $t \in l_{i}$, the restriction map $\C=H^{0}(D_{i},T_{\s_{i}} \otimes \mc I_{T_{i}} \otimes \s_{i}^{\ast}\mc O(2)) \to H^{0}(D_{i,t},T_{D_{i,t}} \otimes \mc I_{T_{i,t}})=\C$ is isomorphic. 
Therefore, the restriction $\a|_{D_{i,t}}$ gives a non-zero map $\a_{t} \colon \Omega_{D_{i,t}} \to \mc I_{T_{i},t}$, which is a surjection by Claim~\ref{LG-claim-surj}. 
Therefore, we have a surjection $\Omega_{\s_{i}} \otimes \s_{i}^{\ast}\mc O(-2) \to \mc I_{T_{i}}$. 
We complete the proof of Lemma~\ref{LG-lem-surj}. 
\end{proof}

\begin{proof}[Proof of Proposition~\ref{LG-prop-quadspi}]
By Lemma~\ref{LG-lem-surj}, we have an embedding 
\begin{align}\label{LG-eq-emb}
\Bl_{T_{i}}D_{i} \hra \P_{D_{i}}(\Omega_{\s_{i}} \otimes \s_{i}^{\ast}\mc O(-2))
\end{align}
over $D_{i}$. 
Since $D_{i} \simeq \P^{3} \times \P^{1}$, we have $\P_{D_{i}}(\Omega_{\s_{i}}) \simeq \Fl(1,2;4) \times \P^{1}$, where $\Fl(1,2;4)$ denotes the flag variety. 
Let $p \colon \P_{D_{i}}(\Omega_{\s_{i}}) \to \Gr(2,4) \times \P^{1}$ be the projection, 
$p_{i}:=p|_{\Bl_{T_{i}}D_{i}}$ and $Z_{i}$ the image of $p_{i}$. 
Then every fiber of $p_{i}$ is at most $1$-dimensional.  
For a general point $t \in l_{i}$, $\Bl_{T_{i,t}}D_{i,t}$ is the blowing-up of $\P^{3}$ along a disjoint union of two lines. 
This is the Fano threefold of No.25 in Table 3 of \cite{MM81} and $p_{i,t} \colon \Bl_{T_{i,t}}D_{i,t} \to Z_{i,t}$ coincides with the $\P^{1}$-bundle over a smooth quadric surface. 
By the upper semi-continuity of the dimension of the fiber, we conclude that $p_{i}$ is a $\P^{1}$-bundle. 
Thus $Z_{i}$ is non-singular and it holds that $\rho(Z_{i})=\rho(D_{i})-1=2$, which implies $Z_{i} \to l_{i}$ is an extremal contraction. 
Therefore, $Z_{i} \to l_{i}$ is a quadric fibration. 
By the above argument, $p_{i}$ is a base change of the $\P^{1}$-bundle structure $\Fl(1,2;4) \times \P^{1} \to \Gr(2,4) \times \P^{1}$. 
Letting $\mc G_{i}$ denote the pull-back of $\mc Q_{\Gr(2,4)}$ by $Z_{i} \hra \Gr(2,4)$, we obtain (1) and (2). 

(3) The pull-back of $\mc O_{\P_{D_{i}}(\Omega_{\s_{i}} \otimes \s_{i}^{\ast}\mc O(-2))}(1)$ under the embedding (\ref{LG-eq-emb}) is linearly equivalent to $\mc O(-G_{i})$ by Lemma~\ref{LG-lem-surj}. 
Since the pull-back of $\det \mc Q_{\Gr(2,4)}$ under the morphism $\P_{D_{i}}(\Omega_{\s_{i}} \otimes \s_{i}^{\ast}\mc O(-2)) \to \Gr(2,4)$ is $\mc O_{\P_{D_{i}}(\Omega_{\s_{i}} \otimes \s_{i}^{\ast}\mc O(-2))}(1) \otimes \s_{i}^{\ast}\mc O(2) \otimes \mc O(2\xi_{i})$, we have (\ref{LG-rel-AZ}). 
We complete the proof of Proposition~\ref{LG-prop-quadspi}. 
\end{proof}


\subsubsection{Construction of the extraction along $R$. }\label{LG-subsubsec-extraction}

Next, we construct the divisorial extraction $Z \to \P^{5}$ along $R$ and show that $Z$ is smooth. 
Set 
\begin{align*}
\t \colon Y:=\Bl_{\wt{R}}\wt{\P(V)} \to \wt{\P(V)} \text{ and } G:=\Exc(\t). 
\end{align*}
We use the notation of $\S$~\ref{LG-subsubsec-li}. 
Set $D_{i,Y} = \t^{\ast}D_{i}$ be the proper transform of $D_{i}$ on $Y$, which is nothing but $\Bl_{T_{i}}D_{i}$. 
Set $D=\bigsqcup_{i=1}^{3}D_{i}$ and $D_{Y}=\bigsqcup_{i=1}^{3}D_{i,Y}$. 
Let $H_{Y}$ be the pull-back of a hyperplane on $\P(V)=\P^{5}$. 

\begin{prop}\label{LG-prop-Z}
There exist a smooth projective fivefold $Z$ of Picard rank $2$, a birational morphism $g \colon Z \to \P^{5}$, and a birational morphism $p_{Y} \colon Y \to Z$ over $\P^{5}$ as in the diagram (\ref{LG-dia-Z}). 
Moreover, 
\begin{itemize}
\item the birational morphism $p_{Y} \colon Y \to Z$ contracts the proper transforms of the three $\s$-exceptional divisor $D_{1,Y}$, $D_{2,Y}$, and $D_{3,Y}$, and 
\item $g$ is the divisorial contraction whose center is $R$. 
\end{itemize}
\end{prop}

\begin{proof}
Set $M=4H_{Y}-(2D_{Y}+G)$. 
Since $D_{Y}|_{D_{i,Y}}=\t_{i}^{\ast}(D_{i}|_{D_{i}})=\t_{i}^{\ast}(H_{i}-\xi_{i})$, we have 
\[M|_{D_{i,Y}}=p_{i}^{\ast}A_{Z_{i}}\]
from (\ref{LG-def-xi}) and (\ref{LG-rel-AZ}). 
In particular, $M|_{D_{Y}}$ is nef and hence $M$ is nef over $\P^{5}$. 
From the construction and (\ref{LG-def-xi}), 
$-K_{Y}|_{D_{i,Y}} \sim (6H_{Y}-3D_{Y}-G)|_{D_{i,Y}} = M|_{D_{i,Y}}+(2H_{Y}-D_{Y})|_{D_{i,Y}}=p_{i}^{\ast}A_{Z_{i}}+\t_{i}^{\ast}(\xi+H_{i})$ is relatively ample over $\P^{5}$. 
By the relative Kawamata-Shokurov base point free theorem, $M$ is relatively semi-ample over $\P^{5}$. 
Let $p_{Y} \colon Y \to Z$ be the contraction over $\P^{5}$ given by $M$. 
Then $p_{Y}$ contracts $D_{i,Y}$ to $Z_{i}$. 
On the other hand, if $\g$ denotes a curve such that $M.\g=0$ and $(\s \circ \t)(\g)$ is a point,  
then $M \sim_{\P^{5}} -K_{Y}+D_{Y}$ implies that $D_{Y}.\g<0$ and hence $\g \subset D_{Y}$. 
Therefore, $\g$ is a fiber of the $\P^{1}$-bundle $D_{i,Y} \to Z_{i}$. 
Using Ando's theorem \cite[Theorem~2.3]{Ando85}, we conclude that $\Exc(p_{Y})=D_{Y}$, that $Z$ is smooth, and that $p_{Y}$ is the blowing-up along $Z_{1} \sqcup Z_{2} \sqcup Z_{3}$. 
Since $M=4H_{Y}-(2D_{Y}+G)$ is the pull-back of some $p_{Y}$-ample divisor on $Z$, we have ${p_{Y}}_{\ast}p_{Y}^{\ast}M=M$, which implies ${p_{Y}}^{\ast}G_{Y}=G+2D_{Y}$. 
Since $\rho(Z)=2$ by the construction, the morphism $g \colon Z \to \P^{5}$ is an extremal divisorial contraction with exceptional divisor $G_{Z}={p_{Y}}_{\ast}G$ and $g(G_{Z})=R \subset \P^{5}$. 
\end{proof}
\begin{rem}
Over the open subset $\P^{5} \setminus (l_{1}\sqcup l_{2} \sqcup l_{3})$, $g$ is the blowing-up along $R \setminus (l_{1}\sqcup l_{2} \sqcup l_{3})$. 
Note that $Z_{i}:=g^{-1}(l_{i}) \to l_{i}$ is the quadric fibration as in Proposition~\ref{LG-prop-quadspi}. 
This contraction $g \colon Z \to \P^{5}$ is a family of the fourfold contractions that are given by Andreatta-Wi{\'s}niewski in \cite[(3.2)]{AW98}. 
\end{rem}

\subsection{The flip of $Z$}\label{LG-subsec-flip}

In this subsection, we will prove that the extraction $Z$ is flipped into a $(\P^{2})^{2}$-fibration with smooth total space.  

\begin{sett}\label{LG-sett-stflip}
We fix a ladder $C \subset S \subset (\P^{1})^{3}$ by members of $|\mc O(1,1,1)|$ such that $S$ and $C$ are smooth. 
Let $W,\ol{W}$ be as in Lemma~\ref{LG-lem-W}, $R$ as in Proposition~\ref{LG-prop-R}, and $\wt{R}$ as in Proposition~\ref{LG-prop-wtR}. 
Let $P_{12},P_{13},P_{23}$ be the $3$-planes as in Definition~\ref{LG-def-Pij}.  

Let $g \colon Z \to \P(V)$ be the divisorial extraction of $\P(V)$ along $R$ as in Proposition~\ref{LG-prop-Z}. 
Let $H_{Z}$ be a pull-back of a hyperplane of $\P(V)$ by $g$ and $G_{Z}=\Exc(g)$. 
Let $W_{Z}$ and $P_{ij,Z}$ be the proper transform of $\ol{W}$ and $P_{ij}$ on $Z$ for each $1 \leq i<j \leq 3$ respectively. 
Note that $W_{Z} \in |3H_{Z}-G_{Z}|$ by Lemma~\ref{LG-lem-W}.
\end{sett}

The conclusion is the following proposition. 
\begin{prop}\label{LG-prop-stflip}
In Setting~\ref{LG-sett-stflip}, the following assertions hold. 
\begin{enumerate}
\item The proper transforms $P_{12,Z}$, $P_{13,Z}$, and $P_{23,Z}$ are disjoint and isomorphic to $\P^{3}$. 
\item There exists a flip $\Psi \colon Z \dra Z^{+}$ of the $P_{12,Z}$, $P_{13,Z}$, and $P_{23,Z}$. 
Moreover, $Z^{+}$ is smooth and has an extremal contraction $\vp_{Z^{+}} \colon Z^{+} \to \P^{1}$. 
\item Let $W_{Z^{+}}$ be the proper transform of $W_{Z}$ on $Z^{+}$. 
Then the birational map $\Psi|_{W_{Z}} \colon W_{Z} \dra W_{Z^{+}}$ is the blowing-down of $P_{12,Z}$, $P_{13,Z}$, and $P_{23,Z}$ to smooth three points. 
Moreover, $W_{Z^{+}}$ is a $\vp_{Z^{+}}$-fiber and $W_{Z^{+}} \simeq (\P^{2})^{2}$. 
\item Every smooth fiber of $\vp_{Z^{+}}$ is isomorphic to $(\P^{2})^{2}$. 
\end{enumerate}

In summary, we obtain the following diagram. 
\begin{align*}
\xymatrix{
&\LF_{\a,\b}(1,3;V) \ar@{}[d]|{\rotatebox{90}{$=$}}&&Y\ar[ld]_{\t} \ar[d]_{p_{Y}}&&&\\
\LG_{\a,\b}(3,V)\ar@{}[d]|{\rotatebox{90}{$=$}}&\ar[ld]_{\mu|_{E}}E\ar[rd]^{f|_{E}} &\wt{\P(V)}\ar[d]_{\s}&Z \ar@{-->}[r]^{\Psi} \ar[ld]_{g} & Z^{+} \ar[rd]^{\vp_{Z^{+}}} & \\
(\P^{1})^{3}\ar@{}[r]|{\Box}  &\ar[ld]W\ar[rd]^{f|_{W}\quad } \ar@{}[u]|{\rotatebox{90}{$\subset$}}&\P(V)&W_{Z}\ar[r]^{ \Psi|_{W_{Z}}} \ar[ld]_{\qquad g|_{W_{Z}}}\ar@{}[u]|{\rotatebox{90}{$\subset$}}&W_{Z^{+}} \ar@{}[r]|{\Box}\ar[rd]\ar@{}[u]|{\rotatebox{90}{$\subset$}}&\P^{1}  \\
S \ar@{}[r]|{\Box} \ar@{}[u]|{\rotatebox{90}{$\subset$}}&\ar[ld] \mu^{-1}(C)\ar[rd]^{f|_{\mu^{-1}(C)}} \ar@{}[u]|{\rotatebox{90}{$\subset$}}&\ol{W} \ar@{}[u]|{\rotatebox{90}{$\subset$}}&&(\P^{2})^{2} \ar@{}[u]|{\rotatebox{90}{$\simeq$}}&pt. \ar@{}[u]|{\rotatebox{90}{$\in$}} \\
C \ar@{}[u]|{\rotatebox{90}{$\subset$}}&&R. \ar@{}[u]|{\rotatebox{90}{$\subset$}}&&&
}
\end{align*}

\end{prop}


\begin{proof}[Proof of Proposition~\ref{LG-prop-stflip}~(1)]
Let $\wt{P_{ij}}$ and $P_{ij,Y}$ be the proper transforms of $P_{ij}$ on $\wt{\P}$, $Y$ respectively. 
Let $H_{ij}$ be the pull-back of a hyperplane of $P_{ij}=\P^{3}$ on $P_{ij,Y}$. 
Note that $\wt{P_{ij}}=\Bl_{l_{i},l_{j}}P_{ij}$.  
By Proposition~\ref{LG-prop-R}~(2), $\wt{R} \cap \wt{P_{ij}}$ is a smooth surface on $\wt{P_{ij}}$, which implies that $P_{ij,Y} \to \wt{P_{ij}}$ is isomorphic and $G|_{P_{ij,Y}} \sim 4H_{ij}-2D_{Y}|_{P_{ij,Y}}$. 
In particular, we have $-(2D_{Y}+G)|_{P_{ij,Z}}=-4H_{ij} \sim_{P_{ij}} 0$, which implies that the morphism $P_{ij,Z} \to P_{ij}$ is isomorphic: 
\[\xymatrix{
P_{ij,Y} \ar[r]^{\t|_{P_{ij,Y}}}_{\simeq} \ar[d]_{p_{Y}|_{P_{ij,Y}}}& \wt{P_{ij}} \ar[d]^{\s|_{\wt{P_{ij}}}} \\
P_{ij,Z} \ar[r]_{g|_{P_{ij,Z}}}^{\simeq}&P_{ij}. 
}\]
The remaining part is to prove that $P_{ij,Z} \cap P_{ik,Z} = \emp$ for $j \neq k$. 
Note that $P_{ij,Y} \cap P_{ik,Y}=\emp$. 
Set 
$T_{ij}:=P_{ij,Y} \cap D_{i,Y} \text{ and } T_{ik}:=P_{ik,Y} \cap D_{i,Y}$. 
Then $T_{ij}$ (resp. $T_{ik}$) is the exceptional divisor of $P_{ij,Y} \simeq \wt{P_{ij}} \to P_{ij}$ (resp. $P_{ik,Y} \simeq \wt{P_{ik}} \to P_{ik}$) which dominates $l_{i}$. 
In particular, we have $T_{ij} \cap T_{ik}=\emp$. 
If we set $t_{ij}=p_{Y}(T_{ij})$ and $t_{ik}=p_{Y}(T_{ik})$, then $t_{ij}$ (resp. $t_{ik}$) is a line on $P_{ij,Z}$ (resp. $P_{ik,Z}$) and it holds that $T_{ij}=p_{Y}^{-1}(t_{ij})$ and $T_{ik}=p_{Y}^{-1}(t_{ik})$. 
Since $T_{ij} \cap T_{ik}=\emp$, we have $t_{ij} \cap t_{ik}=\emp$. 
For $Z_{i}=g^{-1}(l_{i})$, we have $P_{ij,Z} \cap P_{ik,Z} \subset Z_{i}$ and $P_{ij,Z} \cap P_{ik,Z} \cap Z_{i}=t_{ij} \cap t_{ik}$. 
Thus we conclude that $P_{ij,Z} \cap P_{ik,Z}=\emp$. 
\end{proof}

To prove Proposition~\ref{LG-prop-stflip}~(2) and (3), we need two lemmas. 
Recall the flopping contraction $f|_{W} \colon W \to \ol{W}$ as in Lemma~\ref{LG-lem-W}. 
We also recall that $W_{Z}$ denotes the proper transform of $W$ on $Z$. 
Then we have the following first lemma. 

\begin{lem}\label{LG-lem-Wflop}
$W_{Z}$ is the flop of $W \to \ol{W}$. 
In particular, $W_{Z}$ is smooth by Lemma~\ref{LG-lem-W}.  
\end{lem}
\begin{proof}
Let $W_{Y}$ and $\wt{W}$ be the proper transforms of $W$ on $Y$ and $\wt{\P(V)}$ respectively. 
\[\xymatrix{
W_{Y} \ar[r]^{\t|_{W_{Y}}} \ar[d]_{p_{Y}|_{W_{Y}}} & \wt{W} \ar[d]^{\s|_{\wt{W}}} \\
W_{Z} \ar[r]_{g|_{W_{Z}}}&\ol{W}
}\]
Recall that $\ol{W}$ is defined by $x_{0}x_{2}x_{4}+x_{1}x_{3}x_{5}$ in $\P(V)=\Proj \C[x_{0},x_{1},x_{2},x_{3},x_{4},x_{5}]$ as in (\ref{LG-defeq-olW}). 
In this setting, we may assume that the lines $l_{1},l_{2},l_{3}$ are defined as $\{x_{0}=x_{1}=x_{2}=x_{3}=0\}$, $\{x_{0}=x_{1}=x_{4}=x_{5}=0\}$ and $\{x_{2}=x_{3}=x_{4}=x_{5}=0\}$ respectively. 
Then we can easily check that $\wt{W}$ is normal by computing the defining equation of $\wt{W}$ in $\wt{\P(V)}$.

The restriction $\t|_{W_{Y}} \colon W_{Y} \to \wt{W}$ is the blowing-up along a Weil divisor $\wt{R}$ and hence it is a small contraction since the fiber of $\t$ is at most $1$-dimensional. 
Thus $W_{Y}$ is also normal since $W_{Y}$ is Gorenstein. 
\begin{claim}\label{LG-claim1-Wflop}
$W_{Z}$ does not contain $g^{-1}(l_{i})=Z_{i}$ for every $i$. 
\end{claim}
\begin{proof}
Note that $\wt{W} \in |\s^{\ast}\mc O_{\P(V)}(3)-2\Exc(\s)|$. 
For general point $t \in l_{i}$, $\wt{W}_{t}:=\s^{-1}(t) \cap \wt{W}$ is a quadric surface on $\s^{-1}(t)=\P^{3}$ contains the disjoint union of two lines $\wt{R}_{t}:=\s^{-1}(t) \cap \wt{R}$. 
Hence $\wt{W}_{t} \simeq \P^{1} \times \P^{1}$ and $\wt{W}_{t}=W_{Y,t} \to W_{Z,t}$ is a $\P^{1}$-bundle. 
Hence $W_{Z}$ does not contain $Z_{i}$. 
\end{proof}
\begin{claim}\label{LG-claim2-Wflop}
$g|_{W_{Z}} \colon W_{Z} \to \ol{W}$ is a small contraction. 
\end{claim}
\begin{proof} 
Assume the contrary. 
Over the open subset $\P^{5} \setminus (l_{12} \cup l_{13} \cup l_{23})$, $g$ is the blowing-up along $R \setminus ( l_{12} \cup l_{13} \cup l_{23})$. 
Therefore, if $g|_{W_{Z}}$ contracts a divisor, then this must be contained in $g^{-1}(l_{i})$ for some $i$. 
Since $g^{-1}(l_{i})=Z_{i}$ is a $3$-dimensional, $W_{Z}$ contains $Z_{i}$, which contradicts Claim~\ref{LG-claim1-Wflop}. 
\end{proof}
By Claim~\ref{LG-claim2-Wflop}, $W_{Z}$ is regular in codimension $1$. 
Since $W_{Z}$ is a divisor on a smooth fivefold $Z$, $W_{Z}$ is Gorenstein. 
Hence $W_{Z}$ is normal. 

Now $P=P_{23}+P_{13}+P_{12}$ is a Weil divisor on $\ol{W}$. 
Let $P_{W}$ and $P_{W_{Z}}$ be the proper transforms on $W$ and $W_{Z}$ respectively. 
Then we have $P_{W}=\Exc(f|_{E})|_{W}$ and hence $-P_{W}$ is $f|_{W}$-ample. 
Hence it suffices to show that $P_{Z_{W}}$ is $g|_{W_{Z}}$-ample. 
By Proposition~\ref{LG-prop-R}~(3), we have $P+R \sim 3H|_{\ol{W}}$, where $H$ is a hyperplane section of $\P^{5}$. 
Thus we have $\Exc(g)|_{W_{Z}}+P_{Z_{W}} \sim_{g|_{W_{Z}}} 0$. 
The $g$-ampleness of $-\Exc(g)$ implies that $P_{Z_{W}}$ is $g|_{W_{Z}}$-ample, which completes the proof. 
\end{proof}

The following next lemma is essentially the same as Lemma~\ref{KG2-lem-stflop}. 

\begin{lem}\label{LG-lem-stflip}
\begin{enumerate}
\item It holds that $\mc N_{P_{ij,Z}/W_{Z}} \simeq \mc O_{\P^{3}}(-1)$ and $\mc N_{P_{ij,Z}/Z} \simeq \mc O_{\P^{3}}(-1)^{2}$ for any $i,j$. 
\item The complete linear system $|3H_{Z}-G_{Z}|$ is a pencil with $\Bs|3H_{Z}-G_{Z}|=\bigsqcup_{1 \leq i<j \leq 3}P_{ij,Z}$. 
\item The complete linear system $|4H_{Z}-G_{Z}|$ is base point free. 
\item Let $\psi_{Z} \colon Z \to \ol{Z}$ be the Stein factorization of the morphism given by $|4H_{Z}-G_{Z}|$. 
Then $\psi_{Z}$ is a flipping contraction with $\Exc(\psi)=P_{23,Z} \sqcup P_{13,Z} \sqcup P_{12,Z}$. 
Moreover, the flip $Z^{+}$ of $\psi_{Z} \colon Z \to \ol{Z}$ is smooth. 
\item The induced birational map $W_{Z} \dra W_{Z^{+}}$ is the blowing down contracting $P_{12,Z},P_{13,Z},P_{23,Z}$ to smooth three points. 
\end{enumerate}
\end{lem}
\begin{proof}
(1) Consider the exact sequence 
$0 \to \mc N_{P_{ij,Z}/W_{Z}} \to \mc N_{P_{ij,Z}/Z} \to \mc N_{W_{Z}/Z}|_{P_{ij,Z}} \to 0$. 
Since $-K_{W_{Z}}=3H_{Z}|_{W_{Z}}$ and $-K_{P_{ij,Z}}=4H_{Z}|_{P_{ij,Z}}$, we have $\mc N_{P_{ij,Z}/W_{Z}}=\mc O_{\P^{3}}(-1)$ by the adjunction formula. 
Since $W_{Z} \sim 3H_{Z}-G_{Z}$ and $G_{Z} \cap P_{ij,Z}$ is a quartic surface (see Proposition~\ref{LG-prop-wtR}~(1)), 
we have $\mc N_{W_{Z}/Z}|_{P_{ij,Z}}=\mc O_{\P^{3}}(-1)$. 
Therefore, we have $\mc N_{P_{ij,Z}/Z} \simeq \mc O_{\P^{3}}(-1)^{2}$. 

(2) By Proposition~\ref{LG-prop-R}~(2) and (3), we obtain $h^{0}(\P(V),\mc I_{R} \otimes \mc O_{\P^{5}}(3))=2$. 
Let $\ol{W}$ and $\ol{W'}$ be the two members as in Proposition~\ref{LG-prop-R}~(3). 
Let $W_{Z}$ and $W'_{Z}$ denote the proper transforms of $\ol{W}$ and $\ol{W'}$ on $Z$ respectively. 
Then $W_{Z} \cap W'_{Z}=\bigsqcup_{i,j}P_{ij,Z}$ by Proposition~\ref{LG-prop-R}~(3), which implies the last assertion of (2). 

(3) From (2), we can deduce that $\Bs|4H_{Z}-G_{Z}| \subset \bigsqcup_{i,j}\Bs \left| (4H_{Z}-G_{Z})|_{P_{ij,Z}} \right|$ by the same arguments as in the proof of Lemma~\ref{KG2-lem-stflop}~(4). 
For each $i,j$, we have $(4H_{Z}-G_{Z})|_{P_{ij,Z}} \sim 0$ and hence $\Bs |4H_{Z}-G_{Z}|=\emp$. 

(4) By the similar argument as in the proof of (2), we have $\Exc(\psi)=P_{23,Z} \sqcup P_{13,Z} \sqcup P_{12,Z}$. 
Since $-K_{Z}=6H_{Z}-G_{Z}$, $\psi_{Z}$ is a flipping contraction. 
The smoothness of $Z^{+}$ follows from (1). 
(5) follows from (1) and (4). 
We complete the proof. 
\end{proof}

\begin{proof}[Proof of Proposition~\ref{LG-prop-stflip}~(2), (3), and (4)]
Let $H_{Z^{+}}$, $G_{Z^{+}}$ and $W_{Z^{+}}$ be the proper transforms of $H_{Z}$, $G_{Z}$, and $W_{Z}$ on $Z^{+}$ respectively. 
By Lemma~\ref{LG-lem-stflip}, $Z \dra Z^{+}$ is the standard flip of $P_{1,Z}$ and $P_{2,Z}$. 
Then from the same arguments as in the proof of Proposition~\ref{KG2-prop-stflop}, it follows that $\Bs |3H_{Z^{+}}-G_{Z^{+}}|=\Bs |W_{Z^{+}}|=\emp$. 
Let $\vp_{Z^{+}} \colon Z^{+} \to \P^{1}$ be the morphism induced by $|W_{Z^{+}}|$. 
Since $W_{Z^{+}}$ is a fiber of $\vp_{Z^{+}}$, every fiber of $\vp_{Z^{+}}$ is connected. 
Since $\rho(Z^{+})=2$, $\vp_{Z^{+}}$ is an extremal contraction. 
Since $-K_{W_{Z}}=3H_{Z}|_{W_{Z}}$ and $(H_{Z}|_{W_{Z}})^{4}=3$, we have $(H^{+}|_{W_{Z^{+}}})^{4}=6$ and $-K_{W_{Z^{+}}}=3H^{+}|_{W_{Z^{+}}}$. 
By the classification of del Pezzo manifold \cite{Fujita80}, we conclude that $W_{Z^{+}} \simeq (\P^{2})^{2}$. 
Then every smooth fiber of $\vp_{Z^{+}}$ is also a del Pezzo fourfold of degree $6$ and hence isomorphic to $(\P^{2})^{2}$. 
\end{proof}


\subsection{Construction of an example of No.14}\label{LG-subsec-conclu}
Finally, we construct a weak Fano threefold with a sextic del Pezzo fibration of No.14. 

Let $\wh{\mathfrak{H}} \in |H_{3}|_{M}|$ be a general element. 
By using the diagram (\ref{LG-dia-M}), we set 
\[\mathfrak{H}:=\nu(\wh{\mathfrak{H}}) \text{ and } V^{5}_{9}:=\mu(\mathfrak{H}) \subset \LG_{\a}(3,V).\]
We may assume that $\mathfrak{H}$, $V^{5}_{9}$, and $S:=V^{5}_{9} \cap \LG_{\a,\b}(3,V)$ are smooth. 

Let $\wh{A} \in |H_{3}|_{\wh{\mathfrak{H}}}|$ be a general element, 
\[A:=\nu(\wh{A}),\ V^{4}_{9}:=\mu(A),\ \wt{A}:=\pi(\wh{A}), \text{ and } \ol{A}:=\s(\wt{A}).\]
We may assume that $A$, $V^{4}_{9}$, and $C:=\LG_{\a,\b}(3,V) \cap V^{4}_{9}$ are smooth. 

Then $C \subset S \subset \LG_{\a,\b}(3,V)=(\P^{1})^{3}$ satisfies Setting~\ref{LG-sett-stflip}. 
Combining the diagram (\ref{LG-dia-M}) and the diagram (\ref{LG-dia-Z}), we obtain the following diagram. 

\[\xymatrix{
\P_{\wt{\P(V)}}(\mc E|_{\wt{\P(V)}}) \ar@{}[r]|{\hspace{20pt}=}&M \ar[rd]^{\pi} \ar[d]_{\nu}&&\ar[ld]_{\t}Y \ar[d]_{h}&\ar@{}[l]|{=\hspace{20pt}} \Bl_{\wt{R}}\wt{\P(V)}& \\
&\wt{\LG_{\a}(3,V)}\ar[ld]_{\mu} \ar[rd]^{f}&\wt{\P(V)}\ar[d]_{\s}&Z\ar[ld]_{g} \ar@{-->}[r]&Z^{+}\ar[rd]& \\
\LG_{\a}(3,V)&&\P(V)&&&\P^{1}. 
}\]

Let us consider the proper transforms of $\wh{\mathfrak{H}}$. 
Since $\wh{\mathfrak{H}} \in |\mc O_{\P(\mc E|_{\wt{\P(V)}})}(1)|$ is general, 
we may assume that $\pi|_{\wh{\mathfrak{H}}} \colon \wh{\mathfrak{H}} \to \wt{\P(V)}$ is the blowing-up along a smooth threefold $\Theta$, where $\Theta$ is the zero scheme of the corresponding global section $t \in H^{0}(\wt{\P(V)},\mc E|_{\wt{\P(V)}})$. 
Then we have the following diagram: 

\[\xymatrix{
&\wh{\mathfrak{H}}=\Bl_{\Theta}\wt{\P(V)} \ar[rd]^{\pi|_{\wh{\mathfrak{H}}}} \ar[d]_{\nu|_{\wh{\mathfrak{H}}}}&&\ar[ld]_{\t} Y \ar[d]_{h}&& \\
&\mathfrak{H}\ar[ld]_{\mu|_{\mathfrak{H}}} \ar[rd]_{f|_{\mathfrak{H}}}&\wt{\P(V)}\ar[d]^{\s}&Z\ar[ld]^{g} \ar@{-->}[r]&Z^{+}\ar[rd]& \\
V^{5}_{9}&&\P(V)&&&\P^{1}. 
}\]

\begin{lem}\label{LG-lem-key}
It holds that 
\begin{align*}
H_{3}&\sim 4H_{1}-2\Exc(\nu|_{\mathfrak{H}})-\Exc(\pi|_{\wh{\mathfrak{H}}}) \text{ on } \wh{\mathfrak{H}}, \\
3H_{1} &\sim \Exc(\pi|_{\wh{\mathfrak{H}}})+\nu^{-1}_{\ast}\Exc(\mu|_{\mathfrak{H}})+2\Exc(\nu|_{\wh{\mathfrak{H}}}) \text{ on } \wh{\mathfrak{H}}, \text{ and } \\
3H_{1} &\sim \Exc(f|_{\mathfrak{H}})+\Exc(\mu|_{\mathfrak{H}}) \text{ on } \mathfrak{H}. 
\end{align*}
\end{lem}
\begin{proof}
Note that $\det \mc E|_{\wt{\P(V)}} \simeq \mc O(4H_{1}-2\Exc(\s))$ by (\ref{LG-eq-chernG1}). 
Since the embedding $\wh{\mathfrak{H}} \hra \P_{\wt{\P(V)}}(\mc E|_{\wt{\P(V)}})$ is corresponding to the surjection 
$\mc E|_{\wt{\P(V)}} \to \mc I_{\Theta}(4H_{1}-2\Exc(\s))$, 
we obtain 
\[H_{3} \sim -\Exc(\pi|_{\wh{\mathfrak{H}}})+4H_{1}-2(\pi|_{\wh{\mathfrak{H}}})^{\ast}\Exc(\s).\] 
It follows from the construction that 
\begin{align}\label{LG-eq-relexp}
(\pi|_{\wh{\mathfrak{H}}})^{\ast}\Exc(\s)=\Exc(\nu|_{\wh{\mathfrak{H}}}). \end{align}
Thus the first equality holds.  
Note that $H_{3}=H_{1}+(\nu|_{\wh{\mathfrak{H}}})^{\ast}\Exc(\mu|_{\mathfrak{H}})$ by Lemma~\ref{LG-lem-relonwtLG}. Hence we have 
\[\Exc(\pi|_{\wh{\mathfrak{H}}}) \sim 3H_{1}-(\nu|_{\wh{\mathfrak{H}}})^{\ast}\Exc(\mu|_{\mathfrak{H}})-2\Exc(\nu|_{\wh{\mathfrak{H}}}).\]
Since $\nu^{\ast}\Exc(\mu)=\nu^{-1}_{\ast}\Exc(\mu)$ holds from the construction, 
we have the second equality. 
By taking the push-forward and noting that $(\nu|_{\wh{\mathfrak{H}}})_{\ast}\Exc(\pi|_{\wh{\mathfrak{H}}})=\Exc(f|_{\mathfrak{H}})$, 
we obtain the last equality. 
\end{proof}

Let us consider the proper transform of $\wh{A}$.  
By the first equality of Lemma~\ref{LG-lem-key}, $\ol{A}$ is a quartic hypersurface of $\P^{5}$ containing $R$ and $\ol{\Theta}:=\s(\Theta)$. 
Moreover, $\ol{A}$ is smooth at the generic point of each of $R$ and $\ol{\Theta}$. 

Let $\wh{A}'$ and $A'$ be the proper transforms on $Y$ and $Z$. 
Noting that $\ol{A}$ is smooth at the generic point of $R$, we conclude that $A'$ is a member of $|g^{\ast}4H_{1}-\Exc(g)|$ as a divisor on $Z$. 
By Lemma~\ref{LG-lem-stflip}~(3), we may assume that $A'$ and $\wh{A'}$ are smooth and $A'$ is away from the flipping locus of $Z \dra Z^{+}$. 
As a divisor on $Z^{+}$, $A'$ is a member of $|H_{Z^{+}}+W_{Z^{+}}|$: 
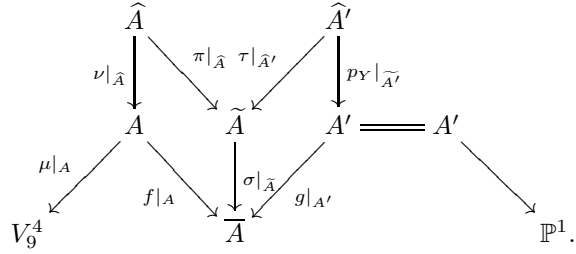
\[\xymatrix{
&\wh{A} \ar[rd]^{\pi|_{\wh{A}}} \ar[d]_{\nu|_{\wh{A}}}&&\ar[ld]_{\t|_{\wh{A}'}} \wh{A}' \ar[d]^{p_{Y}|_{\wt{A'}}}&& \\
&A\ar[ld]_{\mu|_{A}} \ar[rd]_{f|_{A}}&\wt{A}\ar[d]^{\s|_{\wt{A}}}&A'\ar[ld]^{g|_{A'}} \ar@{=}[r]&A'\ar[rd]& \\
V^{4}_{9}&&\ol{A}&&&\P^{1}.
}\]
\begin{lem}\label{LG-lem-A}
\begin{enumerate}
\item The restriction $\pi|_{\wh{A}} \colon \wh{A} \to \wt{A}$ is a family of Atiyah's flopping contractions along a smooth curve $B$, which is numerically equivalent to $c_{2}(\mc E|_{\wt{\P(V)}})^{2}$. 
Moreover, $\wh{A'}$ is the flop of $\pi|_{\wh{A}}$. 
\item The restriction $f|_{A} \colon A \to \ol{A}$ is the flopping contraction and $A'$ is the flop. 
\end{enumerate}
\end{lem}
\begin{proof}
(1) 
Note that $\wh{A}$ is the complete intersection of general two members of $|H_{3}|=|\mc O_{\P(\mc E|_{\wt{\P(V)}})}(1)|$. 
Then by the same argument as in the proof of Proposition~\ref{KG2-prop-conclu}~(1), 
we can see that $\pi|_{\wh{A}}$ is a family of Atiyah's flopping contractions and $\pi(\Exc(\pi|_{\wh{A}}))=B$ is a smooth curve, which is a complete intersection of the zero schemes of two general section of $\mc E|_{\wt{\P(V)}}$. 

Let us prove that $\wh{A'}$ is the flop of $\wh{A}$. 
Since $\pi|_{\wh{A}}$ is isomorphic over the generic point of $\wt{R} \subset \wt{A}$, we can see that $\wh{A'} \to \wt{A}$ is also a flopping contraction. 
By the second equality of Lemma~\ref{LG-lem-key} and (\ref{LG-eq-relexp}), we obtain 
\[\Exc(\pi|_{\wh{\mathfrak{H}}})|_{\wh{A}}+\nu^{-1}_{\ast}\Exc(\mu|_{\mathfrak{H}})|_{\wh{A}} \sim_{\wt{A}} 3H_{1}|_{\wh{A}} .\]
Now $\Exc(\pi|_{\wh{\mathfrak{H}}})|_{\wh{A}}$ (resp.$\nu^{-1}_{\ast}\Exc(\mu|_{\mathfrak{H}})|_{\wh{A}}$) is linearly equivalent to 
the proper transform $\Theta_{\wh{A}}$ (resp. $\wt{R}_{\wh{A}}$) of $\Theta \subset \wt{A}$ (resp. $\wt{R} \subset \wt{A}$). 
By taking the push-forward of the above equality, we have 
\begin{align}\label{LG-eq-keywtA}
\mc O_{\wt{A}}(\Theta+\wt{R}) \sim (\s|_{\wt{A}})^{\ast}\mc O_{\ol{A}}(3).
\end{align}
Let $\wt{R}_{\wh{A'}}$ be the proper transform of $\wt{R} \subset \wt{A}$ on $\wh{A'}$. 
Since $\wh{A}=\Bl_{\Theta}\wt{A}$ and $\wh{A'}=\Bl_{\wt{R}}\wt{A}$,  $-\Theta_{\wh{A}}$ is ample over $\wt{A}$ and $-\wt{R}_{\wh{A'}}$ is ample over $\wt{A}$. 
By (\ref{LG-eq-keywtA}), $\wt{A'}$ is the $\Theta_{\wh{A}}$-flop of $\wt{A}$. 

(2) Since $\nu|_{\wh{A}}$ and $\s|_{\wt{A}}$ contract same divisors, 
$f|_{A} \colon A \to \ol{A}$ is also a small contraction. 
Since $-K_{\wt{A}} \sim 2H_{1}|_{V}$ by the adjunction, $f|_{A} \colon A \to \ol{A}$ is flopping. 
It follows from the construction that $p_{Y}|_{\wh{A}'}$ and $\s|_{\wt{A}}$ also contract same divisors. 
Thus $g|_{A'} \colon A' \to A$ is small. 


Let $\ol{\Theta}=\s(\Theta) \subset \ol{A}$ and $\Theta_{A} \subset A$ be the proper transform of $\ol{\Theta}$ on $A$. 
Then $\Theta_{A}=\Exc(f|_{\wh{V}})|_{A}$ and hence $-\Theta_{A}$ is $f|_{A}$-ample. 
Let $\wt{R}_{A'} \subset A'$ be the proper transform of $R$ on $A'$. 
Then $-\wt{R}_{A'}$ is $g|_{A'}$-ample. 
By (\ref{LG-eq-keywtA}), we obtain 
\begin{align*}
\mc O_{\ol{A}}(\ol{\Theta}+R) \sim \mc O_{\ol{A}}(3). 
\end{align*}
The same argument shows that $A'$ is the flop of $A$. 
\end{proof}

Let $\wh{X} \in |H_{1}|_{\wh{A}}|$ be a general element and 
\[X:=\nu(\wh{X}), V^{3}_{9}:=\mu(X), \wt{X}:=\pi(\wh{X}) \text{ and } \ol{X}=\s(\wt{X}).\]
Let $X' \subset A'$ be the proper transform of $X$ on $A'$. 
We may assume that $X$ and $\wh{X}$ are smooth. 
Since $A \to \ol{A} \gets A'$ is the flop and both of $X$ and $X'$ are the pullback of a general hyperplane section $\ol{X}$ of $\ol{A} \subset \P^{5}$, $X \to \ol{X} \gets X'$ is also a flop. 
In particular, $X'$ is also smooth by \cite[Theorem~6.15]{KM98}. 
Since $\vp_{Z^{+}}|_{X'} \colon X' \to \P^{1}$ is a relative hyperplane section of $\vp_{Z^{+}}|_{A'} \colon A' \to \P^{1}$, $\vp_{Z^{+}}|_{X'}$ is a sextic del Pezzo fibration: 
\[\xymatrix{
&\wh{X} \ar[rd]^{\pi|_{\wh{X}}} \ar[d]_{\nu|_{\wh{X}}}&&\ar[ld]_{\t|_{\wh{X}'}} \wh{X}' \ar[d]^{p_{Y}|_{\wh{X'}}}&& \\
&X\ar[ld]_{\mu|_{X}} \ar[rd]_{f|_{X}}&\wt{X}\ar[d]_{\s|_{\wt{X}}}&X'\ar[ld]^{g|_{X'}} \ar@{=}[r]&X'\ar[rd]^{\vp|_{X'}}& \\
V^{3}_{9}&&\ol{X}&&&\P^{1}
}\]
The existence of an example of No.14 is established by the following proposition. 
\begin{prop}\label{LG-prop-conclu}
\begin{enumerate}
\item For a general $\wh{X} \in |H_{1}|_{\wh{A}}|$, $V^{3}_{9}$ is smooth and $\mu|_{X} \colon X \to V^{3}_{9}$ is the blowing-up along $C$. 
\item $X$ is a weak Fano threefold of Picard rank $2$. 
\item The flop $X \dra X'$ is an Atiyah flop. 
The flopped curves consists of the disjoint union of $24$ sections and $3$ bisections of $\vp_{X'}$. 
\end{enumerate}
\end{prop}
\begin{proof}
(1) We may assume that $\Exc(\mu|_{A}) \cap X$ is a smooth surface. 
Since $\Exc(\mu|_{A}) \cap X$ is a tautological divisor of the $\P^{2}$-bundle $\Exc(\mu|_{A}) \to C$, $\Exc(\mu|_{A}) \cap X$ is a $\P^{1}$-bundle. 
Hence we have the assertion (1). 

(2) By the Lefschetz hyperplane section theorem, we have $\rho(V^{3}_{9})=1$ and hence $\rho(X)=2$. $-K_{X}=H_{1}|_{X}$ holds from the adjunction. 
Thus we obtain the assertion (2). 

(3) First we consider the flop $\chi_{\wh{X}} \colon \wh{X} \dra \wh{X'}$. 
By Lemma~\ref{LG-lem-A}~(1), $\chi_{\wh{X}}$ is an Atiyah flop and 
the number of flopping curves is equal to $n:=c_{2}(\mc E|_{\wt{\P(V)}})^{2}.H_{1}=c_{2}(\mc E|_{\wt{\P(V)}})^{2}.\mu^{\ast}\mc O_{\P(V)}(1)$ on $\wt{\P(V)}$. 
By using (\ref{LG-eq-chernG1}) and (\ref{LG-eq-intwtPV}), we have 
\[n=4H_{1}(3H_{1}^{2}-3H_{1}\Exc(\s)+\Exc(\s)^{2})^{2}=36H_{1}^{5}+4H_{1}\Exc(\s)^{4}=24.\] 
Next we consider the flop $\chi_{X} \colon X \dra X'$. 
Note that $\nu|_{\wh{X}} \colon \wh{X} \to X$ is the blowing-up along $\bigsqcup_{i=1}^{3} Q_{i} \cap \wh{X}$. 
Since $\G_{i}:=Q_{i} \cap \wh{X}$ is the complete intersection of two members of $|H_{3}|_{Q_{i}}|$ and one member of $|H_{1}|_{Q_{i}}|$, $\G_{i}$ is a smooth conic on $\Q^{3}$. 
By Lemma~\ref{LG-lem-Qi}~(3), we obtain $\mc N_{\G_{i}/X} \simeq \mc O(-1)^{2}$. 
Thus the flopping locus of $\chi_{X}$ is the disjoint union of these three smooth rational curves $\bigsqcup_{i=1}^{3}\G_{i}$ and the 24 smooth rational curves $\bigsqcup_{j=1}l_{j}$. 

By taking $X$ generically, we may assume that $\G_{i}$ and $l_{j}$ meet $E_{X}:=\Exc(\mu|_{X})$ transversally. 
Since $E_{X}.\G_{i}=2$ and $E_{X}.l_{j}=1$ for each $i,j$, we obtain ${\chi_{X}}_{\ast}E_{X}.\G_{i}'=-2$ and ${\chi_{X}}_{\ast}E_{X}.l_{j}'=-1$, where $\G_{i}'$ and $l_{j}'$ denote the flopped curves. 
Since $3(-K_{X'})-{\chi_{X}}_{\ast}E_{X}$ is linearly equivalent to a $\vp_{X'}$-fiber, 
we obtain the assertion (3). 
\end{proof}

We complete the proof of Theorem~\ref{mainthm-wf} for No.14. 
\hfill$\square$


\section{Existence of an example of No.16}\label{24-section}

In this section, we will prove the existence of weak Fano sextic del Pezzo fibration of No.16. 
It is enough to prove the existence of such threefolds with $(-K_{X})^{3}=2$ and $h^{1,2}(X)=4$. 
We basically use the same method as in the author's paper \cite{Fukuoka17}. 

Let us consider the following lattice: 
\[\Lambda=\left(\begin{array}{cccc} 
&L&F&T\\
L&2&4&5\\
F&4&0&3\\
T&5&3&2\\
\end{array}\right).\]
It is easy to see that $\Lambda$ is an even lattice with index $(1,2)$. 
By \cite[Corollary~2.9]{Morr84}, there is a K3 surface $S$ whose N\'eron-Severi lattice is isometry to $\Lambda$. 
Moreover, we may assume that $L$ is nef (cf. \cite[Proposition~3.10]{BHPV}). 

By straightforward computation, we have the following lemma. 
\begin{lem}\label{24-lem-(-2)}
For an element $D \in \Pic(S)$, if $D^{2}=-2$, then $L.D \geq 5$. 
\end{lem}
\begin{proof}
For each $D \in \Pic(S)$, there exist integers $x,y,z$ such that $D \sim xL+yF+zT$. 
Then we have $D^{2}=2x^{2}+8xy+10xz+6yz+2z^{2}$ and $L.D=2x+4y+5z$.  
If we assume that $D^{2}=-2$ and $n:=L.D \leq 4$, then we have 
\[x=\frac{2n-3z\pm \sqrt{16+4n^{2}-35z^{2}}}{4} \text{ and } y=\frac{-7z \mp \sqrt{16+4n^{2}-35z^{2}}}{8}.\]
If $n \leq 4$, then $16+4n^{2}-35z^{2} \leq 80-35z^{2}$ and hence $z^{2} \in \{0,1\}$. 
Since $16+4n^{2}-35z^{2}$ is a square number, we have $z=0$ and $n=0$. 
Then we have $y=\mp \frac{1}{2}$, which is a contradiction. 
Hence we obtain the assertion. 
\end{proof}

Let us study positivity of some divisors on $S$ by using \cite[Theorem~2.2]{Fukuoka17}, which is just a consequence of Saint-Donat's results \cite{Saint-Donat74} and Reider's method \cite{Rei88}. 

\begin{lem}\label{24-lem-divs}
\begin{enumerate}
\item $L$ is ample and $|L|$ is base point free. 
\item The complete linear systems $|T|$, $|F|$ and $|2L+F-T|$ are base point free. 
In particular, we may assume $T$ and $F$ are irreducible and smooth. 
\item $|L+F|$ is very ample. 
\end{enumerate}
\end{lem}
\begin{proof}

(1) 
By Lemma~\ref{24-lem-(-2)}, there are no divisors $D$ satisfying $L.D=0$ and $D^{2}=-2$. 
Then it follows from \cite[Theorem~2.2]{Fukuoka17} that $3L$ is very ample.  Hence $L$ is ample. 


(2) By \cite[Theorem~2.2]{Fukuoka17}, it suffices to show that $T$ (resp. $F$, $2L+F-T$) is movable. 
Set $D=T$, $F$ or $2L+F-T$. 
Since $D^{2} \geq 0$ and $5 \geq L.D>0$ for each $D$, we have $h^{0}(\mc O_{S}(D)) \geq 2$. 
Let $N$ be the fixed part of $|T|$ and $M$ the movable part. 
Then $N$ is a union of $(-2)$-curves on $S$ by \cite[Lemma~2.3]{Fukuoka17}. 
Since $M \neq 0$ and $L.M>0$ by (1), we have a $(-2)$-curve $C$ with $L.C \leq 4$, which contradicts to Lemma~\ref{24-lem-(-2)}. 

(3) Note that $L+F$ is ample by (1) and (2) and $(L+F)^{2}=10$. 
By \cite[Theorem~2.2]{Fukuoka17}, it suffices to check that there is no $D \in \Pic(S)$ such that $D^{2}=0$ and $(L+F).D \in \{1,2\}$. 
If such a $D$ exists, then $D$ is effective since $L+F$ is ample. 
Since $|L|$ and $|F|$ are base point free, we have $L.D \neq 1$ and $F.D \neq 1$ by \cite[Theorem~2.2]{Fukuoka17}. 
Thus we have $L.D = 2$ and $F.D=0$. 
Take the integers $x,y,z$ such that $D=xL+yF+zT$. 
By solving $D^{2}=0$, $L.D=2$, and $F.D=0$, we have $x=z=0$ and $y=\frac{1}{2}$, which is impossible. We complete the proof. 
\end{proof}
Let $\pi \colon S \to \P^{2}$ be the morphism given by the complete linear system $|L|$. 
Since $L$ is ample by Lemma~\ref{24-lem-divs}~(1), $\pi$ is a finite flat morphism of degree $2$. 
Set 
\begin{align*}
\mc E:=\pi_{\ast}\mc O_{S}(L+F).
\end{align*}
Then $\mc E$ is a locally free sheaf of rank $2$. 
Using the Grothendieck-Riemann-Roch theorem for $\pi$, we obtain 
\begin{align}\label{24-eq-chern}
c_{1}(\mc E)=3 \text{ and } c_{2}(\mc E)=4.
\end{align}

Let $p \colon \P_{\P^{2}}(\mc E) \to \P^{2}$ be the projection and $\xi$ the tautological divisor.

\begin{lem}[{\cite[Theorem~1]{Ohno16v7}}]\label{24-lem-011}
The vector bundle $\mc E$ on $\P^{2}$ fits in the following exact sequence: 
\[0 \to \mc O(-1) \to \mc O(1)^{\oplus 2} \oplus \mc O \to \mc E \to 0.\]
\end{lem}
\begin{proof}
By the surjection $\pi^{\ast}\mc E=\pi^{\ast}\pi_{\ast}\mc O_{S}(L+F) \to \mc O_{S}(L+F)$, 
$S$ is embedded in $\P_{\P^{2}}(\mc E)$ as a divisor. 
Since $p|_{S}=\pi$ is a double covering and $S$ is a K3 surface, we have $S \sim 2\xi$. 
Since $\xi|_{S}=L+F$ is ample by Lemma~\ref{24-lem-divs}, $\xi$ is nef. 
Then (\ref{24-eq-chern}) and Ohno's classification \cite[Theorem~1]{Ohno16v7} implies the assertion. 
\end{proof}

Let $f \colon \F:=\P_{\P^{2}}(\mc O \oplus \mc O(1)^{2}) \to \P^{2}$ be the projection and $\xi_{\F}$ the tautological divisor. 
Let $l$ be a line of $\P^{2}$. 
Then by Lemma~\ref{24-lem-011}, $\P_{S}(\mc E)$ is embedded into $\F$ as a member of $|\xi_{\F}-f^{\ast}l|$ on $\F$. 
Hence $S$ is a complete intersection of a member of $|\xi_{\F}+f^{\ast}l|$ and $|2\xi_{\F}|$ in $\F$. 
Under this embedding, the restriction $\xi_{\F}|_{S}$ is linearly equivalent to $L+F$. 
The complete linear system $|\xi_{\F}|$ gives a morphism $g \colon \F \to V$, where $V \subset \P^{6}$ is the cone of $\P^{1} \times \P^{2} \subset \P^{5}$. 
The exceptional locus $P_{0}:=\Exc(g)$ is the section of $f$ corresponding to $\mc O \oplus \mc O_{\P^{2}}(1)^{2} \to \mc O$. 

\begin{lem}\label{24-lem-Q}
\begin{enumerate}
\item The morphism $g|_{S} \colon S \to g(S)$ is isomorphic. 
\item There is a smooth member $Q \in |2\xi_{\F}|$ which contains $S$ and is away from $\Exc(g)$. 
\end{enumerate}
\end{lem}
\begin{proof}
(1) 
Since $|L+F|$ is very ample by Lemma~\ref{24-lem-divs}, it suffices to check that $H^{0}(\F,\mc O(\xi_{\F})) \to H^{0}(S,\mc O_{S}(L+F))$ is surjective, 
which is equivalent to say the vanishing $H^{1}(\F,\mc I_{S/\F} \otimes \mc O(\xi_{\F}))=0$, which follows from the exact sequence
\begin{align}\label{24-ex-CI}
0 \to \mc O(-3\xi_{\F}-f^{\ast}l) \to \mc O(-\xi_{\F}-f^{\ast}l) \oplus \mc O(-2\xi_{\F}) \to \mc I_{S/\F} \to 0.
\end{align} 

(2) 
We set $\Lambda:=|\mc I_{S/\F}(2\xi_{\F})|$. 
It follows from the exact sequence (\ref{24-ex-CI}) that $h^{0}(\F,\mc I_{S/\F}(2\xi_{\F}))=3$ and hence $\dim \Lambda=2$. 
From our construction, $S$ is contained in $\E:=\P_{\P^{2}}(\mc E) \in |\xi_{\F}+f^{\ast}l|$. 
Then we have $|\xi_{\F}-f^{\ast}l|+ \E \subset \Lambda$. 
Note that $\dim |\xi_{\F}-f^{\ast}l|=1$ and $\Bs|\xi_{\F}-f^{\ast}l|=\Exc(g)$. 
Then a general member $Q \in \Lambda$ does not contain $\E$ and we have $\Bs \Lambda \subset \E+\Exc(g)$. 
Now $Q \cap \E$ is purely dimension $2$ and contains $S$. 
Hence we have $Q \cap \E=S$ by (\ref{24-ex-CI}). 
Hence $Q$ is smooth along $S$. 

Let us show that $\Bs\Lambda=S$. 
Let $Q \in \Lambda$ be a member such that $Q \cap \E=S$. 
By (1), $S$ is away from $\Exc(g)$. 
Hence $Q$ does not contain $\Exc(g)$, which implies 
$Q$ is away from $\Exc(g)$ 
since $Q$ is a member of $|2\xi_{\F}|$. 
Thus we have $\Bs \Lambda \subset \E$. 
Since $Q \cap \E=S$, we have $\Bs \Lambda =S$. 

Since $\dim \Lambda=2$ and $\Bs\Lambda=S$, a general member $Q$ of $\Lambda$ is integral and smooth away from $S$. 
Since $Q$ is smooth along $S$, $Q$ is a smooth threefold. 
We complete the proof. 
\end{proof}

\begin{proof}[Proof of the existence of an example of No.16]
Let $Q \in |\mc O_{\F}(2\xi_{\F}) \otimes \mc I_{S}|$ be a general smooth member as in Lemma~\ref{24-lem-Q}. 
Under the flip $\F \dra \F'=\P_{\P^{1}}(\mc O_{\P^{1}} \oplus \mc O_{\P^{1}}(1)^{\oplus 3})$, the proper transform of $Q$ on $\F'$ is isomorphic to $Q$. 
Let $q \colon Q \to \P^{1}$ be the induced morphism. 
Then $q$ is a quadric fibration. 
The threefold $Q$ is actually a Fano threefold of type \cite[No.24 in Table~2]{MM81}. 
Now $Q$ contains $S$ as an anti-canonical member. 
When $F_{Q}$ denotes a fiber of $q \colon Q \to \P^{1}$ and $L_{Q}$ the pull-back of a line on $\P^{2}$ under $f|_{Q} \colon Q \to \P^{2}$, 
it follows from the construction that $-K_{Q}=2L_{Q}+F_{Q}$, $L_{Q}|_{S}=L$, and $F_{Q}|_{S}=F$. 
The curve $T \subset S$ is a trisection of this quadric fibration $q \colon Q \to \P^{1}$. 
By Lemma~\ref{24-lem-divs}~(2), $-K_{\Bl_{T}Q}$ is nef (but not big). 
Using \cite[Proposition~3.5]{Fukuoka17}, 
we can conclude that there exist an isomorphism or a flop $\Bl_{T}Q \dra \Bl_{C_{0}}X$ over $\P^{1}$, where $\vp \colon X \to \P^{1}$ is a sextic del Pezzo fibration with a section $C_{0}$ such that $-K_{X}.C_{0}=0$ and $(-K_{X})^{3}=2$. 
Since $-K_{X}.C_{0}=0$ and $-K_{\Bl_{C_{0}}X}$ is nef, so is $-K_{X}$. 
Thus $X$ is a weak Fano sextic del Pezzo fibration. 
Since $(-K_{Q})^{3}=24$, we have $h^{1,2}(Q)=2$ by \cite[Lemma~3.1]{Fukuoka18}. 
Since $T$ is a smooth curve in a K3 surface $S$ such that $(T^{2})_{S}=2$, we have $g(T)=2$. 
Then it follows from the construction that $h^{1,2}(X)=4$. 
We complete the proof. 
\end{proof}

\end{document}